\documentclass[reqno,11pt]{amsart}

\usepackage[framemethod=tikz]{mdframed}
\usepackage{lipsum}

\usepackage[latin1]{inputenc}
\usepackage[english]{babel}

\usepackage{amscd}

\usepackage{amsmath}
\usepackage{amsthm}
\usepackage{amssymb}

\usepackage[notref, notcite, final]{showkeys}
\usepackage{color}
\usepackage{graphicx}
\usepackage[hidelinks]{hyperref}
\usepackage{cleveref}
\usepackage{fullpage}
\usepackage{comment}
\usepackage[shortlabels]{enumitem}

\usepackage{dsfont}

\newcommand{\pp}{\partial}

\newcommand{\rr}{\mathbb{R}}

\renewcommand{\Re}{\mathrm{Re}}

\newcommand{\dom}{\operatorname{dom}}

\usepackage[autostyle]{csquotes}

\newtheorem{theorem}{Theorem}[section]
\newtheorem{lemma}[theorem]{Lemma}
\newtheorem{proposition}[theorem]{Proposition}

\theoremstyle{definition}
\newtheorem{definition}[theorem]{Definition}

\theoremstyle{remark}
\newtheorem{remark}[theorem]{Remark}
\crefname{enumi}{}{}

\title{\bf The $\mathcal{F}$-resolvent equation and Riesz projectors \\ for the $\mathcal{F}$-functional calculus}

\author[F. Colombo]{Fabrizio Colombo}
\address{(FC)
Politecnico di Milano\\Dipartimento di Matematica\\Via E. Bonardi, 9\\20133
Milano, Italy}
\email{fabrizio.colombo@polimi.it}

\author[A. De Martino]{Antonino De Martino}
\address{(ADM)
 Politecnico di Milano\\Dipartimento di Matematica\\Via E. Bonardi, 9\\20133
Milano, Italy
} \email{antoninodemartino@polimi.it}

\author[I. Sabadini]{Irene Sabadini}
\address{(SP)
Politecnico di Milano\\Dipartimento di Matematica\\Via E. Bonardi, 9\\20133
Milano, Italy
} \email{irene.sabadini@polimi.it}

\begin{document}
\maketitle

\begin{abstract}
The Fueter-Sce-Qian mapping theorem
is a two steps procedure to extend holomorphic functions of one complex variable to quaternionic or Clifford algebra-valued functions in the kernel of a suitable generalized Cauchy-Riemann operator.
Using the Cauchy formula of slice monogenic functions it is possible to give the
 Fueter-Sce-Qian extension theorem an integral form and to define the $\mathcal{F}$-functional calculus for $n$-tuples of commuting operators. This functional calculus is defined on the $S$-spectrum
 but it generates a monogenic functional calculus in the spirit of McIntosh and collaborators.
One of the main goals of this paper is to show that the $\mathcal{F}$-functional calculus generates the
 Riesz projectors. The existence of such projectors is obtained via the $\mathcal{F}$-resolvent equation that we have generalized to the Clifford algebra setting.
 This equation was known in the quaternionic setting, but the Clifford algebras setting turned out to be much more complicated.

\end{abstract}

\noindent AMS Classification.


\noindent {\em }
\date{today}

\section{\bf Introduction}
\setcounter{equation}{0}

Holomorphic functions of one complex variable
$f:\Omega \subseteq \mathbb{C} \to \mathbb{C}$  can be extended to quaternionic-valued functions
or, more in general, to Clifford algebra-valued functions using the Fueter-Sce-Qian extension theorem.
This extension theorem is also called the
Fueter-Sce-Qian construction and gives two different notions of hyperholomorphic functions, see  \cite{bookSCE} for more details.

Consider functions defined on an open set $U$ in the quaternions $\mathbb{H}$ or in $\mathbb{R}^{n+1}$ for Clifford algebra-valued functions,
then the Fueter-Sce-Qian extension consists of two steps.

\medskip
Step (A) extends holomorphic function to the class of slice hyperholomorphic functions (later denoted by $SH(U)$).
 These functions are also called slice monogenic for Clifford algebra-valued functions
  (see for example \cite{6Global,6STRUC,CAUMON,6cssisrael,6css,6ISRAEL2,DKS}) and slice regular in the quaternionic case (see \cite{6adv2,6gentilistruppa}).
This function theory is nowadays well developed and, among other works, we mention e.g.
\cite{SOREN1,Cnudde,REN4} or the books \cite{6EntireBook,6css,6GSBook,6GSSbook} and the reference therein.
The case of slice hyperholomorphic functions with values in a real alternative algebra is treated in \cite{6ghiloniperotti}.

\medskip
Step (B) extends slice hyperholomorphic functions to monogenic functions (denoted by $M(U)$) or Fueter regular functions in the case of the quaternions.
The theory of monogenic functions is widely studied and the literature is rich, see e.g.
the books \cite{6DIRAC1,6DIRAC2,6DIRAC3,6DIRAC4,6DIRAC5} and references therein.

\medskip
Both the classes of hyperholomorphic functions have a Cauchy formula that can be used to define
functions of quaternionic operators or of $n$-tuples of
operators.

\medskip
The Cauchy formula of slice hyperholomorphic functions generates the $S$-functional calculus for quaternionic linear operators
or for $n$-tuples of not necessarily commuting operators.
This calculus is based on the notion of
$S$-spectrum, see \cite{6newresol,6hinfty,6MILANO,6Transaction,CAUMON,6FUNC3,JGA,CLOSED}.
This notion was discovered in 2006 by
  F. Colombo and I. Sabadini as is well explained in the introduction of the book \cite{6CKG}. This existence of an appropriate quaternionic spectrum was suggested by the formulation of quaternionic quantum mechanics given by G. Birkhoff and J. von Neumann \cite{BF}. We note that in this framework a spectral theorem for quaternionic operators is necessary, and this was proved in \cite{6SpecThm1} (see also the particular cases \cite{specMILAN,spectcomp,JONAQS}) and further generalized to fully Clifford operators in \cite{CLIFST}.
Preliminary attempts to prove the spectral theorem for quaternionic operators, without a precise notion of quaternionic spectrum, were given in \cite{Teichmueller,Viswanath}, while in \cite{fpp} the spectral theorem for quaternionic matrices is treated on the right spectrum, a subset of the $S$-spectrum.

Step (A) has generated the following research directions:
the foundation of the quaternionic spectral theory on the $S$-spectrum, see the books \cite{6CG,6CKG} and, for paravector operators, \cite{6css}; quaternionic evolution operators; Phillips functional calculus; $H^\infty$-functional calculus, see \cite{6CG}; the characteristic operator functions and applications to linear system theory \cite{6COFBook}; quaternionic spectral operators \cite{6JONAME}; quaternionic perturbation theory and invariant subspaces \cite{6CCKS}; Schur analysis in this setting \cite{6ACSBOOK}.
For some new classes of fractional diffusion problems based on fractional powers of quaternionic linear operators, see the book \cite{6CG} and the more recent contributions \cite{frac4,frac5,frac1,frac2,frac3}.
 For recent advances on the $S$-functional calculus see \cite{UNIV,ClifFUN}.

\medskip
Step (B) generates the monogenic functional calculus based on the monogenic spectrum.
Some of the research directions in this area are:
monogenic spectral theory and applications \cite{6jefferies} and also some of the original contributions
 \cite{6MQ,6JM,6JMP,6MP,6qian1};
    harmonic analysis in higher dimension, singular integrals and Fourier transform,
    see the recent book \cite{BOOKTAO};
boundary value problems treated with quaternionic techniques \cite{6GURLYSPROSS};
 Dirac operator on manifolds and spectral theory \cite{bookTF,6DIRAC4}.

The first mathematicians who understood the importance of hypercomplex analysis to define functions of noncommuting operators on Banach spaces have been
A. McIntosh and his collaborators, staring from preliminary results in \cite{KISIL}.
 Using the theory of monogenic functions they developed the monogenic functional calculus and several of its applications, see \cite{6jefferies}.
 
 \medskip
 The main goal of this paper is to study several crucial properties of
  the $\mathcal{F}$-functional calculus, which is
  is a bridge between the spectral theory on the $S$-spectrum and the monogenic spectral theory.
   For this reason this calculus will be discussed in detail in the sequel.

  \medskip
  Before to state our main results we want to make some considerations on the
  resolvent equations in various settings. These equations have many crucial consequences
  one of those leads to the Riesz projectors and it applies
  to the Riesz-Dunford functional calculus,
   to the $S$-functional calculus and also to the $\mathcal{F}$-functional calculus
   because of the following considerations.

    \medskip
    If $A$ is a complex operator on a complex Banach space then the
   resolvent equation is given by:
   {\small
  \begin{equation}\label{Creseq}
(\lambda I-A)^{-1}(\mu I-A)^{-1}=\frac{(\lambda I-A)^{-1}-(\mu I-A)^{-1}}{\mu-\lambda},\ \ \ \lambda, \ \mu\in \mathbb{C}\setminus \sigma(A),
\end{equation}}
where $\sigma(A)$ is the spectrum of $A$.
 To study the Riesz projectors
 that are defined as
{\small
$$
P_\Omega=\int_{\partial \Omega}(\lambda I-A)^{-1}d\lambda,
$$
}
where $\Omega$ contains part of the spectrum of $A$,
or  to prove the product rule for
 the holomorphic functional calculus we use
some properties of the resolvent equation that we list in the following.

\begin{itemize}
\item[(I)] The product of the resolvent operators
$(\lambda I-A)^{-1}(\mu I-A)^{-1}$, at two different points $\lambda, \ \mu\in \mathbb{C}\setminus \sigma(A)$, is transformed into the difference
 $(\lambda I-A)^{-1}-(\mu I-A)^{-1}$.

 \item[(II)]
 The difference
 $(\lambda I-A)^{-1}-(\mu I-A)^{-1}$
 is  entangled with the Cauchy kernel
 $
 1/(\mu-\lambda)
 $
 of holomorphic functions as follows
{\small
$$\frac{(\lambda I-A)^{-1}-(\mu I-A)^{-1}}{\mu-\lambda}.$$
}
\item[(III)] The resolvent equation preserves the holomorphicity both in
$\lambda$ and in $\mu\in \mathbb{C}\setminus \sigma(A)$.
\end{itemize}
 The above properties of the resolvent equation in the
hyperholomorphic spectral theories can be restored
but with nontrivial considerations.
We discuss first the case of the $S$-resolvent equation,
 but first we need some notations for the Clifford algebras.

  \medskip
Let $\rr_n$ be the real Clifford algebra over $n$ imaginary units $e_1,\ldots ,e_n$
satisfying the relations $e_\ell e_m+e_me_\ell=0$,\  $\ell\not= m$, $e_\ell^2=-1.$
 An element in the Clifford algebra will be denoted by $\sum_A e_Ax_A$ where
$A=\{ \ell_1\ldots \ell_r\}\in \mathcal{P}\{1,2,\ldots, n\},\ \  \ell_1<\ldots <\ell_r$
 is a multi-index
and $e_A=e_{\ell_1} e_{\ell_2}\ldots e_{\ell_r}$, $e_\emptyset =1$.
A point $(x_0,x_1,\ldots,x_n)\in \mathbb{R}^{n+1}$  will be identified with the element
$
 x=x_0+\underline{x}=x_0+ \sum_{j=1}^n x_j e_j\in\mathbb{R}_n
$
called paravector and the real part $x_0$ of $x$ will also be denoted by $\Re(x)$.
The vector part of $x$ is defined by
 $\underline{x}=x_1e_1+\ldots+x_ne_n$.
 The conjugate of $x$ is denoted by $\overline{x}=x_0-\underline{x}$
and the Euclidean modulus of $x$ is given by $|x|^2=x_0^2+\ldots +x_n^2$.

  In this paper we work in a Clifford algebra setting, and in the sequel
 we denote by $\mathcal{B}(V_n)$ the Banach space of all bounded right linear operators acting on a the two sided  Clifford Banach module $V_n=V\otimes \mathbb{R}_n$, where $V$ is a real Banach space.
 Then the appropriate notion of spectrum is the $S$-spectrum,
   that is defined in a very counterintuitive way because it involves
   the square of the linear operator $T$, see \cite{CSSFUNCANAL} and the book \cite{6css} for more details. We now describe how the complex resolvent equation extends to the two case of the $S$-resolvent operators and of the $\mathcal{F}$-resolvent operators.

 \subsection{ The $S$-resolvent equation for the $S$-resolvent operators}
Let $V_n$ be a two sided Clifford Banach module and
let $T:V_n\to V_n$ be a bounded right (or left) linear operator.
 The  S-spectrum of $T$ is defined as
{\small
$$
\sigma_S(T)=\{ s\in \mathbb{R}^{n+1}\ \ :\ \ T^2-2 \Re(s)T+|s|^2\mathcal{I}\ \ \
{\rm is\ not\  invertible\ in \ }\mathcal{B}(V_n)\},
$$}
and the $S$-resolvent set
{\small
$$
\rho_S(T):=\mathbb{R}^{n+1}\setminus\sigma_S(T).
$$}
In the slice hyperholomorphic setting there exist two different resolvent
 operators according to left or right slice hyperholomorphicity.
The left  and the right $S$-resolvent operators are defined as
{\small
\begin{equation}
S_L^{-1}(s,T):=-(T^2-2\Re(s) T+|s|^2\mathcal{I})^{-1}(T-\overline{s}\mathcal{I}),\ \ \ s \in  \rho_S(T)
\end{equation}
}
and
{\small
\begin{equation}
S_R^{-1}(s,T):=-(T-\overline{s}\mathcal{I})(T^2-2\Re(s) T+|s|^2\mathcal{I})^{-1},\ \ \ s \in  \rho_S(T),
\end{equation}}
respectively.
The operator
{\small
$$
\mathcal{Q}_s(T):=(T^2-2\Re(s) T+|s|^2\mathcal{I})^{-1},\ \ \ s \in  \rho_S(T),
$$}
is called the pseudo $S$-resolvent operator.
The left $S$-resolvent
operator satisfies the equation
{\small
\begin{equation}\label{eSR1}
S_L^{-1}(s,T)s-TS_L^{-1}(s,T)=\mathcal{I},\ \ \ s\in\rho_S(T),
\end{equation}}
and the right $S$-resolvent
operator satisfies
{\small
\begin{equation}\label{eSR2}
sS_R^{-1}(s,T)-S_R^{-1}(s,T)T=\mathcal{I},\ \ \ s \in \rho_S(T).
\end{equation}
}

The equations (\ref{eSR1}) and (\ref{eSR2})
cannot be considered the resolvent equations for the $S$-functional calculus
because they do not satisfy the properties of the classical resolvent equation.
 The $S$-resolvent equation involves both the $S$-resolvent operators. Precisely, we have
 {\small
\begin{equation}
\label{reso}
S_R^{-1}(s,T)S_L^{-1}(p,T)=\big[[S_R^{-1}(s,T)-S_L^{-1}(p,T)]p
-\overline{s}[S_R^{-1}(s,T)-S_L^{-1}(p,T)]\big](p^2-2s_0p+|s|^2)^{-1},
\end{equation}
}
for  $s$, $p \in  \rho_S(T)$.

 For the $S$-functional calculus we have the following major
 differences with respect to the holomorphic functional calculus.
 \begin{itemize}
 \item[(I-S)] The $S$-resolvent equation contains both the $S$-resolvent operators and an important fact to point out is that
 $S_L^{-1}(s,T)$ is right slice hyperholomorphic
 and $S_R^{-1}(s,T)$ left slice hyperholomorphic.
  The product $S_R^{-1}(s,T)S_L^{-1}(p,T)$ preserves the right slice hyperholomorphicity in $s$ and the left slice hyperholomorphicity in $p$.

\item[(II-S)]
The product $S_R^{-1}(s,T)S_L^{-1}(p,T)$ is transformed into the difference $S_R^{-1}(s,T)-S_L^{-1}(p,T)$
of the two $S$-resolvent operators, as in the complex case.

\item[(III-S)]
The difference $S_R^{-1}(s,T)-S_L^{-1}(p,T)$ is  entangled with
 $(p-\overline{s})(p^2-2s_0p+|s|^2)^{-1}$,
  which is the Cauchy kernel of slice hyperholomorphic functions, and
  the map
  {\small
 $$(s,p)\mapsto\big[[S_R^{-1}(s,T)-S_L^{-1}(p,T)]p
-\overline{s}[S_R^{-1}(s,T)-S_L^{-1}(p,T)]\big](p^2-2s_0p+|s|^2)^{-1},$$
}
for $s$, $p \in  \mathbb{R}^{n+1}\setminus\sigma_S(T)$,  preserves  the right slice hyperholomorphicity in $s$ and the left slice hyperholomorphicity in $p$.

 \end{itemize}

\begin{remark}
  The product $S_L^{-1}(p,T) S_R^{-1}(s,T)$ cannot be used in the $S$-resolvent equation because it
 destroys slice  hyperholomorphicity.
\end{remark}

The $\mathcal{F}$-functional calculus  is defined for $n$-tuples of commuting operators, $n$ odd, and
 it generates a monogenic functional calculus. It is based on the Fueter-Sce-Qian mapping theorem in integral form which, in the case $n$ odd corresponds to the Fueter-Sce version of the result.

\subsection{ The commutative $S$-resolvent operators and the $\mathcal{F}$-resolvent operators}

In the sequel, we will consider bounded paravector operators $T$,
with commuting components $T_\ell\in\mathcal{B}(V)$ for $\ell=0,1,\ldots ,n$, $n$ odd. By $\mathcal{BC}(V_n)$ we will denote the subset of  ${\mathcal{B}(V_n)}$ consisting of Clifford operators with commuting components, i.e., operators of the type
$\sum_A e_AT_A$ where
$A=\{ \ell_1\ldots \ell_r\}\in \mathcal{P}\{1,2,\ldots, n\},\ \  \ell_1<\ldots <\ell_r$
 is a multi-index and the operators $T_A$ commute among themselves.
The $S$-functional calculus admits a commutative version.
For paravector operators $T=T_0+e_1T_1+\ldots +e_nT_n$ such that  $T\in \mathcal{BC}(V_n)$
 the $\mathcal{F}$-spectrum  of $T$ is defined as
{\small
$$
\sigma_{\mathcal{F}}(T)=\{ s\in \mathbb{R}^{n+1}\ \ :\ \ \ \ \ s^2\mathcal{I}-s(T+\overline{T})+T\overline{T}\ \ \
{\rm is\ not\  invertible \ in \ }\mathcal{B}(V_n)\},
$$}
where we have set $\overline{T}:=T_0-e_1T_1-\ldots -e_nT_n$,
and the $\mathcal{F}$-resolvent set
{\small
$$
\rho_{\mathcal{F}}(T):=\mathbb{R}^{n+1}\setminus\sigma_F(T).
$$}
It turns out that the $\mathcal{F}$-spectrum is the commutative
 version of the $S$-spectrum, i.e., we have
$$
\sigma_{\mathcal{F}}(T)=\sigma_{S}(T),\ \ {\rm for } \ \ T_0+e_1T_1+\ldots +e_nT_n\in \mathcal{BC}(V_n).
$$
In the literature for the commutative version of the $S$-functional calculus and for the $\mathcal{F}$-functional calculus we will use the symbol $\sigma_{F}(T)$
to denote the commutative version of the $S$-spectrum.
The definition comes from the structure of the commutative $S$-resolvent operators. In fact,
for paravector operators $T\in \mathcal{BC}(V_n)$,
  the commutative version of left $S$-resolvent operator is defined as
{\small
\begin{equation}
S_{L}^{-1}(s,T):=(s\mathcal{I}- \overline{T})(s^2\mathcal{I}-s(T+\overline{T})+T\overline{T})^{-1},\ \ \ s \in  \rho_{\mathcal{F}}(T),
\end{equation}}
and the commutative version of the right $S$-resolvent operator is
{\small
\begin{equation}
S_{R}^{-1}(s,T):=(s^2\mathcal{I}-s(T+\overline{T})+T\overline{T})^{-1}(s\mathcal{I}- \overline{T}),\ \ \ s \in \rho_{\mathcal{F}}(T).
\end{equation}}
For the sake of simplicity we have still denoted the commutative version of the $S$-resolvent operators with the same symbols as for the noncommutative ones.
The operator
{\small
$$
\mathcal{Q}_s(T):=(s^2\mathcal{I}-s(T+\overline{T})+T\overline{T})^{-1},\ \ \ s \in  \rho_{\mathcal{F}}(T),
$$
}
is called commutative pseudo $S$-resolvent operator (for short, it is called
pseudo resolvent operator) and also here we use the same symbol for the commutative version of the
pseudo $S$-resolvent operator.

In the sequel, treating the $\mathcal{F}$-functional calculus, when we mention the
$S$-resolvent operators and the pseudo $S$-resolvent operator we intend their commutative versions.

The Fueter-Sce mapping theorem in integral form is the crucial object for the definition of the $\mathcal{F}$-functional calculus and will be discussed in the next section.
Here we just recall that the $\mathcal{F}$-functional calculus was introduced in \cite{6CoSaSo} and further investigated in \cite{CG,6FUNC2}.

\medskip
We now define the $\mathcal{F}$-resolvent operators.
Let  $n$ be an odd number, we define the left $\mathcal{F}$-resolvent operator as
{\small
\begin{equation}
\mathcal{F}_n^L(s,T):=\gamma_n(s\mathcal{I}-\overline{ T})(s^2\mathcal{I}-s(T+\overline{T})+T\overline{T})^{-\frac{n+1}{2}}
,\ \ \ s \in  \rho_{\mathcal{F}}(T)
\end{equation}
}
and the right $\mathcal{F}$-resolvent operator as
{\small
\begin{equation}
\mathcal{F}_n^R(s,T):=\gamma_n(s^2\mathcal{I}-s(T+\overline{T})+T\overline{T})^{-\frac{n+1}{2}}(s\mathcal{I}-\overline{ T}),\ \ \ s \in  \rho_{\mathcal{F}}(T),
\end{equation}
}
where the constant $\gamma_n$ is defined by
{\small
\begin{equation}\label{const}
\gamma_n:=(-1)^{\frac{n-1}{2}}2^{n-1} \left[ \left(\frac{n-1}{2}\right)!\right]^2.
\end{equation}
}

The $\mathcal{F}$-resolvent equation has further differences with respect to the complex resolvent equation and with respect to the $S$-resolvent equation. This is a consequence of the fact that the $\mathcal{F}$-functional calculus is based on an integral transform and not on a Cauchy formula.

The $\mathcal{F}$-resolvent equation for $n=3$ is known since some years and it coincides with the quaternionic $\mathcal{F}$-resolvent equation, precisely it is given by (see \cite{CG})
{\small
\begin{equation}\label{FREQ}
\begin{split}
& \mathcal{F}_3^R(s,T)S_{L}^{-1}(p,T)+S_{R}^{-1}(s,T)\mathcal{F}_3^L(p,T)
\\
&
- \frac{1}{4}\Big( s\mathcal{F}_3^R(s,T)\mathcal{F}_3^L(p,T)p -s\mathcal{F}_3^R(s,T)T\mathcal{F}_3^L(p,T)
  -\mathcal{F}_3^R(s,T)T\mathcal{F}_3^L(p,T)p+\mathcal{F}_3^R(s,T)T^2\mathcal{F}_3^L(p,T)\Big)\\
 = &\big[ \left( \mathcal{F}_3^R(s,T) - \mathcal{F}_3^L(p,T)\right)p-\bar{s}\left( \mathcal{F}_3^R(s,T) - \mathcal{F}_3^L(p,T)\right) \big](p^2-2s_0 p+|s|^2)^{-1}.
\end{split}
\end{equation}
}
for $T\in\mathcal{BC}(V_3)$ and for any  $p,s\in \rho_{\mathcal{F}}(T)$, with $s\not\in[p]$.
As we can see the equation is written just in terms of $S$-resolvent operators and of   $\mathcal{F}$-resolvent operators.
The general case is not so simple, it has been an open
 problem for some years and the reason will be clear in the outline of our main results.

\subsection{Outline of the main results}

Keeping in mind the preliminary considerations,
 we can now state our main results and compare the structure of the $\mathcal{F}$-resolvent equations with
 the ones discussed above.

\medskip
Section \ref{Prelmat} contains preliminary material of functions and operator theories.
In Section \ref{newserF}
we obtain the new series expansion of the $\mathcal{F}$-resolvent operators written  in terms of $T$ and $\overline{T}$.
Precisely for $\|T\| < |s|$, we have
{\small
$$ \mathcal{F}^L_n(s,T)= \sum_{m=2h}^{+\infty} \sum_{\ell=1}^{m-2h+1}
K_\ell(m,h) \ T^{m-2h- \ell+1}\ \overline{T}^{\ell-1}\ s^{-1-m},$$
}
and
{\small
$$ \mathcal{F}^R_n(s,x)= \sum_{m=2h}^{+\infty}\sum_{\ell=1}^{m-2h+1} K_\ell(m,h)\ s^{-1-m}\  T^{m-2h- \ell+1} \ \overline{T}^{\ell-1},$$
}
where the constants $ K_\ell(m,h)$ are explicitly computed in
(\ref{costKL}) and we have set $h:= \frac{n-1}{2}$, the Sce exponent.

\medskip
Since the structure of the $\mathcal{F}$-resolvent equation, of an arbitrary odd dimension $n$
is complicated we  have to consider first the cases $n=5$ and $n=7$.
The case $n=3$ discussed above is a very simple case and does not involve explicitly the pseudo $S$-resolvent operators and similarly for the case $n=5$.
However, the case $n=7$ clearly shows that we cannot pursue such strategy
  without getting a relation which is rather complicated and difficult to handle.

 The structure of the $\mathcal{F}$-resolvent equation has been obtained by considering
 a suitable form that allows to prove that suitable operators are the Riesz projectors for the $\mathcal{F}$-functional calculus.

For $n=5$, the $\mathcal{F}$-resolvent equation with pseudo $S$-resolvent operators
 and for a paravector operator $T \in \mathcal{BC}(V_5)$ is given by
{\small
\[
\begin{split}
& \mathcal{F}_5^R(s,T)S^{-1}_L(p,T)+S^{-1}_R(s,T)\mathcal{F}_{5}^L(p,T)+\gamma_5 \mathcal{Q}_s(T) S^{-1}_R(s,T)S^{-1}_{L}(p,T)\mathcal{Q}_p(T)
\\
&+ \gamma_5[\mathcal{Q}_s^2(T) \mathcal{Q}_p(T)+\mathcal{Q}_s(T) \mathcal{Q}_p^2(T)]
\\ \nonumber
&
= \bigl[ [\mathcal{F}_5^R(s,T)- \mathcal{F}_5^L(p,T)]p- \bar{s}[\mathcal{F}_5^R(s,T)- \mathcal{F}_5^L(p,T)]\bigl] (p^2-2s_0 p+|s|^2)^{-1},
\end{split}
\]
}
for $p, s \in \rho_{\mathcal{F}}(T)$ and where $\gamma_5$ is given by (\ref{const}) for $n=5$.
As one can note the above equation, proved in Lemma \ref{res1},
contains also the pseudo $S$-resolvent operators.
This equation can still be written
 using  the $\mathcal{F}$-resolvent operators only, see Theorem \ref{res3}.

 The case of $n=7$  becomes clearly too complicated to be
 written just in terms of the
$\mathcal{F}$- resolvent operators. This case shows that the general case of
 an arbitrary number of operators must be kept
  in the form with the pseudo $S$-resolvent operators.
This form can be achieved
thanks to the relations:
{\small
$$
\mathcal{F}_{n}^L(s,T)s-T \mathcal{F}_{n}^L(s,T)= \gamma_n Q_s(T)^{\frac{n-1}{2}}
 \ \ {\rm and} \ \ s \mathcal{F}_{n}^R(s,T)- \mathcal{F}_{n}^R(s,T)T= \gamma_n Q_s(T)^{\frac{n-1}{2}},
 $$
 }
where the constants $\gamma_n$ are given by (\ref{const}).

The $\mathcal{F}$-resolvent equation in the two cases $n=5$ and $n=7$ is treated in detail in Section \ref{casi5e7}.
Keeping in mind the cases $n=5$ and for $n=7$ one can better understand
 the structure of the general $\mathcal{F}$-resolvent equation.

 The general case is treated in Section \ref{GENERALCASE},
 where we prove the $\mathcal{F}$-resolvent equation and in its version with pseudo $S$-resolvent operators.
The general structure of the
 $\mathcal{F}$-resolvent equation involving the pseudo $S$-resolvent operators is obtained in Lemma \ref{PSGEN}: for $n>3$, $n$ odd, and Sce exponent $h= \frac{n-1}{2}$ the  equation is
{\small
\begingroup\allowdisplaybreaks\begin{align}
\label{resnfff}
& \, \, \, \,\mathcal{F}_n^R(s,T)S^{-1}_L(p,T)+S^{-1}_R(s,T)\mathcal{F}_n^L(p,T)
\\
&+ \gamma_n \biggl[ \sum_{i=0}^{h-2} \mathcal{Q}^{h-i-1}_s(T) S^{-1}_R(s,T)S^{-1}_L(p,T) \mathcal{Q}_p^{i+1}(T)
\nonumber
+ \sum_{i=0}^{h-1}  \mathcal{Q}_s^{h-i}(T) \mathcal{Q}_p^{i+1}(T) \biggl]
\\
&
=\bigl \{ \bigl[\mathcal{F}_n^R(s,T)-\mathcal{F}_n^L(p,T) \bigl]p- \bar{s}\bigl[\mathcal{F}_n^R(s,T)-\mathcal{F}_n^L(p,T) \bigl] \bigl \} (p^2-2s_0p+|s|^2)^{-1}
\nonumber
\end{align}\endgroup
}
for paravector operators $T \in \mathcal{BC}(V_n)$ and for $p,s \in \rho_{\mathcal{F}}(T)$.

The general structure of the
pseudo $\mathcal{F}$-resolvent equation is deduced in the spirit of the special case $n=5$
and $n=7$ and it is useful to study the Riesz projectors of the $\mathcal{F}$-functional calculus.
This equation
has two different expressions according to the fact
 the Sce exponent $h:=(n-1)/2$ is an even or odd number.
Precisely, for $h$ odd is obtained in Theorem \ref{FREhodd}
while the case in which $h$ is even is obtained in Theorem \ref{FREeven}.

 Let us summarize, for the $\mathcal{F}$-resolvent equation, the major similarities and the main differences
 with respect to the previous cases.
 \begin{itemize}
 \item[(I-Fa)]  There are two different $\mathcal{F}$-resolvent operators
 $\mathcal{F}_n^{L}(s,T)$ and $\mathcal{F}_n^{R}(p,T)$ which are right slice hyperholomorphic in $s$ and left slice hyperholomorphic in $p$, respectively.

 \item[(I-Fb)]
 There are additional terms containing:

 \begin{itemize}
 \item[(i)]
  The operator
   $\mathcal{F}_n^{L}(s,T) B\mathcal{F}_n^{R}(s,T)$ which, for any bounded operator $B$,
   preserves the right slice hyperholomorphicity in $s$ and the left slice hyperholomorphicity in $p$ .

\item[(ii)]
   The commutative version of the $S$-resolvent operators which appears in the terms:
   $$\mathcal{F}_n^R(s,T)S_{L}^{-1}(p,T) \ \ \ \ {\rm and} \ \ \ \ S_{R}^{-1}(s,T)\mathcal{F}_n^L(p,T).$$
 \item[(iii)]
    The commutative pseudo $S$-resolvent operators $\mathcal{Q}_s(T)$ and $\mathcal{Q}_p(T)$.
\end{itemize}
\item[(II-F)]
The difference $\mathcal{F}_n^{L}(s,T) -\mathcal{F}_n^{R}(s,T)$
 is  entangled with of the Cauchy kernel of slice hyperholomorphic functions, in the same way as the $S$-resolvent equation, i.e.,
 {\small
 $$
 (s,p)\mapsto\big[[\mathcal{F}_n^{L}(s,T) -\mathcal{F}_n^{R}(s,T)]p
-\overline{s}[\mathcal{F}_n^{L}(s,T) -\mathcal{F}_n^{R}(s,T)]\big](p^2-2s_0p+|s|^2)^{-1}
$$
}
and it preserves  the slice hyperholomorphicity on the right in $s$ and on the left in $p$.
\item[(III-F)]
 The term
{\small
 $$
 \big[[\mathcal{F}_n^{L}(s,T) -\mathcal{F}_n^{R}(s,T)]p
-\overline{s}[\mathcal{F}_n^{L}(s,T) -\mathcal{F}_n^{R}(s,T)]\big](p^2-2s_0p+|s|^2)^{-1}
$$}
 equals not only terms that involve suitable  product of the $\mathcal{F}$-resolvent operators, but also the terms described in the above
items (i), (ii) and (iii).
\end{itemize}
\begin{remark} Similarly to the case of the $S$-resolvent equation,
the products of the form
  {\small$$
  \mathcal{F}_n^{R}(s,T) B\mathcal{F}_n^{L}(s,T)
  $$}
   cannot be used in the $\mathcal{F}$-resolvent equation because they
 destroys slice hyperholomorphicity.
\end{remark}

\medskip
Even though the $\mathcal{F}$-resolvent equation is very involved  it is still the correct tool to show that the $\mathcal{F}$-functional calculus
  generate the Riesz projectors, and in Section \ref{RPROJ} we prove the following result.
Let $n>3$ be an odd number and assume that the  vector operators
$T=e_1T_1+\ldots +e_nT_n \in \mathcal{BC}(V_n)$ is such that
 $ \sigma_{\mathcal{F}}(T)= \sigma_{\mathcal{F},1}(T) \cup \sigma_{\mathcal{F},2}(T)$ with
{\small
$$
 \hbox{dist} \left(\sigma_{\mathcal{F},1}(T),\sigma_{\mathcal{F},2}(T)\right)>0,
$$}
and with
{\small
$$
\sigma(T_\ell)\subset \mathbb{R}\ \  {\rm for \ all} \ \ell=1,...,n.
$$}
Let $G_1$, $G_2$ be two admissible sets for $T$ such that $ \sigma_{\mathcal{F},1}(T) \subset G_1$ and $ \bar{G}_1 \subset G_2$ and with $dist \left(G_2, \sigma_{\mathcal{F},2}(T) \right)>0$. Then the operator
{\small
\begin{equation}
\check{P}= \frac{1}{2 \pi\gamma_n } \int_{\partial(G_1 \cap \mathbb{C}_I)} \mathcal{F}_n^L(p,T) dp_I p^{n-1}=\frac{1}{2 \pi\gamma_n} \int_{\partial(G_2 \cap \mathbb{C}_I)}  s^{n-1} ds_I \mathcal{F}_n^R(s,T),
\end{equation}
}
is a projector and the constants $\gamma_n$ are given in (\ref{const}).

\section{\bf Preliminary material}\label{Prelmat}
\setcounter{equation}{0}

\medskip
The real Clifford algebra  $\rr_n$ over $n$ imaginary units $e_1,\ldots ,e_n$ has been defined in the Introduction.
Slice hyperholomorphic and monogenic functions
can be seen as the two classes of functions appearing in the two steps in the Fueter-Sce-Qian constructions, which extends holomorphic functions to
 quaternionic or Clifford algebra valued-functions satisfying a suitable analyticity condition.

In fact, starting from holomorphic functions, R. Fueter  in 1934, see \cite{6fueter}, showed
 an interesting way to generate quaternionic Cauchy-Fueter regular functions.
More then 20 years later in 1957  M. Sce, see \cite{6sce2},  extended this result to Clifford algebras in a very pioneering and general way,
see the English translation of his works in hypercomplex analysis with commentaries collected in the recent book \cite{bookSCE}.

\medskip
In the original construction of R. Fueter the holomorphic
  functions are defined on open sets of the complex upper half plane. This condition can be relaxed
by taking functions
$$
g(z) = g_0(u,v)+ig_1(u,v),\ \ \ z=u+iv
$$
 defined in a set $D\subseteq\mathbb C$, symmetric with respect to the real axis such that
 $$
g_0(u,-v)= g_0(u,v) \ \ \ {\rm and} \ \ \  g_1(u,-v)= -g_1(u,v)
$$
 namely if $g_0$ and $g_1$ are, respectively, even and odd functions in the variable $v$.
 Additionally the pair $(g_0,g_1)$ satisfies the Cauchy-Riemann system.
 The same conditions are required also in the higher dimension case originally proved by Sce:

 \begin{theorem}[Sce \cite{6sce2}]\label{SCETH}
Consider the Euclidean space $\rr^{n+1}$, $n$ odd, whose elements are identified with
 paravectors $x=x_0+\underline{x}$.
Let $\tilde{f}(z) = f_0(u,v)+if_1(u,v)$ be a  holomorphic function
defined in a domain (open and connected) $D$ in the upper-half complex plane and let
$$
\Omega _D= \{x =x_0+\underline{x}\ \ |\ \  (x_0, |\underline{x}|) \in D\}
$$
be the open set induced by $D$ in $\mathbb{R}^{n+1}$.
The map
{\small $$
f(x)=T_{FS1}(\tilde{f}):=\textcolor{black}{f_0(x_0,|\underline{x}|)+\frac{\underline{x}}{|\underline{x}|}f_1(x_0,|\underline{x}|)}
$$}
takes the holomorphic functions $\tilde{f}(z)$ and gives the Clifford-valued function $f(x)$.
Then the function
{\small
$$
\breve{f}(x):=\textcolor{black}{T_{FS2}} \Big(\textcolor{black}{f_0(x_0,|\underline{x}|)+\frac{\underline{x}}{|\underline{x}|}f_1(x_0,|\underline{x}|)}\Big),
$$}
where $T_{FS2}:=\Delta_{n+1}^{\frac{n-1}{2}}$ and $\Delta_{n+1}$ is the Laplacian in $n+1$ dimensions,
is in the kernel of the Dirac operator, i.e.,
$$
D\breve{f}(x)=0 \ \ \ {\it on} \ \ \Omega_D.
$$

\end{theorem}
Functions kernel of the Dirac operator are called monogenic functions and their precise definition is as follows.
\begin{definition}[Monogenic functions]\index{Monogenic functions} Let $U$ be an open set in $\mathbb{R}^{n+1}$.
A real differentiable function $\breve{f}: U\to \mathbb{R}_n$ is left monogenic if
{\small$$
D\breve{f}(x):=\frac{\pp}{\pp x_0}\breve{f}(x)+\sum_{i=1}^n e_i\frac{\pp}{\pp x_i}\breve{f}(x)=0.
$$
}
It is right monogenic if
{\small
$$
\breve{f}(x)D:=\frac{\pp}{\pp x_0}\breve{f}(x)+\sum_{i=1}^n \frac{\pp}{\pp x_i}\breve{f}(x)e_i=0.
$$}
\end{definition}

The case where the operator $\Delta_{n+1}^{\frac{n-1}{2}}$ has a fractional index has been treated by T. Qian in \cite{6qian}.
Observe that for the Fueter's theorem  the operator $T_{FS2}$ is equal to the Laplacian $\Delta$ in $4$ dimensions.
Further developments can be found in \cite{6DIX1,6DIX2,6DIX4} see also the survey \cite{6qianSR}.

Recently also the problem of construction the inversion of the maps that appear in the Fueter-Sce-Qian extension has been treated, see
the papers \cite{6BKQS,6BKQS1,6DIX3, 6CoSaSo1,6csso1}.

To define the notion of slice monogenic functions, we need to introduce more notations. The sphere of purely imaginary paravectors with modulus $1$, is defined by
$$
\mathbb{S}=\{x=e_1x_1+\ldots +e_nx_n\ |\  x_1^2+\ldots +x_n^2=1\}.
$$
The element  $I\in \mathbb{S}$ is such that $I^2=-1$, so $I$ is an imaginary unit, and we will denote the complex space with imaginary unit $I$ by $\mathbb{C}_I$.
Given a non-real paravector $x=x_0+\underline{x}=x_0+J_x |\underline{x}|$, we set $J_x:=\underline{x}/|\underline{x}|\in\mathbb{S}$, and we associate to $x$ the sphere defined by
$$
[x]=\{x_0+J  |\underline{x}| \ \ | \ \ J \in\mathbb{S}\}.
$$
\begin{definition}
Let $U \subseteq \mathbb R^{n+1}$ . We say that $U$ is
\textnormal{axially symmetric}  if, for every $u+Iv \in U$,  all the elements $u+Jv$ for $J\in\mathbb{S}$ are contained in $U$.
\end{definition}
For operator theory the most appropriate definition of
 slice hyperholomorphic functions is the one that comes from the Fueter-Sce-Qian mapping theorem
 because it allows to define functions on axially symmetric open sets. It can be given as follows, see \cite{6ghiloniperotti}.
\begin{definition}\label{CAUSLICE}
 Let $U\subseteq\mathbb{R}^{n+1}$ be an axially symmetric open set
 and let $\mathcal{U}\subseteq\mathbb{R}\times \mathbb{R}$ be such that $x=u+J v\in U$ for all $(u,v)\in\mathcal{U}$.
We say that a function $f:U \to \mathbb{R}_n$ of the form
$$
f(x)=f_0(u,v)+J f_1(u,v)
$$
is left slice hyperholomorphic if
 $f_0$, $f_1$ are $\mathbb{R}_n$-valued differentiable functions such that
 $$
 f_0(u,v)=f_0(u,-v), \ \ \ f_1(u,v)=-f_1(u,-v)\ \ \ {\rm for \ all}\ \  (u,v)\in \mathcal{U}
 $$
 and if $f_0$ and $f_1$ satisfy the Cauchy-Riemann system
 $$
\partial_u f_0-\partial_vf_1=0,\ \ \ \ \
\partial_vf_0+\partial_u f_1=0.
$$
We recall that right slice hyperholomorphic  functions are of the form
$$
f(x)=f_0(u,v)+f_1(u,v)J
$$
where $f_0$, $f_1$ satisfy the above conditions.
\end{definition}
The set of left (resp. right) slice hyperholomorphic function
on $U$ is denoted with the symbol ${SH}_L(U)$ (resp. ${SH}_R(U)$).
The subset of intrinsic functions consist of those slice hyperholomorphic functions such that
$f_0$, $f_1$ are real-valued and is denoted by $N(U)$.

\medskip
With this language we can reformulate the Fueter-Sce theorem since $T_{FS2}$ acts on slice hyperholomorphic functions. We note that a different method to connect slice hyperholomorphic and monogenic functions is the
Radon and dual Radon transform, see \cite{6Radon}.

\medskip

We now recall the hyperholomorphic Cauchy formulas that are the heart of the hyperholomorphic spectral theories.
It is important to remark that the hypotheses of the following Cauchy formula
are related to the Definition \ref{CAUSLICE} of slice hyperholomorphic functions.

\begin{definition}
Let  $x\not\in [s]$. We define and
\begin{equation}\label{LformI}
S_L^{-1}(s,x):=-(x^2 -2 \Re  (s) x+|s|^2)^{-1}(x-\overline s)
\end{equation}
\begin{equation}\label{LformII}
\ \ \ \ \ \ \ \ \ \ \ =(s-\bar x)(s^2-2\Re (x) s+|x|^2)^{-1},
\end{equation}
\begin{equation}\label{RformI}
S_R^{-1}(s,x):=-(x-\bar s)(x^2-2\Re (s)x+|s|^2)^{-1}
\end{equation}
\begin{equation}\label{RformII}
\ \ \ \ \ \ \ \ \ \ \ =(s^2-2\Re (x)s+|x|^2)^{-1}(s-\bar x).
\end{equation}
We say that \eqref{LformI} is the Cauchy kernel $S_L^{-1}(s,x)$ in form I, while \eqref{LformII} is in the form II.\\
Analogously, we say that \eqref{RformI} is the Cauchy kernel $S_R^{-1}(s,x)$ in form I, while \eqref{RformII} is in the form II.\\
\end{definition}
The following result is crucial:
\begin{lemma} Let  $s\notin [x]$.
The left slice hyperholomorphic Cauchy kernel $S_L^{-1}(s,x)$ is left slice hyperholomorphic in $x$ and right slice hyperholomorphic in $s$.
The right slice hyperholomorphic Cauchy kernel $S_R^{-1}(s,x)$ is left slice hyperholomorphic in $s$ and right slice hyperholomorphic in $x$.
\end{lemma}
\begin{theorem}[The Cauchy formulas for slice monogenic functions]
\label{CauchygeneraleMONOG}
Let $U\subset\mathbb{R}^{n+1}$ be a bounded slice Cauchy domain, let $J\in\mathbb{S}$ and set  $ds_J=ds (-J)$.
If $f$ is a (left) slice monogenic function on a set that contains $\overline{U}$ then
{\small
\begin{equation}\label{cauchynuovo}
 f(x)=\frac{1}{2 \pi}\int_{\partial (U\cap \mathbb{C}_J)} S_L^{-1}(s,x)\, ds_J\,  f(s),\qquad\text{for any }\ \  x\in U.
\end{equation}}
If $f$ is a right slice hyperholomorphic function on a set that contains $\overline{U}$,
then
{\small
\begin{equation}\label{Cauchyright}
 f(x)=\frac{1}{2 \pi}\int_{\partial (U\cap \mathbb{C}_J)}  f(s)\, ds_J\, S_R^{-1}(s,x),\qquad\text{for any }\ \  x\in U.
 \end{equation}
 }
These integrals  depend neither on $U$ nor on the imaginary unit $J\in\mathbb{S}$.
\end{theorem}
We are now ready to observe that the $\mathcal{F}$-functional calculus is based on the following strategy.
We apply the second Fueter-Sce operator $T_{FS2}:=\Delta^{(n-1)/2}$,
where $\Delta=\sum_{i=0}^n\frac{\partial^2}{\partial x_i^2}$ is the Laplace operator in the variables $(x_0,x_1,...,x_n)$, to the slice hyperholomorphic Cauchy kernel as illustrated by the diagram:
{\small
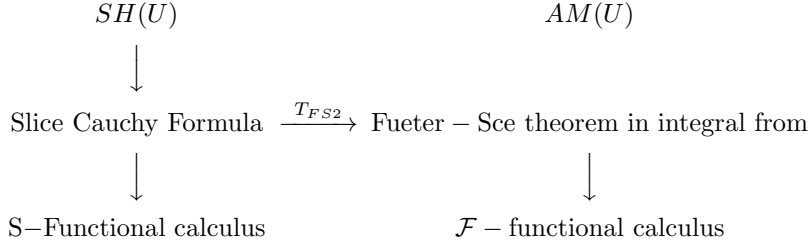
\begin{equation*}
\begin{CD}
{SH(U)} @.  {AM(U)} \\   @V  VV
  @.
\\
{{\rm  Slice\ Cauchy \ Formula}}  @> T_{FS2}>> {{\rm Fueter-Sce\ theorem \ in \  integral\  from}}
\\
@V VV    @V VV
\\
$S$-{{\rm Functional \ calculus}} @.\mathcal{F}-{{\rm functional \ calculus}}
\end{CD}
\end{equation*}
}
\begin{remark}
Observe that in the above diagram the arrow from the space of axially monogenic function $AM(U)$ is missing because the $\mathcal{F}$-functional calculus is deduced from the slice hyperholomorphic Cauchy formula.
\end{remark}

\begin{proposition}\label{Laplacian}
Let $n$ be an odd number and
let $x$,
$s\in \rr^{n+1}$
be such that
 $x\not\in [s]$.
 Let $S_L^{-1}(s,x)$ and $S_R^{-1}(s,x)$ be the slice hyperholomorphic Cauchy kernels in form II. Then:
 \begin{itemize}
 \item
The function $\Delta^{\frac{n-1}{2}}S_L^{-1}(s,x)$ is a left monogenic function in the variable $x$ and right slice hyperholomorphic in $s$.
\item
The function  $\Delta^{\frac{n-1}{2}}S_R^{-1}(s,x)$ is a right monogenic function in the variable $x$ and left slice hyperholomorphic in $s$.
\end{itemize}
\end{proposition}

Based on the explicit computations of functions
{\small
$$
(s,x)\mapsto\Delta^{\frac{n-1}{2}}S_L^{-1}(s,x) \ \ \  \ \text{and}\ \ \ \ \ (s,x)\mapsto\Delta^{\frac{n-1}{2}}S_R^{-1}(s,x),
$$}
for $s\not\in[x]$,  we define the $\mathcal{F}$-kernels.
\begin{theorem}\label{DSF}
Let $n$ be an odd number and  let $x$, $s\in \rr^{n+1}$.
For $s\not\in[x]$, we have
{\small
$$
\Delta^{\frac{n-1}{2}}S_L^{-1}(s,x)=\gamma_n(s-\bar x)(s^2-2\Re (x)s +|x|^2)^{-\frac{n+1}{2}},
$$}
and
{\small
$$
\Delta^{\frac{n-1}{2}}S_R^{-1}(s,x)=\gamma_n(s^2-2\Re (x)s +|x|^2)^{-\frac{n+1}{2}}(s-\bar x),
$$}
where the constants $\gamma_n$ are defined in
(\ref{const}).
\end{theorem}
\begin{definition}[The $\mathcal{F}$-kernels]
Let $n$ be an odd number and  let $x$, $s\in \rr^{n+1}$.
We define, for $s\not\in[x]$, the $\mathcal{F}_n^L$-kernel as
{\small
$$
\mathcal{F}_n^L(s,x):=\Delta^{\frac{n-1}{2}}S_L^{-1}(s,x)=\gamma_n(s-\bar x)(s^2-2\Re (x)s +|x|^2)^{-\frac{n+1}{2}},
$$}
and the $\mathcal{F}_n^R$-kernel as
{\small
$$
\mathcal{F}_n^R(s,x):=\Delta^{\frac{n-1}{2}}S_R^{-1}(s,x)=\gamma_n(s^2-2\Re (x)s +|x|^2)^{-\frac{n+1}{2}}(s-\bar x),
$$}
where the constant $\gamma_n$  are defined in
(\ref{const}).
\end{definition}

\begin{theorem}[The Fueter-Sce mapping theorem in integral form]
Let $U\subset\mathbb{R}^{n+1}$ be a bounded slice Cauchy domain, let $J\in\mathbb{S}$ and set  $ds_J=ds (-J)$.

\begin{enumerate}[(a)]
\item
If $f$ is a (left) slice monogenic function on a set that contains $\overline{U}$, then
the left monogenic function  $\breve{f}(x)=\Delta^{\frac{n-1}{2}}f(x)$
admits the integral representation
{\small
\begin{equation}\label{FuetLSEC}
\breve{f}(x)=\frac{1}{2 \pi}\int_{\pp (U\cap \mathbb{C}_J)} \mathcal{F}_n^L(s,x)ds_J f(s).
\end{equation}}
\item
If $f$ is a right slice monogenic function on a set that contains $\overline{U}$, then
the right monogenic function $\breve{f}(x)=\Delta^{\frac{n-1}{2}}f(x)$
 admits the integral representation
{\small
\begin{equation}\label{FuetRSCE}
  \breve{f}(x)=\frac{1}{2 \pi}\int_{\pp (U\cap \mathbb{C}_J)} f(s)ds_J \mathcal{F}_n^R(s,x).
\end{equation}
}
\end{enumerate}
The integrals  depend neither on $U$ and nor on the imaginary unit $J\in\mathbb{S}$.
\end{theorem}

 In the sequel, we will consider bounded paravector operators $T=e_1T_1+\ldots +e_nT_n$,
   with commuting components $T_\ell$ acting on a real vector space $V$, i.e. $T_\ell\in\mathcal{B}(V)$ for $\ell=0,1,\ldots ,n$. The set of bounded paravector operators is denoted by $\mathcal{BC}^{0,1}(V_n)$ where $V_n=V\otimes \mathbb{R}_n$.
The subset of  ${\mathcal{B}(V_n)}$ given by the operators $T$ with commuting components $T_\ell$
will be denoted by $\mathcal{BC}(V_n)$.

\medskip

Let $T\in \mathcal{BC}^{0,1}(V_n)$. We denote by $\mathcal{SM}_L(\sigma_{S}(T))$, $\mathcal{SM}_R(\sigma_{S}(T))$
the set of all  left (or right)  slice hyperholomorphic functions $f$ with $\sigma_{S}(T)\subset \dom(f)$. We now recall the following definition:

\begin{definition}[The  $S$-functional calculus for $n$-tuples of operators]
Let $V_n$ be a two sided  Banach module and  $T\in\mathcal{B}^{0,1}(V_n)$.
  Let $U\subset  \mathbb{R}^{n+1}$  be a bounded slice Cauchy domain that contains $\sigma_{S}(T)$ and set $ds_J=- ds J$. We define
{\small
\begin{equation}\label{Scalleft}
f(T)={{1}\over{2\pi }} \int_{\partial (U\cap \mathbb{C}_J)} S_L^{-1} (s,T)\  ds_J \ f(s), \ \ {\it for}\ \ f\in \mathcal{SM}_L(\sigma_S(T)),
\end{equation}}
and
{\small
\begin{equation}\label{Scalright}
f(T)={{1}\over{2\pi }} \int_{\partial (U\cap \mathbb{C}_J)} \  f(s)\ ds_J
 \ S_R^{-1} (s,T),\ \  {\it for}\ \ f\in \mathcal{SM}_R(\sigma_S(T)).
\end{equation}
}
\end{definition}
The definition of the $S$-functional calculus is well posed since
the integrals in (\ref{Scalleft}) and (\ref{Scalright}) depend neither on $U$ and nor on the imaginary unit $J\in\mathbb{S}$.
\begin{remark}
The $S$-functional calculus works more in general for fully Clifford operators with non commuting components.
\end{remark}

The $\mathcal{F}$-functional calculus  is limited to paravector operators with commuting components. It is based on the commutative version of the $S$-spectrum, the so-called $\mathcal{F}$-spectrum, and on the $\mathcal{F}$-resolvent operators that we shall define below.
\begin{definition}[The $\mathcal{F}$-functional calculus for bounded operators]
Let $n$ be an odd number, let
 $T= e_1T_1 + \dots  +  e_n T_n\in\mathcal{BC}(V_n)$, assume that the operators $T_{\ell}$, $\ell=1,..,n$ have real spectrum
 and set $ds_J=ds/J$.
 For any function $f\in\mathcal{SM}_L(\sigma_S(T))$, we define
{\small
\begin{equation}\label{DefFCLUb}
\breve{f}(T):=\frac{1}{2\pi}\int_{\pp(U\cap \mathbb{C}_J)} \mathcal{F}_n^L(s,T) \, ds_J\, f(s).
\end{equation}
}
For any $f\in\mathcal{SM}_R(\sigma_S(T))$, we define
{\small
\begin{equation}\label{SCalcMON}
\breve{f}(T):=\frac{1}{2\pi}\int_{\pp(U\cap \mathbb{C}_J)} f(s) \, ds_J\, \mathcal{F}_n^R(s,T),
\end{equation}
}
where $J\in \mathbb{S}$ and  $U$ is a slice Cauchy domain $U$.
\end{definition}
The definition of the $\mathcal{F}$-functional calculus is well posed since
the integrals in (\ref{DefFCLUb}) and (\ref{SCalcMON}) depend neither on $U$ and nor on the imaginary unit $J\in\mathbb{S}$.

Observe that the left $\mathcal{F}$-resolvent operator can be written in terms of the pseudo $S$-resolvent operators as follows.

\begin{definition}[$ \mathcal{F}$- resolvent operators]
Let $n$ be an odd number. Let us consider $T=T_0+ T_1e_1 + \dots  + T_n e_n \in \mathcal{BC}^{0,1}(V_n)$. For $s \in \rho_\mathcal{F}(T)$, we define the left $ \mathcal{F}$- resolvent operator as
{\small
\begin{equation}
\label{0pre0}
\mathcal{F}_{n}^L(s,T)= \gamma_n(s \mathcal{I}- \overline{T}) Q_s(T)^{\frac{n+1}{2}},
\end{equation}
}
and the right $ \mathcal{F}$-resolvent operator as
{\small
\begin{equation}
\label{1pre1}
\mathcal{F}_{n}^R(s,T)= \gamma_n Q_s(T)^{\frac{n+1}{2}}(s \mathcal{I}- \overline{T}),
\end{equation}
}
where $\gamma_n$ are defined in \eqref{const}.
\end{definition}
For these operators hold the following relations, proved in \cite[Thm. 5.1]{CG}
\begin{theorem}[The left and right $ \mathcal{F}$-resolvent equations]\label{LRFRESEQ}
\label{FRE}
Let $n$ be an odd number and let $T \in \mathcal{B}^{0,1}(V_n)$. Let $s \in \rho_{\mathcal{F}}(T)$. Then the $ \mathcal{F}$-resolvent operators satisfy the equations
{\small
\begin{equation}
\label{eq2}
\mathcal{F}_{n}^L(s,T)s-T \mathcal{F}_{n}^L(s,T)= \gamma_n Q_s(T)^{\frac{n-1}{2}}
\end{equation}
}
and
{\small
\begin{equation}
\label{eq3}
s \mathcal{F}_{n}^R(s,T)- \mathcal{F}_{n}^R(s,T)T= \gamma_n Q_s(T)^{\frac{n-1}{2}},
\end{equation}
}
where the constants $\gamma_n$ are given by (\ref{const}).
\end{theorem}

\section{\bf New series expansions for the $\mathcal{F}$-resolvent operators}\label{newserF}
\setcounter{equation}{0}

The Cauchy kernel of slice monogenic functions was deduced by computing the  series expansions
for the Cauchy kernels
{\small \begin{equation}\label{CKSer}
 \sum_{m = 0}^{+\infty} x^ms^{-m-1} \ \ {\rm and } \ \ \sum_{m = 0}^{+\infty} s^{-m-1}x^m
 \end{equation}
 }
which converge for $x,s\in \mathbb{R}^{n+1}$ with $|x|<|s|$. The following result can be found e.g. in \cite{6css}:
\begin{lemma}
Let $x,s\in \mathbb{R}^{n+1}$ with $|x|<|s|$. Then we have
{\small \begin{equation}\label{CKSerL}
\sum_{m = 0}^{+\infty} x^ms^{-m-1} = -(x^2 - 2\Re(s) x + |s|^2)^{-1}(x-\overline{s}),
\end{equation}}
and
{\small
\begin{equation}\label{CKSerR}
\sum_{m = 0}^{+\infty} s^{-m-1} x^m= -(x-\overline{s})(x^2 - 2\Re(s) x + |s|^2)^{-1}.
\end{equation}
}

\end{lemma}
One of the crucial facts in the spectral theory of the $S$-spectrum is that the above series expansions and the sum of the series keeps the same structure also when we replace the paravector $x$ by a Clifford operators $T$. From here and further crucial considerations follows the $S$-functional calculus.

In this section we want to give a new series expansions of the $\mathcal{F}$-resolvent operators
starting from the explicit computations of the terms $\Delta^{\frac{n-1}{2}} x^m$ in the series expansion
{\small
$$
\sum_{m = 0}^{+\infty} \Delta^{\frac{n-1}{2}} x^ms^{-m-1}
$$}
in terms of the powers of $x$ and of $\overline{x}$.
 In order to do this we need the following result, see \cite{DDG}.
 See also the interesting paper \cite{B}.
\begin{theorem}
\label{aak}
Let $x \in \mathbb{R}^{n+1}$ and set $h:= \frac{n-1}{2}$. Then:
\\
if $m> 2h$ we have
{\small
\begin{equation}
\label{g1}
\Delta^{h} x^m= \sum_{\ell=1}^{m-2h+1} \, K_\ell(m,h) \ x^{m-2h- \ell+1}\ \bar{x}^{\ell-1},
\end{equation}
}
where
{\small
\begin{equation}\label{costKL}
K_\ell(m,h):=4^h (-1)^h h\ \frac{(m- \ell-h+1)!( \ell+h-2)!}{(\ell-1)! (m- \ell-2h+1)!};
\end{equation}
}
if $m= 2h$ we have
{\small
$$ \Delta^{h} x^{2h}=h4^h(-1)^h h!(h-1)!,$$
}
and if $m< 2h$
{\small
$$ \Delta^{h} x^m=0.$$
}
\end{theorem}
We can now give the following
\begin{definition}[$\mathcal{F}$-kernel series]
Let $s,x \in \mathbb{R}^{n+1}$ and $h= \frac{n-1}{2}$ be the Sce exponent. We define the
left $\mathcal{F}$-kernel series as
{\small
$$ \sum_{m = 2h}^{+\infty} \Delta^h x^m s^{-1-m},$$
}
and the right $\mathcal{F}$-kernel series as
{\small
$$ \sum_{m =2h}^{+\infty}  s^{-1-m}\Delta^h x^m.$$
}
\end{definition}

\begin{proposition}
\label{g2}
For $s,x \in \mathbb{R}^{n+1}$ with $|x| < |s|$, the $ \mathcal{F}$-kernel series converge.
\end{proposition}
\begin{proof}
By Theorem \ref{aak} for $m \geq 2h$ we have that
{\small
 \begingroup\allowdisplaybreaks\begin{align}
| \Delta^h x^m| &\leq 4^h (-1)^h h\sum_{\ell=1}^{m-2h+1} \frac{(m- \ell-h+1)!( \ell+h-2)!}{(\ell-1)! (m- \ell-2h+1)!} |x|^{m-2h- \ell+1} |x|^{\ell-1}
\\ \nonumber
&= 4^h (-1)^h h \left( \sum_{\ell=1}^{m-2h+1} \frac{(m- \ell-h+1)!( \ell+h-2)!}{(\ell-1)! (m- \ell-2h+1)!} \right) |x|^{m-2h}.
\end{align}\endgroup
}
By the technical Lemma \ref{app2} (proved in the Appendix) we get
{\small
$$ \sum_{\ell=1}^{m-2h+1} \frac{(m- \ell-h+1)!( \ell+h-2)!}{(\ell-1)! (m- \ell-2h+1)!}= \frac{(h-1)! h! m!}{(2h)!(m-2h)!}, \qquad m \geq 2h,$$
}
and so
{\small
$$ \sum_{m = 2h}^{+\infty} |\Delta^h x^m s^{-1-m}| \leq \frac{4^h (-1)^h (h!)^2}{(2 h)!}\sum_{m = 2h}^{+\infty} \frac{m!}{(m-2h)!} |x|^{m-2h}|s|^{-1-m}.$$
}
The last series converge, by the ratio test since $|x| < |s|$. Indeed
{\small
$$
\lim_{m \to + \infty} \frac{(m+1)!}{(m+1-2h)!}\frac{(m-2h)!}{m!} \frac{|x|^{m+1-2h}|s|^{-2-m}}{ |x|^{m-2h}|s|^{-1-m}}= \lim_{m \to + \infty} \frac{(m+1)}{(m+1-2h)} |x| |s|^{-1}=|x| |s|^{-1}<1.
$$
}
The convergence of the right $ \mathcal{F}$-kernel series can be proved with similar computations.
\end{proof}

\begin{lemma}
Let $x, s \in \mathbb{R}^{n+1}$. For $|x| < |s|$, we have
{\small
$$ \mathcal{F}^L_n(s,x)= \sum_{m=0}^{+ \infty} \Delta^h x^m s^{-1-m},$$
}
and
{\small
$$ \mathcal{F}^R_n(s,x)= \sum_{m=0}^{+ \infty}  s^{-1-m}\Delta^h x^m.$$
}
\end{lemma}
\begin{proof}
We rewrite the left Cauchy kernel using the Taylor expansion
{\small
$$ S^{-1}_L(s,x)= \sum_{m=0}^{+\infty} x^m s^{-1-m},$$
}
and from (\ref{CKSer}) we get
{\small
$$ \mathcal{F}^L_n(s,x)= \Delta^h S^{-1}_L(s,x)=\sum_{m=0}^{+\infty} \left( \Delta^h x^m \right) s^{-1-m},$$
}
where we can exchange the sum and the Laplacian by Proposition \ref{g2}.
\\ A similar reasoning holds for the right $ \mathcal{F}$-kernel series.
\end{proof}

\begin{proposition}
Let $x, s \in \mathbb{R}^{n+1}$ and set $h=\frac{n-1}{2}$. Then, for $|x| < |s|$, we have
{\small
\begin{equation}
\sum_{m=2h}^{+\infty} \Delta^h x^m s^{-1-m}=\gamma_n (s - \bar{x}) (s^2-2\hbox{Re}(x)s+|x|^2)^{-(h+1)},
\end{equation}
}
and
{\small
\begin{equation}
\label{g31}
\sum_{m=2h}^{+\infty} s^{-1-m} \Delta^h x^m =\gamma_n (s^2-2\hbox{Re}(x)s+|x|^2)^{-(h+1)} (s- \bar{x}).
\end{equation}
}
\end{proposition}
\begin{proof}
By Proposition \ref{g2} and Theorem \ref{DSF} we have that
{\small
 \begingroup\allowdisplaybreaks\begin{align}
\sum_{m = 2h}^{+\infty} \Delta^h x^m s^{-1-m} &= \Delta^h \sum_{m=0}^{+ \infty} x^m s^{-1-m}
\\ \nonumber
&= \Delta^h S^{-1}_{L}(s,x)= \gamma_n (s- \bar{x}) (s^2- 2 \hbox{Re}(x) s+|x|^2)^{-(h+1)},
\end{align}\endgroup
}
and similarly we can prove \eqref{g31}.
\end{proof}

\begin{definition}[Series expansions of the $\mathcal{F}$-kernels]
Let $x, s \in \mathbb{R}^{n+1}$. We set $h:= \frac{n-1}{2}$. For $|x| < |s|$, we have
$$ \mathcal{F}^L_n(s,x)= \sum_{m=2h}^{+\infty} \sum_{\ell=1}^{m-2h+1}
K_\ell(m,h) \ x^{m-2h- \ell+1}\ \bar{x}^{\ell-1}\ s^{-1-m},$$
and
$$ \mathcal{F}^R_n(s,x)= \sum_{m=2h}^{+\infty}\sum_{\ell=1}^{m-2h+1} K_\ell(m,h)\ s^{-1-m}\  x^{m-2h- \ell+1} \ \bar{x}^{\ell-1},$$
where $ K_\ell(m,h)$
is defined in (\ref{costKL}).
\end{definition}
Using the new expression of the $\mathcal{F}$-kernels as functions of $x$ and $\overline{x}$ we obtain at the following new expansions of the $\mathcal{F}$-resolvent operators in terms on $T$ and $\overline{T}$.

\begin{definition}[Series expansions of the $\mathcal{F}$-resolvent operators]
Let $ s \in \mathbb{R}^{n+1}$. We set $h:= \frac{n-1}{2}$. For $\|T\| < |s|$, we have
{\small
$$ \mathcal{F}^L_n(s,T)= \sum_{m=2h}^{+\infty}  \sum_{\ell=1}^{m-2h+1}
K_\ell(m,h) \ T^{m-2h- \ell+1}\ \overline{T}^{\ell-1}\ s^{-1-m},$$
}
and
{\small
$$ \mathcal{F}^R_n(s,x)= \sum_{m=2h}^{+\infty}\sum_{\ell=1}^{m-2h+1} K_\ell(m,h)\ s^{-1-m}\  T^{m-2h- \ell+1} \ \overline{T}^{\ell-1}$$
}
where $ K_\ell(m,h)$
is as in (\ref{costKL}).
\end{definition}

\begin{theorem}[{\bf Series expansions of the $\mathcal{F}$-resolvent operators as functions of $T$ and $\overline{T}$}] Let $s \in \mathbb{R}^{n+1}$ and set $h:= \frac{n-1}{2}$.
For $\|T\| < |s|$, we have
{\small
\begin{equation}
\sum_{m=2h}^{+\infty}  \sum_{\ell=1}^{m-2h+1}
K_\ell(m,h) \ T^{m-2h- \ell+1}\ \overline{T}^{\ell-1}\ s^{-1-m}
=\gamma_n (s\mathcal{I} - \overline{T}) (s^2\mathcal{I}-(T+\overline{T})s +T\overline{T})^{-(h+1)},
\end{equation}
}
and
{\small
\begin{equation}
\label{g3}
\sum_{m=2h}^{+\infty}\sum_{\ell=1}^{m-2h+1} K_\ell(m,h)\ s^{-1-m}\  T^{m-2h- \ell+1} \ \overline{T}^{\ell-1} =\gamma_n (s^2\mathcal{I}-(T+\overline{T})s +T\overline{T})^{-(h+1)} (s\mathcal{I}- \overline{T}),
\end{equation}
}
where $ K_\ell(m,h)$
is defined in (\ref{costKL}) and $\gamma_n$ are as in (\ref{const}).
\end{theorem}
\begin{proof}
It follows by the previous results by replacing $x$ by the paravector operator $T$.
\end{proof}

\section{\bf The $\mathcal{F}$-resolvent equation for $n=5$ and for $n=7$}\label{casi5e7}
\setcounter{equation}{0}

The $\mathcal{F}$-resolvent equation is a quite complicated object. In order to explain
how to obtain it in the general case we treat separately the cases $n=5$ and $n=7$.
In the case $n=5$ it is clear that the equation can be written in
 a quite reasonable way in terms of the $\mathcal{F}$-resolvent operators.
 But starting from $n=7$ this choice cannot be made anymore because it
leads to an equation that is ways too complicated.
 This is the reason for which we cannot
replace the pseudo $S$-resolvent operators.

\subsection{\bf The $\mathcal{F}$-resolvent equation for $n=5$}

In this case we show the $ \mathcal{F}$- resolvent equation
   establishes a link between the difference $\mathcal{F}_5^R(s,T) -\mathcal{F}_5^L(p,T)$, the slice Cauchy kernel and suitable operators as listed in  the introduction.

To prove the $ \mathcal{F}$ resolvent equation we need the following technical result involving the pseudo $S$-resolvent operators.
\begin{lemma}[The $ \mathcal{F}$- resolvent equation for $n=5$, with the pseudo $S$-resolvent operators]
\label{res1}
Let $T \in \mathcal{BC}^{0,1}(V_5)$. Then for $p, s \in \rho_{\mathcal{F}}(T)$ the following equation holds
{\small
\begin{eqnarray}
\label{F5}
& \mathcal{F}_5^R(s,T)S^{-1}_L(p,T)+S^{-1}_R(s,T)\mathcal{F}_{5}^L(p,T)+\gamma_5 \mathcal{Q}_s(T) S^{-1}_R(s,T)S^{-1}_{L}(p,T)\mathcal{Q}_p(T)
\\ \nonumber
&+ \gamma_5[\mathcal{Q}_s^2(T) \mathcal{Q}_p(T)+\mathcal{Q}_s(T) \mathcal{Q}_p^2(T)]
\\ \nonumber
&
= \bigl \{ [\mathcal{F}_5^R(s,T)- \mathcal{F}_5^L(p,T)]p- \bar{s}[\mathcal{F}_5^R(s,T)- \mathcal{F}_5^L(p,T)]\} (p^2-2s_0 p+|s|^2)^{-1}.
\end{eqnarray}
}
where $\gamma_5$ is given by (\ref{const}) for $n=5$.
\end{lemma}
\begin{proof}
Let us start by left multiplying the S-resolvent equation \eqref{reso} by $ \gamma_5 \mathcal{Q}_s^2(T)$ so that we get
{\small
\begin{eqnarray}
\nonumber
\label{eq4}
\mathcal{F}_5^R(s,T)S^{-1}_L(p,T)&= \bigl\{[\mathcal{F}_5^R(s,T)-\gamma_5 \mathcal{Q}_s^2(T) S^{-1}_L(p,T)]p-\bar{s}[\mathcal{F}_5^R(s,T)-\gamma_5 \mathcal{Q}_s^2(T)S^{-1}_L(p,T)]\bigl\} \cdot\\
& \cdot (p^2-2s_0p+|s|^2)^{-1}.
\end{eqnarray}
}
Now, we multiply the S-resolvent equation on the right by $ \gamma_5 \mathcal{Q}_p^2(T)$ and we obtain
{\small
\begin{eqnarray}
\nonumber
\label{eq5}
S^{-1}_R(s,T)\mathcal{F}_{5}^L(p,T)&= \bigl\{[\gamma_5 S^{-1}_R(s,T)\mathcal{Q}^2_p(T)-\mathcal{F}_{5}^L(p,T)]p-\bar{s}[\gamma_5S^{-1}_R(s,T)\mathcal{Q}^2_p(T)-\mathcal{F}_{5}^L(p,T)]\bigl\}\cdot\\
&\cdot(p^2-2s_0 p+|s|^2)^{-1}.
\end{eqnarray}
}
We multiply the S-resolvent equation on the left by $  \mathcal{Q}_s(T)$ and on the right by $\mathcal{Q}_p(T)$, we get
{\small
\begin{eqnarray}
\label{eq6}
\mathcal{Q}_s(T) S^{-1}_R(s,T)  S^{-1}_L(p,T) \mathcal{Q}_p(T) &=\bigl \{[
 \mathcal{Q}_s(T)S^{-1}_R(s,T) \mathcal{Q}_p(T)- \mathcal{Q}_s(T)S^{-1}_L(p,T)\mathcal{Q}_p(T)]p\\ \nonumber
& - \bar{s}[\mathcal{Q}_s(T)S^{-1}_R(s,T) \mathcal{Q}_p(T)- \mathcal{Q}_s(T)S^{-1}_L(p,T)\mathcal{Q}_p(T)] \bigl\}\cdot \\ \nonumber
&\cdot (p^2-2s_0 p+|s|^2)^{-1}.
\end{eqnarray}
}
Now, we sum \eqref{eq4}, \eqref{eq5} and \eqref{eq6}  multiplied by $ \gamma_5$ to get
{\small
\begingroup\allowdisplaybreaks\begin{align}
& \mathcal{F}_5^R(s,T)S^{-1}_L(p,T)+S^{-1}_R(s,T)\mathcal{F}_{5}^L(p,T)+\gamma_5 \mathcal{Q}_s(T)S^{-1}_R(s,T)S^{-1}_L(s,T)\mathcal{Q}_p(T)
\\ \nonumber
&= \bigl \{ [\mathcal{F}_5^R(s,T)- \mathcal{F}_5^L(p,T)]p- \bar{s}[\mathcal{F}_5^R(s,T)- \mathcal{F}_5^L(p,T)]\}(p^2-2s_0 p+|s|^2)^{-1}
\\ \nonumber
& + \bigl\{[\gamma_5 S^{-1}_R(s,T) \mathcal{Q}_p^2(T)-\gamma_5 \mathcal{Q}_s(T)S^{-1}_L(p,T) \mathcal{Q}_p(T)- \gamma_5 \mathcal{Q}_s^2(T) S^{-1}_{L}(p,T)
\\ \nonumber
&+\gamma_5 \mathcal{Q}_s(T)S^{-1}_R(s,T) \mathcal{Q}_p(T)]p - \bar{s}[\gamma_5 S^{-1}_R(s,T) \mathcal{Q}_p^2(T)-\gamma_5 \mathcal{Q}_s(T)S^{-1}_L(p,T) \mathcal{Q}_p(T)
\\ \nonumber
&-\gamma_5 \mathcal{Q}_s^2(T) S^{-1}_{L}(p,T)+\gamma_5 \mathcal{Q}_s(T)S^{-1}_R(s,T) \mathcal{Q}_p(T)]         \bigl\}(p^2-2s_0 p+|s|^2)^{-1}.
\end{align}\endgroup

}
Finally, we verify that
{\small

\begingroup\allowdisplaybreaks\begin{align}
& \bigl\{[\gamma_5 S^{-1}_R(s,T) \mathcal{Q}_p^2(T)-\gamma_5 \mathcal{Q}_s(T)S^{-1}_L(p,T) \mathcal{Q}_p(T)- \gamma_5 \mathcal{Q}_s^2(T) S^{-1}_{L}(p,T)+\gamma_5 \mathcal{Q}_s(T)S^{-1}_R(s,T) \mathcal{Q}_p(T)]p
\\ \nonumber
& - \bar{s}[\gamma_5 S^{-1}_R(s,T) \mathcal{Q}_p^2(T)-\gamma_5 \mathcal{Q}_s(T)S^{-1}_L(p,T) \mathcal{Q}_p(T)- \gamma_5 \mathcal{Q}_s^2(T) S^{-1}_{L}(p,T)+\gamma_5 \mathcal{Q}_s(T)S^{-1}_R(s,T) \mathcal{Q}_p(T)]        \bigl\}\cdot\\ \nonumber
&\cdot(p^2-2s_0 p+|s|^2)^{-1}=-\gamma_5[\mathcal{Q}_s^2(T)\mathcal{Q}_p(T)+ \mathcal{Q}_s(T)\mathcal{Q}^2_p(T)].
\end{align}\endgroup

}
Now observe that by the definitions of the S-resolvent operators we have
{\small
 \begingroup\allowdisplaybreaks\begin{align}
& \gamma_5\bigl(S^{-1}_R(s,T) \mathcal{Q}_p^2(T)- \mathcal{Q}_s(T)S^{-1}_L(p,T) \mathcal{Q}_p(T)-  \mathcal{Q}_s^2(T) S^{-1}_{L}(p,T)+ \mathcal{Q}_s(T)S^{-1}_R(s,T)\mathcal{Q}_p(T)\bigl)\\ \nonumber
& = \gamma_5\biggl(\mathcal{Q}_s(T)(s \mathcal{I}- \overline{T}) \mathcal{Q}_p^2(T)-\mathcal{Q}_s(T)(p \mathcal{I}- \overline{T}) \mathcal{Q}_p^{2}(T)- \mathcal{Q}_s^2(T)(p \mathcal{I}- \overline{T}) \mathcal{Q}_{p}(T)+\\ \nonumber
&+\mathcal{Q}_s^2(T)(s \mathcal{I}- \overline{T}) \mathcal{Q}_p(T) \biggl) = \gamma_5[\mathcal{Q}_s(T)(s-p)\mathcal{Q}_p^2(T)+ \mathcal{Q}_s^2(T)(s-p) \mathcal{Q}_p(T)],
\end{align}\endgroup
}
and this implies that
{\small

\begingroup\allowdisplaybreaks\begin{align}
& \bigl\{[\gamma_5 S^{-1}_R(s,T) \mathcal{Q}_p^2(T)-\gamma_5 \mathcal{Q}_s(T)S^{-1}_L(p,T)\mathcal{Q}_p- \gamma_5 \mathcal{Q}_s^2(T) S^{-1}_{L}(p,T)
\\ \nonumber
&+\gamma_5 \mathcal{Q}_s(T)S^{-1}_R(s,T) \mathcal{Q}_p(T)]p - \bar{s}[\gamma_5 S^{-1}_R(s,T) \mathcal{Q}_p^2(T)-\gamma_5 \mathcal{Q}_s(T)S^{-1}_L(p,T) \mathcal{Q}_p(T)+\gamma_5 \mathcal{Q}_s^2(T) S^{-1}_{L}(p,T)
\\ \nonumber
&+\gamma_5 \mathcal{Q}_s(T)S^{-1}_R(s,T) \mathcal{Q}_p(T)] \bigl\}(p^2-2s_0 p+|s|^2)^{-1}\\ \nonumber
& = \gamma_5 \bigl\{[\mathcal{Q}_s(T)(s-p)\mathcal{Q}_p^2(T)+ \mathcal{Q}_s^2(T)(s-p) \mathcal{Q}_p(T)]p- \bar{s}[\mathcal{Q}_s(T)(s-p)\mathcal{Q}_p^2(T)+ \mathcal{Q}_s^2(T)(s-p) \mathcal{Q}_p(T)] \bigl\}\cdot\\ \nonumber
& \cdot (p^2-2s_0 p+|s|^2)^{-1}\\ \nonumber
& = \gamma_5 \bigl\{[\mathcal{Q}_s(T)(sp-p^2)\mathcal{Q}_p^2(T)+ \mathcal{Q}_s^2(T)(sp-p^2) \mathcal{Q}_p(T)]- [\mathcal{Q}_s(T)(\bar{s}s-\bar{s}p)\mathcal{Q}_p^2(T) + \mathcal{Q}_s^2(T)(\bar{s}s-\bar{s}p) \mathcal{Q}_p(T)] \bigl\}\cdot\\ \nonumber
&\cdot(p^2-2s_0 p+|s|^2)^{-1}\\ \nonumber
& =\gamma_5[\mathcal{Q}_s^2(T)(sp-p^2- \bar{s}s+ \bar{s} p) \mathcal{Q}_p(T)+ \mathcal{Q}_s(T)(sp-p^2-s \bar{s}+ \bar{s} p) \mathcal{Q}_p^2(T)](p^2-2s_0 p+|s|^2)^{-1}\\ \nonumber
& =- \gamma_5[\mathcal{Q}_s^2(T)(p^2-2s_0 p+|s|^2)\mathcal{Q}_p(T)+ \mathcal{Q}_s(T)(p^2-2s_0 p+|s|^2)\mathcal{Q}_p^2(T)](p^2-2s_0 p+|s|^2)^{-1}\\ \nonumber
& =- \gamma_5[\mathcal{Q}_s^2(T) \mathcal{Q}_p(T)+ \mathcal{Q}_s(T) \mathcal{Q}_p^2(T)].
\end{align}\endgroup

}
\end{proof}
Using the above preliminary lemma and the relations between the pseuso $S$-resolvent operators and the $\mathcal{F}$-resolvent operators we obtain for $n=5$ the $\mathcal{F}$- resolvent equation.
This equation has strong similarities with the case $n=3$.
\begin{theorem}[The $\mathcal{F}$- resolvent equation for $n=5$]
\label{res3}
Let $T \in \mathcal{BC}^{0,1}(V_5)$. Then, for $p,s \in \rho_{\mathcal{F}}(T)$, the following equation holds
{\small

\begingroup\allowdisplaybreaks\begin{align}
&               \mathcal{F}_5^R(s,T)S^{-1}_L(p,T)+S^{-1}_R(s,T)\mathcal{F}_{5}^L(p,T)+\gamma_5^{-1} \bigl( s^{2} \mathcal{F}_{5}^R(s,T)\mathcal{F}_{5}^L(p,T) p^2- 3s^{2} \mathcal{F}_{5}^R(s,T)T\mathcal{F}_{5}^L(p,T) p  \\ \nonumber
&
-3s \mathcal{F}_{5}^R(s,T)T\mathcal{F}_{5}^L(p,T) p^2 +3s\mathcal{F}_{5}^R(s,T) T^2\mathcal{F}_{5}^L(p,T)p -2s\mathcal{F}_{5}^R(s,T)|T|^2T\mathcal{F}_{5}^L(p,T)  \\ \nonumber
&
+2s\mathcal{F}_{5}^R(s,T)|T|^2\mathcal{F}_{5}^L(p,T)p-2\mathcal{F}_{5}^R(s,T)|T|^2T\mathcal{F}_{5}^L(p,T)p+s \mathcal{F}_{5}^R(s,T) \overline{T}^2\mathcal{F}_{5}^L(p,T) p- s \mathcal{F}_{5}^R(s,T)|T|^2 \overline{T}\mathcal{F}_{5}^L(p,T)\\ \nonumber
&
- \mathcal{F}_{5}^R(s,T)|T|^2 \overline{T}\mathcal{F}_{5}^L(p,T)p+\mathcal{F}_{5}^R(s,T)|T|^4\mathcal{F}_{5}^L(p,T)+s \mathcal{F}_5^R(s,T) \mathcal{F}_5^L(p,T)p^3-\mathcal{F}_5^R(s,T) T\mathcal{F}_5^L(p,T)p^3\\ \nonumber
&
+2 \mathcal{F}_5^R(s,T) T^2\mathcal{F}_5^L(p,T)p^2 -\mathcal{F}_5^R(s,T) T^3 \mathcal{F}_5^L(p,T)p+2\mathcal{F}_5^R(s,T) T^2|T|^2 \mathcal{F}_5^L(p,T)+s^3\mathcal{F}_5^R(s,T) \mathcal{F}_5^L(p,T)p+\\ \nonumber
&
-s^3\mathcal{F}_5^R(s,T) T\mathcal{F}_5^L(p,T) +2s^2\mathcal{F}_5^R(s,T) T^2 \mathcal{F}_5^L(p,T)-s\mathcal{F}_5^R(s,T) T^3 \mathcal{F}_5^L(p,T)\bigl)
\\ \nonumber
&= \bigl \{ [\mathcal{F}_5^R(s,T)- \mathcal{F}_5^L(p,T)]p
- \bar{s}[\mathcal{F}_5^R(s,T)- \mathcal{F}_5^L(p,T)]\}  (p^2-2s_0 p+|s|^2)^{-1},
\end{align}\endgroup

}
where we have set for the sake of simplicity
$$
|T|^2=\overline{T}T.
$$
\end{theorem}
\begin{proof}
Firstly we remark that
{\small
 \begingroup\allowdisplaybreaks\begin{align}
& \gamma_5 \mathcal{Q}_s(T)S_R^{-1}(s,T)S_L^{-1}(p,T)\mathcal{Q}_p(T)=\gamma_5 \mathcal{Q}_s^2(T)(s \mathcal{I}- \overline{T})(p \mathcal{I}- \overline{T})\mathcal{Q}_p^2(T)\\ \nonumber
&=\gamma_5 \mathcal{Q}_s^2(T) \left(sp \mathcal{I}-s \overline{T}- \overline{T}p+ \overline{T}^2 \right) \mathcal{Q}_p^2(T)\\ \nonumber
&= \gamma_5 \left[s\mathcal{Q}_s^2(T)\mathcal{Q}_p^2(T)p-s\mathcal{Q}_s^2(T) \overline{T}\mathcal{Q}_p^2(T)-\mathcal{Q}_s^2(T) \overline{T}\mathcal{Q}_p^2(T)p+\mathcal{Q}_s^2(T) \overline{T}^2\mathcal{Q}_p^2(T) \right].
\end{align}\endgroup
}
Putting this in \eqref{F5} we deduce that
{\small

\begingroup\allowdisplaybreaks\begin{align}
& \mathcal{F}_5^R(s,T)S^{-1}_L(p,T)+S^{-1}_R(s,T)\mathcal{F}_{5}^L(p,T)+\gamma_5 \bigl[s\mathcal{Q}_s^2(T)\mathcal{Q}_p^2(T)p-s\mathcal{Q}_s^2(T) \overline{T}\mathcal{Q}_p^2(T)+\\
\nonumber
&-\mathcal{Q}_s^2(T) \overline{T}\mathcal{Q}_p^2(T)p+\mathcal{Q}_s^2(T) \overline{T}^2\mathcal{Q}_p^2(T) \bigl]+ \gamma_5[\mathcal{Q}_s^2(T) \mathcal{Q}_p(T)+\mathcal{Q}_s(T) \mathcal{Q}_p^2(T)]\\ \nonumber
& = \bigl \{ [\mathcal{F}_5^R(s,T)- \mathcal{F}_5^L(p,T)]p- \bar{s}[\mathcal{F}_5^R(s,T)- \mathcal{F}_5^L(p,T)]\} (p^2-2s_0 p+|s|^2)^{-1}.
\end{align}\endgroup

}
Now, we use the right and left $ \mathcal{F}$-resolvent equation for $n=5$
(see Theorem \ref{LRFRESEQ}) namely
\begin{equation}
\label{use1}
\mathcal{F}_{5}^L(p,T)p-T \mathcal{F}_{5}^L(p,T)= \gamma_5 Q_p^{2}(T),
\end{equation}
and
\begin{equation}
\label{use2}
s \mathcal{F}_{5}^R(s,T)- \mathcal{F}_{5}^R(s,T)T= \gamma_5 Q_s^{2}(T).
\end{equation}
We go through the computations terms by terms
{\small

\begingroup\allowdisplaybreaks\begin{align}
\label{eq9}
s \mathcal{Q}_s^2(T) \mathcal{Q}_p^2(T) p&= \gamma_5^{-2} s \left(s \mathcal{F}_{5}^R(s,T)- \mathcal{F}_{5}^R(s,T)T\right) \left(\mathcal{F}_{5}^L(p,T)p-T \mathcal{F}_{5}^L(p,T)\right)p\\ \nonumber
&=  \gamma_5^{-2} \bigl( s^{2} \mathcal{F}_{5}^R(s,T)\mathcal{F}_{5}^L(p,T) p^2- s^{2} \mathcal{F}_{5}^R(s,T)T\mathcal{F}_{5}^L(p,T) p  \\ \nonumber
&-s \mathcal{F}_{5}^R(s,T)T\mathcal{F}_{5}^L(p,T) p^2 +s\mathcal{F}_{5}^R(s,T)T^2 \mathcal{F}_{5}^L(p,T)p  \bigl),
\end{align}\endgroup

}

{\small

\begingroup\allowdisplaybreaks\begin{align}
\label{eq10}
\nonumber
s \mathcal{Q}_s^2(T) \overline{T} \mathcal{Q}_p^2(T) &= \gamma_5^{-2} s \left(s \mathcal{F}_{5}^R(s,T)- \mathcal{F}_{5}^R(s,T)T\right) \overline{T} \left(\mathcal{F}_{5}^L(p,T)p-T \mathcal{F}_{5}^L(p,T)\right)\\ \nonumber
&=  \gamma_5^{-2} \bigl( s^{2} \mathcal{F}_{5}^R(s,T) \overline{T}\mathcal{F}_{5}^L(p,T) p- s^{2} \mathcal{F}_{5}^R(s,T)|T|^2\mathcal{F}_{5}^L(p,T)
 \\
&- s \mathcal{F}_{5}^R(s,T)|T|^2\mathcal{F}_{5}^L(p,T) p +s\mathcal{F}_{5}^R(s,T)|T|^2T\mathcal{F}_{5}^L(p,T)  \bigl),
\end{align}\endgroup
}
{\small
\begingroup\allowdisplaybreaks\begin{align}
\nonumber
\label{eq11}
\mathcal{Q}_s^2(T) \overline{T} \mathcal{Q}_p^2(T)p&=  \gamma_5^{-2} \bigl( s \mathcal{F}_{5}^R(s,T) \overline{T}\mathcal{F}_{5}^L(p,T) p^2- s \mathcal{F}_{5}^R(s,T)|T|^2\mathcal{F}_{5}^L(p,T)p
\\
&- \mathcal{F}_{5}^R(s,T)|T|^2\mathcal{F}_{5}^L(p,T) p^2 +\mathcal{F}_{5}^R(s,T)|T|^2T\mathcal{F}_{5}^L(p,T)p  \bigl),
\end{align}\endgroup
}

{\small
\begingroup\allowdisplaybreaks\begin{align}
\nonumber
\label{eq12}
\mathcal{Q}_s^2(T) \overline{T}^2 \mathcal{Q}_p^2(T)&=  \gamma_5^{-2} \bigl( s \mathcal{F}_{5}^R(s,T) \overline{T}^2\mathcal{F}_{5}^L(p,T) p- s \mathcal{F}_{5}^R(s,T)|T|^2 \overline{T}\mathcal{F}_{5}^L(p,T) \\
&- \mathcal{F}_{5}^R(s,T)|T|^2 \overline{T}\mathcal{F}_{5}^L(p,T)p +\mathcal{F}_{5}^R(s,T)|T|^4\mathcal{F}_{5}^L(p,T)  \bigl).
\end{align}\endgroup
}

Moreover by \eqref{use1} and \eqref{use2} we have
{\small
\begingroup\allowdisplaybreaks\begin{align}
\mathcal{Q}_{p}(T) &= \mathcal{Q}_{p}(T)^2 \mathcal{Q}_p(T)^{-1}= \gamma_{5}^{-1} \left(\mathcal{F}_5^L(p,T)p-T \mathcal{F}_5^L(p,T) \right)\left( p^2-(T+ \overline{T})p+|T|^2 \right)
\\ \nonumber
&= \gamma_5^{-1} \biggl( \mathcal{F}_5^{L}(p,T) p^3-2T \mathcal{F}_5^{L}(p,T) p^2- \overline{T}\mathcal{F}_5^{L}(p,T)p^2+T^2 \mathcal{F}_5^{L}(p,T) p + 2|T|^2 \mathcal{F}_5^{L}(p,T) p
\\ \nonumber
& -T|T|^2 \mathcal{F}_5^{L}(p,T) \biggl),
\end{align}\endgroup

}
{\small

\begingroup\allowdisplaybreaks\begin{align}
\mathcal{Q}_s(T) &= \mathcal{Q}_s(T)^{-1} \mathcal{Q}_s(T)^2= \gamma_5^{-1} \left(s^2-s(T+ \overline{T})+|T|^2 \right) \left(s \mathcal{F}_5^R(s,T)- \mathcal{F}_5^R(s,T)T \right)
\\ \nonumber
&= \gamma_5^{-1}\biggl( s^3 \mathcal{F}_5^R(s,T)-2s^2 \mathcal{F}_5^R(s,T)T-s^2 \mathcal{F}_5^R(s,T) \overline{T}+s \mathcal{F}_5^R(s,T) T^2
+2s \mathcal{F}_5^R(s,T)|T|^2\\ \nonumber
& -\mathcal{F}_5^R(s,T) T |T|^2 \biggl).
\end{align}\endgroup

}

This implies that
{\small
\begingroup\allowdisplaybreaks\begin{align}
\label{eq13}
\nonumber
\mathcal{Q}_s^2(T) \mathcal{Q}_p(T) &= \gamma_5^{-2}\left(s \mathcal{F}_{5}^R(s,T) - \mathcal{F}_{5}^R(s,T)T\right) \bigl( \mathcal{F}_5^{L}(p,T) p^3-2T \mathcal{F}_5^{L}(p,T) p^2- \overline{T}\mathcal{F}_5^{L}(p,T)p^2 \\
\nonumber
& +T^2 \mathcal{F}_5^{L}(p,T) p +2|T|^2 \mathcal{F}_5^{L}(p,T) p-T|T|^2 \mathcal{F}_5^{L}(p,T) \bigl)\\
\nonumber
&= \gamma_5^{-2} \bigl( s  \mathcal{F}_5^{R}(s,T)\mathcal{F}_5^{L}(p,T)p^3
-2s\mathcal{F}_5^{R}(s,T)T\mathcal{F}_5^{L}(p,T)p^2
-s\mathcal{F}_5^{R}(s,T)\overline{T}\mathcal{F}_5^{L}(p,T)p^2\\
\nonumber
&+s\mathcal{F}_5^{R}(s,T)T^2\mathcal{F}_5^{L}(p,T)p+2s \mathcal{F}_5^{R}(s,T)|T|^2 \mathcal{F}_5^{L}(p,T)p-s\mathcal{F}_5^{R}(s,T)T|T|^2\mathcal{F}_5^{L}(p,T)\\
\nonumber
&-\mathcal{F}_5^{R}(s,T)T\mathcal{F}_5^{L}(p,T)p^3
+2\mathcal{F}_5^{R}(s,T)T^2\mathcal{F}_5^{L}(p,T)p^2
+\mathcal{F}_5^{R}(s,T)|T|^2\mathcal{F}_5^{L}(p,T)p^2\\
&-\mathcal{F}_5^{R}(s,T)T^3\mathcal{F}_5^{L}(p,T)p
-2\mathcal{F}_5^{R}(s,T)|T|^2T\mathcal{F}_5^{L}(p,T)p
+\mathcal{F}_5^{R}(s,T)T^2|T|^2\mathcal{F}_5^{L}(p,T),
\end{align}\endgroup

}
and
{\small

\begingroup\allowdisplaybreaks\begin{align}
\label{eq14}
\nonumber
\mathcal{Q}_s(T)\mathcal{Q}_p^2(T)&= \gamma_{5}^{-2} \bigl(s^3 \mathcal{F}_5^R(s,T)-2s^2 \mathcal{F}_5^R(s,T)T-s^2 \mathcal{F}_5^R(s,T) \overline{T}+s \mathcal{F}_5^R(s,T) T^2+2s \mathcal{F}_5^R(s,T)|T|^2
\\
\nonumber
& -\mathcal{F}_5^R(s,T) T |T|^2 \bigl) \bigl( \mathcal{F}_5^{L}(p,T)p-T \mathcal{F}_5^L(p,T) \bigl)\\
\nonumber
&= \gamma_5^{-2}\bigl(s^3 \mathcal{F}_5^R(s,T) \mathcal{F}_5^L(p,T) p-2s^2\mathcal{F}_5^R(s,T)T \mathcal{F}_5^L(p,T)p-s^2\mathcal{F}_5^R(s,T) \overline{T} \mathcal{F}_5^L(p,T)p\\
\nonumber
&+s\mathcal{F}_5^R(s,T)T^2 \mathcal{F}_5^L(p,T)p+2s \mathcal{F}_5^R(s,T) |T|^2 \mathcal{F}_5^L(p,T)p -\mathcal{F}_5^R(s,T) T |T|^2 \mathcal{F}_5^L(p,T)p\\
\nonumber
&-s^3 \mathcal{F}_5^R(s,T)T \mathcal{F}_5^L(p,T)+2s^2\mathcal{F}_5^R(s,T)T^2 \mathcal{F}_5^L(p,T)+s^2 \mathcal{F}_5^R(s,T) |T|^2 \mathcal{F}_5^L(p,T)\\
&-s\mathcal{F}_5^R(s,T) T^3 \mathcal{F}_5^L(p,T)-2s \mathcal{F}_5^R(s,T)|T|^2 T \mathcal{F}_5^L(p,T)+\mathcal{F}_5^R(s,T)T^2 |T|^2\mathcal{F}_5^L(p,T) \bigl).
\end{align}\endgroup
}
The statement is obtained by making the sum of \eqref{eq9}, \eqref{eq10}, \eqref{eq11}, \eqref{eq12}, \eqref{eq13}, \eqref{eq14}.
\end{proof}

\subsection{\bf The $\mathcal{F}$-resolvent equation for $n=7$}

As it is clearly visible from the case $n=5$, that there are
  intrinsic complications in the structure of the $\mathcal{F}$-resolvent equation.
  The case $n=7$ shows that it is not possible to have a reasonable
  closed form for the $\mathcal{F}$-resolvent equation just in terms of the $S$-resolvent operators
  and of the $\mathcal{F}$-resolvent operators. Instead, the use of the pseudo $S$-resolvent operators allows a reasonable structure of the resolvent equation.
 The point of the matter is that in this case it is not possible to obtain a $\mathcal{F}$-resolvent equation as a relation between $ \mathcal{F}_7^R(s,T) \mathcal{F}_7^L(p,T)$, and $\mathcal{F}_7^R(s,T) -\mathcal{F}_7^L(p,T)$.
 However, it is possible to prove a form of the
 $\mathcal{F}$- resolvent equation which will
  be fundamental in the section on the Riesz projects.
As before, we begin with a technical result.
\begin{lemma}[The $ \mathcal{F}$- resolvent equation for $n=7$ with the pseudo $S$-resolvent operators]
\label{res7}
Let $T \in \mathcal{BC}^{0,1}(V_7)$. Then for $p,s \in \rho_{\mathcal{F}}(T)$ the following equation holds
{\small
\begin{eqnarray}
\label{n70}
& \mathcal{F}_7^R(s,T)S^{-1}_L(p,T)+S^{-1}_R(s,T)\mathcal{F}_7^L(p,T)+ \gamma_7 \bigl[ \mathcal{Q}_s(T)S^{-1}_R(s,T)S^{-1}_L(p,T)\mathcal{Q}_p^2(T)\\
\nonumber
&+\mathcal{Q}_s^2(T)S^{-1}_R(s,T)S^{-1}_L(p,T)\mathcal{Q}_p(T)+\mathcal{Q}_s(T)\mathcal{Q}_p^3(T)+ \mathcal{Q}_s^3(T) \mathcal{Q}_p(T)+ \mathcal{Q}_s^2(T) \mathcal{Q}_p^2(T)]\\
\nonumber
& = \bigl \{ \bigl[\mathcal{F}_7^R(s,T)-\mathcal{F}_7^L(p,T) \bigl]p- \bar{s}\bigl[\mathcal{F}_7^R(s,T)-\mathcal{F}_7^L(p,T) \bigl] \bigl \} (p^2-2s_0p+|s|^2)^{-1}.
\end{eqnarray}
}
\end{lemma}
\begin{proof}
First of all, we left multiply the S-resolvent equation \eqref{reso} by $ \gamma_7 \mathcal{Q}_s^3(T)$, so that we get
{\small
\begin{eqnarray}
\nonumber
\label{n71}
\mathcal{F}_7^R(s,T)S^{-1}_L(p,T)&= \left\{[\mathcal{F}_7^R(s,T)-\gamma_7 \mathcal{Q}_s^3(T)S^{-1}_L(p,T)]p- \bar{s}[\mathcal{F}_7^R(s,T)-\gamma_7 \mathcal{Q}_s^3(T)S^{-1}_L(p,T)]\right\}\cdot\\
& \cdot (p^2-2s_0p+|s|^2)^{-1}.
\end{eqnarray}
}
Now, we right multiply the S-resolvent equation \eqref{reso} by $ \gamma_7 \mathcal{Q}_p^3(T)$
{\small
\begin{eqnarray}
\nonumber
\label{n72}
S^{-1}_R(s,T)\mathcal{F}_7^L(p,T)&= \left\{[\gamma_7 S^{-1}_R(s,T)\mathcal{Q}_p^3(T)-\mathcal{F}_7^L(p,T)]p- \bar{s}[\gamma_7  S^{-1}_R(s,T) \mathcal{Q}_p^3(T)-\mathcal{F}_7^L(p,T)]\right\}\cdot\\
& \cdot (p^2-2s_0p+|s|^2)^{-1}.
\end{eqnarray}
}
Then we multiply the S-resolvent equation on the left by $ \mathcal{Q}_s(T)$ and on the right by $ \mathcal{Q}_p^2(T)$
{\small
\begin{eqnarray}
\label{n73}
&\mathcal{Q}_s(T)S^{-1}_R(s,T)S^{-1}_L(p,T)\mathcal{Q}_p^2(T)
\\ \nonumber
&= \bigl\{[ \mathcal{Q}_s(T)S^{-1}_R(s,T)\mathcal{Q}_p^2(T)- \mathcal{Q}_s(T)S^{-1}_L(p,T)\mathcal{Q}_p^2(T)]p \\ \nonumber
&-\bar{s}[ \mathcal{Q}_s(T)S^{-1}_R(s,T)\mathcal{Q}_p^2(T)- \mathcal{Q}_s(T)S^{-1}_L(p,T)\mathcal{Q}_p^2(T)]\bigl\} (p^2-2s_0p+|s|^2)^{-1}.
\end{eqnarray}
}
We now multiply the S-resolvent equation on the left by $ \mathcal{Q}_s^2(T)$ and on the right by $ \mathcal{Q}_p(T)$
{\small

\begingroup\allowdisplaybreaks\begin{align}
\label{n74}
&\mathcal{Q}_s^2(T)S^{-1}_R(s,T)S^{-1}_L(p,T)\mathcal{Q}_p(T)
\\ \nonumber
&= \bigl\{[ \mathcal{Q}_s^2(T)S^{-1}_R(s,T)\mathcal{Q}_p(T)- \mathcal{Q}_s^2(T)S^{-1}_L(p,T)\mathcal{Q}_p(T)]p \\ \nonumber
&-\bar{s}[ \mathcal{Q}_s^2(T)S^{-1}_R(s,T)\mathcal{Q}_p(T)- \mathcal{Q}_s^2(T)S^{-1}_L(p,T)\mathcal{Q}_p(T)]\bigl\} (p^2-2s_0p+|s|^2)^{-1}.
\end{align}\endgroup

}
We sum \eqref{n71}, \eqref{n72} and \eqref{n73}, \eqref{n74} multiplied by $ \gamma_7$, and we obtain
{\small

\begingroup\allowdisplaybreaks\begin{align}
& \mathcal{F}_7^R(s,T)S^{-1}_L(p,T)+S^{-1}_R(s,T)\mathcal{F}_7^L(p,T)+ \gamma_7\mathcal{Q}_s(T)S^{-1}_R(s,T)S^{-1}_L(p,T)\mathcal{Q}_p^2(T)
\\ \nonumber
&+\gamma_7 \mathcal{Q}_s^2(T)S^{-1}_R(s,T)S^{-1}_L(p,T)\mathcal{Q}_p(T)= \bigl \{ \bigl[\mathcal{F}_7^R(s,T)-\mathcal{F}_7^L(p,T)+\gamma_7 S^{-1}_R(s,T)\mathcal{Q}_p^3(T)
\\ \nonumber
& -\gamma_7 \mathcal{Q}_s^3(T)S^{-1}_L(p,T)+ \gamma_7 \mathcal{Q}_s(T)S^{-1}_R(s,T)\mathcal{Q}_p^2(T)- \gamma_7\mathcal{Q}_s(T)S^{-1}_L(p,T)\mathcal{Q}_p^2(T)  \\ \nonumber
&+\gamma_7 \mathcal{Q}_s^2(T)S^{-1}_R(s,T)\mathcal{Q}_p(T)-\gamma_7\mathcal{Q}_s^2(T)S^{-1}_L(p,T)\mathcal{Q}_p(T) \bigl] p- \bar{s}\bigl[\mathcal{F}_7^R(s,T)-\mathcal{F}_7^L(p,T)\\ \nonumber
& +\gamma_7 S^{-1}_R(s,T)\mathcal{Q}_p^3(T)-\gamma_7 \mathcal{Q}_s^3(T)S^{-1}_L(p,T)+ \gamma_7 \mathcal{Q}_s(T)S^{-1}_R(s,T)\mathcal{Q}_p^2(T)
  \\ \nonumber
&- \gamma_7\mathcal{Q}_s(T)S^{-1}_L(p,T)\mathcal{Q}_p^2(T)+\gamma_7 \mathcal{Q}_s^2(T)S^{-1}_R(s,T)\mathcal{Q}_p(T)-\gamma_7\mathcal{Q}_s^2(T)S^{-1}_L(p,T)\mathcal{Q}_p(T) \bigl](p^2-2s_0p+|s|^2)^{-1}.
\end{align}\endgroup

}
By the definition of $S$- resolvent operators we have
{\small

\begingroup\allowdisplaybreaks\begin{align}
& S^{-1}_R(s,T)\mathcal{Q}_p^3(T) -\mathcal{Q}_s^3(T)S^{-1}_L(p,T)+  \mathcal{Q}_s(T)S^{-1}_R(s,T)\mathcal{Q}_p^2(T)- \mathcal{Q}_s(T)S^{-1}_L(p,T)\mathcal{Q}_p^2(T)
\\ \nonumber
&+\mathcal{Q}_s^2(T)S^{-1}_R(s,T)\mathcal{Q}_p(T)-\mathcal{Q}_s^2(T)S^{-1}_L(p,T)\mathcal{Q}_p(T)\\ \nonumber
& = \mathcal{Q}_s(T)(s \mathcal{I}- \overline{T}) \mathcal{Q}_p^3(T)- \mathcal{Q}_s^3(T)(p \mathcal{I}- \overline{T}) \mathcal{Q}_p(T)+ \mathcal{Q}_s^{2}(T)(s \mathcal{I}- \overline{T}) \mathcal{Q}_p^2(T)- \mathcal{Q}_s(T)(p \mathcal{I}- \overline{T}) \mathcal{Q}_p^3(T)\\ \nonumber
&+ \mathcal{Q}_s^3(T)(s \mathcal{I}- \overline{T}) \mathcal{Q}_{p}(T)- \mathcal{Q}_s^2(T)(p \mathcal{I}- \overline{T})\mathcal{Q}_p^2(T)\\ \nonumber
&= \mathcal{Q}_s(T) s \mathcal{Q}_p^3(T)- \mathcal{Q}_s^3(T)p \mathcal{Q}_p(T)+ \mathcal{Q}_s^2(T)s \mathcal{Q}_p^2(T)-\mathcal{Q}_s(T) p \mathcal{Q}_p^3(T)+\mathcal{Q}_s^3(T)s \mathcal{Q}_p(T)- \mathcal{Q}_s^2(T)p \mathcal{Q}_p^2(T)\\ \nonumber
& = \mathcal{Q}_s(T)(s-p) \mathcal{Q}_p^3(T)+ \mathcal{Q}_s^3(T)(s-p) \mathcal{Q}_p(T)+ \mathcal{Q}_s^2(T)(s-p) \mathcal{Q}_p^2(T).
\end{align}\endgroup

}
Hence
{\small

\begingroup\allowdisplaybreaks\begin{align}
& \mathcal{F}_7^R(s,T)S^{-1}_L(p,T)+S^{-1}_R(s,T)\mathcal{F}_7^L(p,T)+ \gamma_7\mathcal{Q}_s(T)S^{-1}_R(s,T)S^{-1}_L(p,T)\mathcal{Q}_p^2(T)+\\ \nonumber
&+\gamma_7 \mathcal{Q}_s^2(T)S^{-1}_R(s,T)S^{-1}_L(p,T)\mathcal{Q}_p(T)= \bigl \{ \bigl[\mathcal{F}_7^R(s,T)-\mathcal{F}_7^L(p,T) \bigl]p- \bar{s}\bigl[\mathcal{F}_7^R(s,T)-\mathcal{F}_7^L(p,T) \bigl] \bigl \}\cdot\\ \nonumber
& \cdot (p^2-2s_0p+|s|^2)^{-1}+ \gamma_7 \bigl \{ \bigl[\mathcal{Q}_s(T)(s-p) \mathcal{Q}_p^3(T)+ \mathcal{Q}_s^3(T)(s-p) \mathcal{Q}_p(T)+ \mathcal{Q}_s^2(T)(s-p) \mathcal{Q}_p^2(T)\bigl]p
\\ \nonumber
&- \bar{s} \bigl[\mathcal{Q}_s(T)(s-p) \mathcal{Q}_p^3(T)+ \mathcal{Q}_s^3(T)(s-p) \mathcal{Q}_p(T)+ \mathcal{Q}_s^2(T)(s-p) \mathcal{Q}_p^2(T)\bigl]\bigl\} (p^2-2s_0p+|s|^2)^{-1}.
\end{align}\endgroup

}
Finally we have to verify that
{\small

\begingroup\allowdisplaybreaks\begin{align}
& \gamma_7 \bigl \{ \bigl[\mathcal{Q}_s(T)(s-p) \mathcal{Q}_p^3(T)+ \mathcal{Q}_s^3(T)(s-p) \mathcal{Q}_p(T)+ \mathcal{Q}_s^2(T)(s-p) \mathcal{Q}_p^2(T)\bigl]p+\\ \nonumber
&- \bar{s} \bigl[\mathcal{Q}_s(T)(s-p) \mathcal{Q}_p^3(T)+ \mathcal{Q}_s^3(T)(s-p) \mathcal{Q}_p(T)+ \mathcal{Q}_s^2(T)(s-p) \mathcal{Q}_p^2(T)\bigl]\bigl\} (p^2-2s_0p+|s|^2)^{-1}\\ \nonumber
&=- \gamma_7[\mathcal{Q}_s(T)\mathcal{Q}_p^3(T)+ \mathcal{Q}_s^3(T) \mathcal{Q}_p(T)+ \mathcal{Q}_s^2(T) \mathcal{Q}_p^2(T)].
\end{align}\endgroup

}
This follows from
{\small

\begingroup\allowdisplaybreaks\begin{align}
& \gamma_7 \bigl \{ \bigl[\mathcal{Q}_s(T)(s-p) \mathcal{Q}_p^3(T)+ \mathcal{Q}_s^3(T)(s-p) \mathcal{Q}_p(T)+ \mathcal{Q}_s^2(T)(s-p) \mathcal{Q}_p^2(T)\bigl]p
\\ \nonumber
&- \bar{s} \bigl[\mathcal{Q}_s(T)(s-p) \mathcal{Q}_p^3(T)+ \mathcal{Q}_s^3(T)(s-p) \mathcal{Q}_p(T)+ \mathcal{Q}_s^2(T)(s-p) \mathcal{Q}_p^2(T)\bigl]\bigl\} (p^2-2s_0p+|s|^2)^{-1}
\\ \nonumber
&= \gamma_7 \bigl[ \mathcal{Q}_s(T)(sp-p^2) \mathcal{Q}_p^3(T)+ \mathcal{Q}_s^3(T)(sp-p^2)\mathcal{Q}_p(T)+ \mathcal{Q}_s^2(T)(sp-p^2) \mathcal{Q}_p^2(T)- \mathcal{Q}_s(T)( \bar{s}s- \bar{s}p) \mathcal{Q}_p^3(T)
\\ \nonumber
&- \mathcal{Q}_s^3(T)( \bar{s}s- \bar{s}p) \mathcal{Q}_p(T)- \mathcal{Q}_s^2(T)( \bar{s}s- \bar{s}p) \mathcal{Q}_p^2(T) \bigl] (p^2-2s_0p+|s|^2)^{-1}\\ \nonumber
& = \gamma_7 \bigl[\mathcal{Q}_s(T)(sp-p^2- \bar{s}s+ \bar{s}p) \mathcal{Q}_p^3(T)+\mathcal{Q}_s^3(T)(sp-p^2- \bar{s}s+ \bar{s}p) \mathcal{Q}_p(T)
\\ \nonumber
&
+\mathcal{Q}_s^2(T)(sp-p^2- \bar{s}s+ \bar{s}p) \mathcal{Q}_p^2(T)] (p^2-2s_0p+|s|^2)^{-1}
\\ \nonumber
&= - \gamma_7[\mathcal{Q}_s(T)\mathcal{Q}_p^3(T)+ \mathcal{Q}_s^3(T) \mathcal{Q}_p(T)+ \mathcal{Q}_s^2(T) \mathcal{Q}_p^2(T)].
\end{align}\endgroup

}
\end{proof}
The results of the previous lemma allows us to obtain the so-called pseudo $\mathcal{F}$-resolvent equation for $n=7$.
\begin{theorem}[The pseudo $\mathcal{F}$-resolvent equation for $n=7$]
\label{pres1}
Let $T \in \mathcal{BC}^{0,1}(V_7)$. Then for $p,s \in \rho_{\mathcal{F}}(T)$ the following equation holds
{\small

\begingroup\allowdisplaybreaks\begin{align}
\label{n701}
&
\mathcal{F}_7^R(s,T)S^{-1}_L(p,T)+S^{-1}_R(s,T)\mathcal{F}_7^L(p,T)
\\ \nonumber
&
+\gamma_7^{-1}\bigl[\mathcal{Q}_s(T) S^{-1}_R(s,T) \mathcal{F}_7^L(p,T)p^2-\mathcal{Q}_s(T) S^{-1}_R(s,T)T \mathcal{F}_7^L(p,T)p\\
\nonumber
&             -\mathcal{Q}_s(T) S^{-1}_R(s,T) \overline{T} \mathcal{F}_7^L(p,T)p+\mathcal{Q}_s(T) S^{-1}_R(s,T)|T|^2 \mathcal{F}_7^L(p,T) +s^2 \mathcal{F}_7^R(s,T)S^{-1}_L(p,T) \mathcal{Q}_p^2(T)
\\ \nonumber
&             -s\mathcal{F}_7^R(s,T)TS^{-1}_L(p,T) \mathcal{Q}_p^2(T)-s \mathcal{F}_7^R(s,T) \overline{T}S^{-1}_L(p,T) \mathcal{Q}_p^2(T)
\\ \nonumber
&
+\mathcal{F}_7^R(s,T)|T|^2 S^{-1}_L(p,T) \mathcal{Q}_p^2(T)+\mathcal{Q}_s(T)\mathcal{F}_7^L(p,T)p \\ \nonumber
&             -\mathcal{Q}_s(T)T\mathcal{F}_7^L(p,T)+s\mathcal{F}_7^R(s,T)\mathcal{Q}_p(T)- \mathcal{F}_7^R(s,T)T\mathcal{Q}_p(T)\bigl]+ \gamma_7^{-2}\bigl[s^3\mathcal{F}_7^R(s,T)\mathcal{F}_7^L(p,T)p^3
\\ \nonumber
&
-s^3\mathcal{F}_7^R(s,T)T\mathcal{F}_7^L(p,T)p^2-s^2\mathcal{F}_7^R(s,T)T\mathcal{F}_7^L(p,T)p^3\\
\nonumber
& +s^2\mathcal{F}_7^R(s,T)T^2\mathcal{F}_7^L(p,T)p^2+ s^3 \mathcal{F}_7^R(s,T) \mathcal{F}_7^L(p,T)p[|T|^2-p(T+ \overline{T})]\\
\nonumber
& -s^3\mathcal{F}_7^R(s,T) T \mathcal{F}_7^L(p,T)[|T|^2-p(T+ \overline{T})] -s^2\mathcal{F}_7^R(s,T) T \mathcal{F}_7^L(p,T)p[|T|^2-p(T+ \overline{T})]\\
\nonumber
&+ s^2\mathcal{F}_7^R(s,T) T^2 \mathcal{F}_7^L(p,T)[|T|^2-p(T+ \overline{T})] +\left[|T|^2-s(T+ \overline{T})\right]s\mathcal{F}_7^R(s,T) \mathcal{F}_7^L(p,T) p^3\\
\nonumber
&-\left[|T|^2-s(T+ \overline{T})\right]s\mathcal{F}_7^R(s,T) T \mathcal{F}_7^L(p,T) p^2-\left[|T|^2-s(T+ \overline{T})\right]\mathcal{F}_7^R(s,T)T \mathcal{F}_7^L(p,T) p^3\\
\nonumber
&+\left[|T|^2-s(T+ \overline{T})\right]\mathcal{F}_7^R(s,T) T^2 \mathcal{F}_7^L(p,T)p^2 +\left[|T|^2-s(T+ \overline{T})\right]\bigl( s\mathcal{F}_7^R(s,T)\mathcal{F}_7^L(p,T)p
\\
\nonumber
&-s\mathcal{F}_7^R(s,T)T\mathcal{F}_7^L(p,T)
-\mathcal{F}_7^R(s,T)T\mathcal{F}_7^L(p,T)p+\mathcal{F}_7^R(s,T)T^2\mathcal{F}_7^L(p,T)\bigl)\left[|T|^2-p(T+ \overline{T})\right]\bigl\}
\\ \nonumber
&              = \bigl \{ \bigl[\mathcal{F}_7^R(s,T)-\mathcal{F}_7^L(p,T) \bigl]p- \bar{s}\bigl[\mathcal{F}_7^R(s,T)-\mathcal{F}_7^L(p,T) \bigl] \bigl \} (p^2-2s_0p+|s|^2)^{-1}.
\end{align}\endgroup

}
\end{theorem}
\begin{proof}
We use the $\mathcal{F}$-resolvent equations \eqref{eq2} and \eqref{eq3} with $n=7$
\begin{equation}
\label{n79b}
\mathcal{F}_7^L(p,T)p-T \mathcal{F}_7^L(p,T)= \gamma_7 \mathcal{Q}_p^3(T),
\end{equation}
\begin{equation}
\label{n710b}
s \mathcal{F}_7^R(s,T)-\mathcal{F}_7^R(s,T)T= \gamma_7 \mathcal{Q}_s^3(T).
\end{equation}
Now we will substitute \eqref{n79b} and \eqref{n710b} in the equation of Lemma \ref{res7}. We go terms by terms
{\small

\begingroup\allowdisplaybreaks\begin{align}
\label{n711b}
\mathcal{Q}_s(T) S^{-1}_R(s,T) S^{-1}_L(p,T) \mathcal{Q}_p^2(T)&=\mathcal{Q}_s(T) S^{-1}_R(s,T) (p \mathcal{I}- \overline{T})\mathcal{Q}_p^3(T)\\
\nonumber
&=\mathcal{Q}_s(T) S^{-1}_R(s,T) \mathcal{Q}_p^3(T)p-\mathcal{Q}_s(T) S^{-1}_R(s,T) \overline{T}\mathcal{Q}_p^3(T)\\
\nonumber
&=\gamma_7^{-1}\bigl\{\mathcal{Q}_s(T) S^{-1}_R(s,T) \left[\mathcal{F}_7^L(p,T)p-T \mathcal{F}_7^L(p,T)\right]p
\\
\nonumber
&-\mathcal{Q}_s(T) S^{-1}_R(s,T) \overline{T}\left[\mathcal{F}_7^L(p,T)p-T \mathcal{F}_7^L(p,T)\right] \bigl\}\\
\nonumber
&= \gamma_7^{-1}\bigl[\mathcal{Q}_s(T) S^{-1}_R(s,T) \mathcal{F}_7^L(p,T)p^2-\mathcal{Q}_s(T) S^{-1}_R(s,T)T \mathcal{F}_7^L(p,T)p\\ \nonumber
&-\mathcal{Q}_s(T) S^{-1}_R(s,T) \overline{T} \mathcal{F}_7^L(p,T)p+\mathcal{Q}_s(T) S^{-1}_R(s,T)|T|^2 \mathcal{F}_7^L(p,T) \bigl],
\end{align}\endgroup

}
{\small

\begingroup\allowdisplaybreaks\begin{align}
\label{n712b}
\mathcal{Q}_s^2(T) S^{-1}_R(s,T) S^{-1}_L(p,T) \mathcal{Q}_p(T)&=\mathcal{Q}_s^3(T)(s \mathcal{I}- \overline{T}) S^{-1}_L(p,T) \mathcal{Q}_p^2(T)\\
\nonumber
&=s\mathcal{Q}_s^3(T) S^{-1}_L(p,T) \mathcal{Q}_p^2(T)-\mathcal{Q}_s^3(T) \overline{T} S^{-1}_L(p,T) \mathcal{Q}_p^2(T)\\
\nonumber
&=\gamma_{7}^{-1} \bigl\{s \left[s \mathcal{F}_7^R(s,T)-\mathcal{F}_7^R(s,T)T\right]S^{-1}_L(p,T) \mathcal{Q}_p^2(T)\\
\nonumber
&-\left[s \mathcal{F}_7^R(s,T)-\mathcal{F}_7^R(s,T)T\right]\overline{T} S^{-1}_L(p,T) \mathcal{Q}_p^2(T)\bigl\}\\
\nonumber
& =\gamma_{7}^{-1} \bigl[s^2 \mathcal{F}_7^R(s,T)S^{-1}_L(p,T) \mathcal{Q}_p^2(T)
-s\mathcal{F}_7^R(s,T)TS^{-1}_L(p,T) \mathcal{Q}_p^2(T)
\\ \nonumber
&
-s \mathcal{F}_7^R(s,T)\overline{T} S^{-1}_L(p,T) \mathcal{Q}_p^2(T)
+\mathcal{F}_7^R(s,T)|T|^2S^{-1}_L(p,T) \mathcal{Q}_p^2(T)\bigl],\\
\nonumber
\end{align}\endgroup

}
{\small
\begin{equation}
\label{n713b}
\mathcal{Q}_s(T)\mathcal{Q}_p^3(T)= \gamma_{7}^{-1}\left( \mathcal{Q}_s(T)\mathcal{F}_7^L(p,T)p-\mathcal{Q}_s(T)T \mathcal{F}_7^L(p,T) \right),
\end{equation}
}
{\small
\begin{equation}
\label{n714b}
\mathcal{Q}_s^3(T)\mathcal{Q}_p(T)= \gamma_{7}^{-1} \left(s\mathcal{F}_7^R(s,T)\mathcal{Q}_p(T)- \mathcal{F}_7^R(s,T)T\mathcal{Q}_p(T) \right).
\end{equation}
}
Now, since
{\small
 \begingroup\allowdisplaybreaks\begin{align}
\mathcal{Q}_s^3(T) \mathcal{Q}_p^3(T)&= \gamma_7^{-2} \bigl[s\mathcal{F}_7^R(s,T)\mathcal{F}_7^L(p,T)p-s\mathcal{F}_7^R(s,T)T\mathcal{F}_7^L(p,T)
\\ \nonumber
&-\mathcal{F}_7^R(s,T)T\mathcal{F}_7^L(p,T)p+\mathcal{F}_7^R(s,T)T^2\mathcal{F}_7^L(p,T) \bigl],
\end{align}\endgroup
}
we get
{\small

\begingroup\allowdisplaybreaks\begin{align}
\label{n715b}
\mathcal{Q}_s^2(T) \mathcal{Q}_p^2(T)&=
\mathcal{Q}_s^{-1}(T)\mathcal{Q}_s^3(T) \mathcal{Q}_p^3(T) \mathcal{Q}_{p}^{-1}(T)\\
\nonumber
&=(s^{2}-s(T+ \overline{T})+|T|^2)\mathcal{Q}_s^3(T) \mathcal{Q}_p^3(T) (p^{2}-p(T+ \overline{T})+|T|^2)
\\
\nonumber
& =s^{2} \mathcal{Q}_s^3(T) \mathcal{Q}_p^3(T) p^{2} +s^{2} \mathcal{Q}_s^3(T) \mathcal{Q}_p^3(T) \left(|T|^2-p(T+ \overline{T})\right)
\\
\nonumber
&  +\left(|T|^2-s(T+ \overline{T})\right)\mathcal{Q}_s^3(T) \mathcal{Q}_p^3(T)p^2+\left(|T|^2-s(T+ \overline{T})\right)\mathcal{Q}_s^3(T) \mathcal{Q}_p^3(T)\left(|T|^2-p(T+ \overline{T})\right)
\\
\nonumber
&=\gamma^{-2}_7 \bigl\{ s^3\mathcal{F}_7^R(s,T)\mathcal{F}_7^L(p,T)p^3
-s^3\mathcal{F}_7^R(s,T)T\mathcal{F}_7^L(p,T)p^2
-s^2\mathcal{F}_7^R(s,T)T\mathcal{F}_7^L(p,T)p^3\\
\nonumber
& +s^2\mathcal{F}_7^R(s,T)T^2\mathcal{F}_7^L(p,T)p^2+ s^3 \mathcal{F}_7^R(s,T) \mathcal{F}_7^L(p,T)p[|T|^2-p(T+ \overline{T})]\\
\nonumber
& -s^3\mathcal{F}_7^R(s,T) T \mathcal{F}_7^L(p,T)[|T|^2-p(T+ \overline{T})] -s^2\mathcal{F}_7^R(s,T) T \mathcal{F}_7^L(p,T)p[|T|^2-p(T+ \overline{T})]\\
\nonumber
&+ s^2\mathcal{F}_7^R(s,T) T^2 \mathcal{F}_7^L(p,T)[|T|^2-p(T+ \overline{T})] +\left[|T|^2-s(T+ \overline{T})\right]s\mathcal{F}_7^R(s,T) \mathcal{F}_7^L(p,T) p^3\\
\nonumber
&-\left[|T|^2-s(T+ \overline{T})\right]s\mathcal{F}_7^R(s,T) T \mathcal{F}_7^L(p,T) p^2-\left[|T|^2-s(T+ \overline{T})\right]\mathcal{F}_7^R(s,T)T \mathcal{F}_7^L(p,T) p^3\\
\nonumber
&+\left[|T|^2-s(T+ \overline{T})\right]\mathcal{F}_7^R(s,T) T^2 \mathcal{F}_7^L(p,T)p^2 +\left[|T|^2-s(T+ \overline{T})\right]\bigl( s\mathcal{F}_7^R(s,T)\mathcal{F}_7^L(p,T)p
\\
\nonumber
&-s\mathcal{F}_7^R(s,T)T\mathcal{F}_7^L(p,T)
-\mathcal{F}_7^R(s,T)T\mathcal{F}_7^L(p,T)p+\mathcal{F}_7^R(s,T)T^2\mathcal{F}_7^L(p,T)\bigl)\cdot
\\ \nonumber
& \cdot \left[|T|^2-p(T+ \overline{T})\right]\bigl\}.
\end{align}\endgroup

}
Summing \eqref{n711b}, \eqref{n712b}, \eqref{n713b}, \eqref{n714b} and \eqref{n715b} we get the statement.
\end{proof}

\section{\bf The $\mathcal{F}$-resolvent equation for $n$ odd}\label{GENERALCASE}
\setcounter{equation}{0}

With the $\mathcal{F}$-resolvent equations in the cases $n=3,5,7$ we can now find
a reasonable form in the general case involving the pseudo $S$-resolvent operators when it is not necessary to replace the $\mathcal{F}$-resolvent operators.
  Using the pseudo $S$-resolvent operators there are interesting symmetries that allows to use the $\mathcal{F}$-resolvent equation for the applications, for example to compute the Riesz projectors. We start by proving a technical result which generalizes Lemma \ref{res1} and Lemma \ref{res7}, see Remark \ref{gen}.

\begin{lemma}[The general structure of the
 $\mathcal{F}$-resolvent equation with the pseudo $S$-resolvent operators]\label{PSGEN}
Let $n>3$ be an odd number, and let $ h= \frac{n-1}{2}$ be the Sce exponent. Let us consider $T \in \mathcal{BC}^{0,1}(V_n)$. Then for $p,s \in \rho_{\mathcal{F}}(T)$ the following equation holds
{\small

\begingroup\allowdisplaybreaks\begin{align}
\label{resn}
& \, \, \, \,\mathcal{F}_n^R(s,T)S^{-1}_L(p,T)+S^{-1}_R(s,T)\mathcal{F}_n^L(p,T)
\\
&+ \gamma_n \biggl[ \sum_{i=0}^{h-2} \mathcal{Q}^{h-i-1}_s(T) S^{-1}_R(s,T)S^{-1}_L(p,T) \mathcal{Q}_p^{i+1}(T)
\nonumber
+ \sum_{i=0}^{h-1}  \mathcal{Q}_s^{h-i}(T) \mathcal{Q}_p^{i+1}(T) \biggl]
\\
&
=\bigl \{ \bigl[\mathcal{F}_n^R(s,T)-\mathcal{F}_n^L(p,T) \bigl]p- \bar{s}\bigl[\mathcal{F}_n^R(s,T)-\mathcal{F}_n^L(p,T) \bigl] \bigl \} (p^2-2s_0p+|s|^2)^{-1}.
\nonumber
\end{align}\endgroup

}
\end{lemma}
\begin{proof}
We left multiply the S-resolvent equation  \eqref{reso} by $ \gamma_n \mathcal{Q}_s^{h}(T)$
{\small

\begingroup\allowdisplaybreaks\begin{align}
\label{n1}
\mathcal{F}_n^R(s,T)S^{-1}_L(p,T)&= \left\{[\mathcal{F}_n^R(s,T)-\gamma_n \mathcal{Q}_s^h(T)S^{-1}_L(p,T)]p- \bar{s}[\mathcal{F}_n^R(s,T)-\gamma_n \mathcal{Q}_s^h(T)S^{-1}_L(p,T)]\right\}\cdot\\ \nonumber
&\cdot (p^2-2s_0p+|s|^2)^{-1}
\end{align}\endgroup

}

and we right multiply it by $ \gamma_n \mathcal{Q}_p^h(T)$, so we get
{\small

\begingroup\allowdisplaybreaks\begin{align}
\label{n2}
S^{-1}_R(s,T)\mathcal{F}_n^L(p,T)&= \left\{[\gamma_n S^{-1}_R(s,T)\mathcal{Q}_p^h(T)-\mathcal{F}_n^L(p,T)]p- \bar{s}[\gamma_n  S^{-1}_R(s,T) \mathcal{Q}_p^h(T)-\mathcal{F}_n^L(p,T)]\right\}\cdot\\ \nonumber
& \cdot (p^2-2s_0p+|s|^2)^{-1}.
\end{align}\endgroup

}
We now multiply S-resolvent equation on the left and on the right by $ \mathcal{Q}_s^{h-1-i}(T)$ and $ \mathcal{Q}_p^{i+1}(T)$, respectively. Then, we sum on the index $0 \leq i \leq h-2$ and we obtain

{\small

\begingroup\allowdisplaybreaks\begin{align}
\label{n3}
& \, \, \, \, \, \, \, \, \, \, \, \, \, \, \, \, \, \, \, \sum_{i=0}^{h-2} \mathcal{Q}_s^{h-1-i}(T) S^{-1}_R(s,T) S^{-1}_L(p,T)\mathcal{Q}_p^{i+1}(T)= \biggl\{ \biggl[ \sum_{i=0}^{h-2} \mathcal{Q}_s^{h-1-i}(T) S^{-1}_R(s,T)\mathcal{Q}_p^{i+1}(T)
\\ \nonumber
\nonumber
& - \mathcal{Q}_s^{h-1-i}(T) S^{-1}_L(p,T)\mathcal{Q}_p^{i+1}(T)\biggl]p- \bar{s} \biggl[ \sum_{i=0}^{h-2} \mathcal{Q}_s^{h-1-i}(T) S^{-1}_R(s,T)\mathcal{Q}_p^{i+1}(T)- \mathcal{Q}_s^{h-1-i}(T) S^{-1}_L(p,T)\mathcal{Q}_p^{i+1}(T)\biggl] \biggl\} \cdot
\\
\nonumber
& \cdot (p^2-2s_0p+|s|^2)^{-1}.
\end{align}\endgroup

}

Now we sum \eqref{n1}, \eqref{n2} and \eqref{n3} multiplied by $ \gamma_n$, and we get
{\small
\begingroup\allowdisplaybreaks\begin{align}
& \, \, \, \,\mathcal{F}_n^R(s,T)S^{-1}_L(p,T)+S^{-1}_R(s,T)\mathcal{F}_n^L(p,T)+ \gamma_n  \sum_{i=0}^{h-2} \mathcal{Q}^{h-i-1}_s(T) S^{-1}_R(s,T)S^{-1}_L(p,T) \mathcal{Q}_p^{i+1}(T)
\\ \nonumber
&
= \biggl \{ \biggl[\mathcal{F}_n^R(s,T)-\gamma_n \mathcal{Q}_s^h(T)S^{-1}_L(p,T)
\\ \nonumber
&
+\gamma_n S^{-1}_R(s,T)\mathcal{Q}_p^h(T)-\mathcal{F}_n^L(p,T) + \gamma_n\sum_{i=0}^{h-2} \mathcal{Q}_s^{h-1-i}(T) S^{-1}_R(s,T)\mathcal{Q}_p^{i+1}(T)
\\ \nonumber
&
-\mathcal{Q}_s^{h-1-i}(T) S^{-1}_L(p,T)\mathcal{Q}_p^{i+1}(T)\biggl]p
\\ \nonumber
&
- \bar{s}\biggl[\mathcal{F}_n^R(s,T)
-\gamma_n \mathcal{Q}_s^h(T)S^{-1}_L(p,T)+\gamma_n S^{-1}_R(s,T)\mathcal{Q}_p^h(T)-\mathcal{F}_n^L(p,T)
\\ \nonumber
&
 +\gamma_n\sum_{i=0}^{h-2} \mathcal{Q}_s^{h-1-i}(T) S^{-1}_R(s,T)\mathcal{Q}_p^{i+1}(T)
\\ \nonumber
&
-\mathcal{Q}_s^{h-1-i}(T) S^{-1}_L(p,T)\mathcal{Q}_p^{i+1}(T)\biggl]
\biggl\}(p^2-2s_0p+|s|^2)^{-1}.
\end{align}\endgroup

}
Putting in order the terms in the right hand side of the previous equation we get
{\small

\begingroup\allowdisplaybreaks\begin{align}
\label{n5}
& \mathcal{F}_n^R(s,T)S^{-1}_L(p,T)+S^{-1}_R(s,T)\mathcal{F}_n^L(p,T)
\\ \nonumber
&
+ \gamma_n  \sum_{i=0}^{h-2} \mathcal{Q}^{h-i-1}_s(T) S^{-1}_R(s,T)S^{-1}_L(p,T) \mathcal{Q}_p^{i+1}(T)\\ \nonumber
&
= \bigl \{ \bigl[\mathcal{F}_n^R(s,T)-\mathcal{F}_n^L(p,T) \bigl]p- \bar{s}\bigl[\mathcal{F}_n^R(s,T)-\mathcal{F}_n^L(p,T) \bigl] \bigl \} (p^2-2s_0p+|s|^2)^{-1}\\ \nonumber
& +  \gamma_n\bigl \{ \bigl[S^{-1}_R(s,T)\mathcal{Q}_p^h(T)- \mathcal{Q}_s^h(T)S^{-1}_L(p,T)+\sum_{i=0}^{h-2} \mathcal{Q}_s^{h-1-i}(T) S^{-1}_R(s,T)\mathcal{Q}_p^{i+1}(T)\\
\nonumber
& -\mathcal{Q}_s^{h-1-i}(T) S^{-1}_L(p,T)\mathcal{Q}_p^{i+1}(T)\bigl]p- \bar{s}\bigl[S^{-1}_R(s,T)\mathcal{Q}_p^h(T)- \mathcal{Q}_s^h(T)S^{-1}_L(p,T)\\
\nonumber
& + \sum_{i=0}^{h-2} \mathcal{Q}_s^{h-1-i}(T) S^{-1}_R(s,T)\mathcal{Q}_p^{i+1}(T)-\mathcal{Q}_s^{h-1-i}(T) S^{-1}_L(p,T)\mathcal{Q}_p^{i+1}(T)\bigl]\bigl\}(p^2-2s_0p+|s|^2)^{-1}.
\end{align}\endgroup

}
Now, using the definition of left and right S-resolvent operators we get
{\small

\begingroup\allowdisplaybreaks\begin{align}
&                   S^{-1}_R(s,T)\mathcal{Q}_p^h(T)- \mathcal{Q}_s^h(T)S^{-1}_L(p,T)+\sum_{i=0}^{h-2} \mathcal{Q}_s^{h-1-i}(T) S^{-1}_R(s,T)\mathcal{Q}_p^{i+1}(T) -\sum_{i=0}^{h-2}\mathcal{Q}_s^{h-1-i}(T) S^{-1}_L(p,T)\mathcal{Q}_p^{i+1}(T)
\\ \nonumber
&                   = \mathcal{Q}_s(T)(s \mathcal{I}- \overline{T}) \mathcal{Q}_p^{h}(T)-\mathcal{Q}_s^h(T) (p \mathcal{I}- \overline{T})
 \mathcal{Q}_p(T)+ \sum_{i=0}^{h-2} \mathcal{Q}_s^{h-i}(T)(s \mathcal{I}- \overline{T}) \mathcal{Q}_p^{i+1}(T)
\\ \nonumber
&                  -\sum_{i=0}^{h-2}\mathcal{Q}_s^{h-1-i}(T)(p \mathcal{I}- \overline{T}) \mathcal{Q}_p^{i+2}(T)\\ \nonumber
&                 = \mathcal{Q}_s(T)(s \mathcal{I}- \overline{T}) \mathcal{Q}_p^{h}(T)-\mathcal{Q}_s^h(T) (p \mathcal{I}- \overline{T}) \mathcal{Q}_p(T)+\mathcal{Q}_s^h(T) (s \mathcal{I}- \overline{T}) \mathcal{Q}_p(T)+ \sum_{i=1}^{h-2} \mathcal{Q}_s^{h-i}(T)(s \mathcal{I}- \overline{T}) \mathcal{Q}_p^{i+1}(T)
\\ \nonumber
&                 -\mathcal{Q}_s(T)(p \mathcal{I}- \overline{T}) \mathcal{Q}_p^{h}(T)-\sum_{i=0}^{h-3}\mathcal{Q}_s^{h-1-i}(T)(p \mathcal{I}- \overline{T}) \mathcal{Q}_p^{i+2}(T)\\ \nonumber
&                   =\mathcal{Q}_s(T)(s-p) \mathcal{Q}_p^h(T)+ \mathcal{Q}_s^h(T)(s-p) \mathcal{Q}_p(T)+\sum_{i=1}^{h-2} \mathcal{Q}_s^{h-i}(T)(s \mathcal{I}- \overline{T}) \mathcal{Q}_p^{i+1}(T) -\sum_{i=0}^{h-3}\mathcal{Q}_s^{h-1-i}(T)(p \mathcal{I}- \overline{T}) \mathcal{Q}_p^{i+2}(T)
\\ \nonumber
&
 = \mathcal{Q}_s(T)(s-p) \mathcal{Q}_p^h(T)+ \mathcal{Q}_s^h(T)(s-p) \mathcal{Q}_p(T)+\sum_{i=1}^{h-2} \mathcal{Q}_s^{h-i}(T)(s \mathcal{I}- \overline{T}) \mathcal{Q}_p^{i+1}(T)-\sum_{i=1}^{h-2} \mathcal{Q}_s^{h-i}(T)(p \mathcal{I}- \overline{T}) \mathcal{Q}_p^{i+1}(T)
 \\
 \nonumber
&                 =  \mathcal{Q}_s(T)(s-p) \mathcal{Q}_p^h(T)+ \mathcal{Q}_s^h(T)(s-p) \mathcal{Q}_p(T)+ \sum_{i=1}^{h-2} \mathcal{Q}_s^{h-i}(T)(s-p) \mathcal{Q}_p^{i+1}(T)
\\ \nonumber
&                 = \sum_{i=0}^{h-1} \mathcal{Q}_s^{h-i}(T)(s-p) \mathcal{Q}_p^{i+1}(T).
\end{align}\endgroup

}
Then we compute
{\small

\begingroup\allowdisplaybreaks\begin{align*}
& \gamma_n\bigl \{ \bigl[S^{-1}_R(s,T)\mathcal{Q}_p^h(T)- \mathcal{Q}_s^h(T)S^{-1}_L(p,T)+\sum_{i=0}^{h-2} \mathcal{Q}_s^{h-1-i}(T) S^{-1}_R(s,T)\mathcal{Q}_p^{i+1}(T)
\\ \nonumber
\nonumber
& -\mathcal{Q}_s^{h-1-i}(T) S^{-1}_L(p,T)\mathcal{Q}_p^{i+1}(T)\bigl]p- \bar{s}\bigl[S^{-1}_R(s,T)\mathcal{Q}_p^h(T)- \mathcal{Q}_s^h(T)S^{-1}_L(p,T)
\\ \nonumber
\nonumber
& + \sum_{i=0}^{h-2} \mathcal{Q}_s^{h-1-i}(T) S^{-1}_R(s,T)\mathcal{Q}_p^{i+1}(T)-\mathcal{Q}_s^{h-1-i}(T) S^{-1}_L(p,T)\mathcal{Q}_p^{i+1}(T) \bigl]\bigl\}(p^2-2s_0p+|s|^2)^{-1}
\\ \nonumber
&= \gamma_n\biggl \{ \biggl[\sum_{i=0}^{h-1} \mathcal{Q}_s^{h-i}(T)(s-p) \mathcal{Q}_p^{i+1}(T) \biggl]p- \bar{s}\biggl[\sum_{i=0}^{h-1} \mathcal{Q}_s^{h-i}(T)(s-p) \mathcal{Q}_p^{i+1}(T) \biggl]\biggl \}(p^2-2s_0p+|s|^2)^{-1}
\\ \nonumber
&= \gamma_n\biggl[\sum_{i=0}^{h-1} \mathcal{Q}_s^{h-i}(T)(sp-p^2) \mathcal{Q}_p^{i+1}(T) -\sum_{i=0}^{h-1} \mathcal{Q}_s^{h-i}(T)(|s|^2-\bar{s}p) \mathcal{Q}_p^{i+1}(T) \biggl](p^2-2s_0p+|s|^2)^{-1}
\\ \nonumber
&= \gamma_n\biggl[\sum_{i=0}^{h-1} \mathcal{Q}_s^{h-i}(T)(sp-p^2-|s|^2+\bar{s}p) \mathcal{Q}_p^{i+1}(T) \biggl](p^2-2s_0p+|s|^2)^{-1}
\\ \nonumber
&= - \gamma_n \sum_{i=0}^{h-1} \mathcal{Q}_s^{h-i}(T)\mathcal{Q}_p^{i+1}(T)(p^2-2s_0p+|s|^2)(p^2-2s_0p+|s|^2)^{-1}
\\ \nonumber
&=- \gamma_n \sum_{i=0}^{h-1} \mathcal{Q}_s^{h-i}(T)\mathcal{Q}_p^{i+1}(T).
\end{align*}\endgroup

}
Hence
{\small

\begingroup\allowdisplaybreaks\begin{align}
\label{n6}
& \gamma_n\bigl \{ \bigl[S^{-1}_R(s,T)\mathcal{Q}_p^h(T)- \mathcal{Q}_s^h(T)S^{-1}_L(p,T)+\sum_{i=0}^{h-2} \mathcal{Q}_s^{h-1-i}(T) S^{-1}_R(s,T)\mathcal{Q}_p^{i+1}(T)
\\
\nonumber
& -\mathcal{Q}_s^{h-1-i}(T) S^{-1}_L(p,T)\mathcal{Q}_p^{i+1}(T)\bigl]p- \bar{s}\bigl[S^{-1}_R(s,T)\mathcal{Q}_p^h(T)- \mathcal{Q}_s^h(T)S^{-1}_L(p,T)
\\
\nonumber
& + \sum_{i=0}^{h-2} \mathcal{Q}_s^{h-1-i}(T) S^{-1}_R(s,T)\mathcal{Q}_p^{i+1}(T)-\mathcal{Q}_s^{h-1-i}(T) S^{-1}_L(p,T)\mathcal{Q}_p^{i+1}(T) \bigl]\bigl\}(p^2-2s_0p+|s|^2)^{-1}\\
\nonumber
&=- \gamma_n \sum_{i=0}^{h-1} \mathcal{Q}_s^{h-i}(T)\mathcal{Q}_p^{i+1}(T).
\end{align}\endgroup

}
Finally, by substituting \eqref{n6} in \eqref{n5} we get \eqref{resn}.
\end{proof}
\begin{remark}
The proof of the previous lemma shows that the structure of the resolvent equations of the hyperholomorphic functional calculi is crucial. In fact the term
{\small
$$
 \{ \bigl[\mathcal{F}_n^R(s,T)-\mathcal{F}_n^L(p,T) \bigl]p- \bar{s}\bigl[\mathcal{F}_n^R(s,T)-\mathcal{F}_n^L(p,T) \bigl] \bigl \} (p^2-2s_0p+|s|^2)^{-1}
$$
}
involves the difference of the $\mathcal{F}$-resolvent operators entangled with the Cauchy kernel of slice monogenic functions. This term is equal to a function involving the products of the
$F$-resolvent operators and of the $S$-resolvent operators that appear in the term
{\small
$$
\mathcal{F}_n^R(s,T)S^{-1}_L(p,T)+S^{-1}_R(s,T)\mathcal{F}_n^L(p,T)
$$
}
and of a more complicated part that involves the
$S$-resolvent operators and the pseudo $S$-resolvent operators, namely
{\small
$$
\gamma_n \biggl[ \sum_{i=0}^{h-2} \mathcal{Q}^{h-i-1}_s(T) S^{-1}_R(s,T)S^{-1}_L(p,T) \mathcal{Q}_p^{i+1}(T)
+ \sum_{i=0}^{h-1}  \mathcal{Q}_s^{h-i}(T) \mathcal{Q}_p^{i+1}(T) \biggl].
$$
}
\end{remark}

\begin{remark}
\label{gen}
Equation \eqref{resn} generalizes \eqref{F5} and \eqref{n70}. Indeed if we put $n=5$, then $h=2$ and we get
{\small

\begingroup\allowdisplaybreaks\begin{align}
& \mathcal{F}_5^R(s,T)S^{-1}_L(p,T)+S^{-1}_R(s,T)\mathcal{F}_5^L(p,T)+ \gamma_5 \biggl[ \sum_{i=0}^{0} \mathcal{Q}^{1-i}_s(T) S^{-1}_R(s,T)S^{-1}_L(p,T) \mathcal{Q}_p^{i+1}(T)
\\
\nonumber
& + \sum_{i=0}^{1}  \mathcal{Q}_s^{2-i}(T) \mathcal{Q}_p^{i+1}(T) \biggl]=\bigl \{ \bigl[\mathcal{F}_5^R(s,T)-\mathcal{F}_5^L(p,T) \bigl]p- \bar{s}\bigl[\mathcal{F}_5^R(s,T)-\mathcal{F}_5^L(p,T) \bigl] \bigl \} (p^2-2s_0p+|s|^2)^{-1}.
\end{align}\endgroup

}
By developing the computations we obtain
{\small

\begingroup\allowdisplaybreaks\begin{align}
& \mathcal{F}_5^R(s,T)S^{-1}_L(p,T)+S^{-1}_R(s,T)\mathcal{F}_{5}^L(p,T)+\gamma_5 \bigl[ \mathcal{Q}_s(T) S^{-1}_R(s,T)S^{-1}_{L}(p,T)\mathcal{Q}_p(T)+ \mathcal{Q}_s^2(T) \mathcal{Q}_p(T)
\\
\nonumber
&+\mathcal{Q}_s(T) \mathcal{Q}_p^2(T)\bigl]= \bigl \{ [\mathcal{F}_5^R(s,T)- \mathcal{F}_5^L(p,T)]p- \bar{s}[\mathcal{F}_5^R(s,T)- \mathcal{F}_5^L(p,T)]\} (p^2-2s_0 p+|s|^2)^{-1},
\end{align}\endgroup

}
which is exactly \eqref{F5}.
\\ Now if we put $n=7$ in \eqref{resn}, then $h=3$ and we obtain
\vspace{-4mm}
{\small

\begingroup\allowdisplaybreaks\begin{align}
& \, \, \, \,\mathcal{F}_7^R(s,T)S^{-1}_L(p,T)+S^{-1}_R(s,T)\mathcal{F}_7^L(p,T)+ \gamma_7 \biggl[ \sum_{i=0}^{1} \mathcal{Q}^{2-i}_s(T) S^{-1}_R(s,T)S^{-1}_L(p,T) \mathcal{Q}_p^{i+1}(T)
\\
\nonumber
& + \sum_{i=0}^{2}  \mathcal{Q}_s^{3-i}(T) \mathcal{Q}_p^{i+1}(T) \biggl]=\bigl \{ \bigl[\mathcal{F}_7^R(s,T)-\mathcal{F}_7^L(p,T) \bigl]p- \bar{s}\bigl[\mathcal{F}_7^R(s,T)-\mathcal{F}_7^L(p,T) \bigl] \bigl \} (p^2-2s_0p+|s|^2)^{-1}.
\end{align}\endgroup

}
By developing the computations we get
{\small

\begingroup\allowdisplaybreaks\begin{align}
& \mathcal{F}_7^R(s,T)S^{-1}_L(p,T)+S^{-1}_R(s,T)\mathcal{F}_7^L(p,T)+ \gamma_7 \bigl[ \mathcal{Q}_s^2(T)S^{-1}_R(s,T)S^{-1}_L(p,T)\mathcal{Q}_p(T)\\
\nonumber
&+\mathcal{Q}_s(T)S^{-1}_R(s,T)S^{-1}_L(p,T)\mathcal{Q}_p^2(T)+\mathcal{Q}_s^3(T)\mathcal{Q}_p(T)+ \mathcal{Q}_s^2(T) \mathcal{Q}_p^2(T)+ \mathcal{Q}_s(T) \mathcal{Q}_p^3(T)]\\
\nonumber
& = \bigl \{ \bigl[\mathcal{F}_7^R(s,T)-\mathcal{F}_7^L(p,T) \bigl]p- \bar{s}\bigl[\mathcal{F}_7^R(s,T)-\mathcal{F}_7^L(p,T) \bigl] \bigl \} (p^2-2s_0p+|s|^2)^{-1},
\end{align}\endgroup

}
which is exactly \eqref{n70}.
\end{remark}

In order to find a pseudo $ \mathcal{F}$-resolvent equation we divide into two cases according to the parity of the Sce exponent
$h= \dfrac{n-1}{2}$.

\subsection{The general structure of the
pseudo $\mathcal{F}$-resolvent equation for $h$ odd}

The main result of this subsection is the following theorem.
\begin{theorem}[The general structure of the
pseudo $\mathcal{F}$-resolvent equation for $h$ odd number]\label{FREhodd}
\label{1T}
Let $n>3$ be an odd number as well as $h$. Let $T \in \mathcal{BC}^{0,1}(V_n)$. Then for $p,s \in \rho_{\mathcal{F}}(T)$ the following equation holds
{\small

\begingroup\allowdisplaybreaks\begin{align}
& \, \, \, \,\mathcal{F}_n^R(s,T)S^{-1}_L(p,T)+S^{-1}_R(s,T)\mathcal{F}_n^L(p,T)+ \gamma_n \biggl[ s \sum_{i=0}^{h-2} \mathcal{Q}^{h-i}_s(T)\mathcal{Q}_p^{i+2}(T)p-s \sum_{i=0}^{h-2} \mathcal{Q}^{h-i}_s(T) \overline{T}\mathcal{Q}_p^{i+2}(T)
\\ \nonumber
& -\sum_{i=0}^{h-2} \mathcal{Q}^{h-i}_s(T)\overline{T}\mathcal{Q}_p^{i+2}(T)p+ \sum_{i=0}^{h-2} \mathcal{Q}^{h-i}_s(T)\overline{T}^2\mathcal{Q}_p^{i+2}(T)+\sum_{i=0, i \neq \frac{h-1}{2}}^{h-1}  \mathcal{Q}_s^{h-i}(T) \mathcal{Q}_p^{i+1}(T) \biggl]
\\ \nonumber
&
+ \gamma_n^{-1} \biggl[\ s^{h} \mathcal{F}_{n}^{R}(s,T) \mathcal{F}_n^L(p,T) p^h
+\mathcal{A}_0(s,p,T)+\mathcal{B}_0(s,p,T)+\mathcal{C}_0(s,p,T) \biggl]
\\ \nonumber
&
=\bigl \{ \bigl[\mathcal{F}_n^R(s,T)-\mathcal{F}_n^L(p,T) \bigl]p- \bar{s}\bigl[\mathcal{F}_n^R(s,T)-\mathcal{F}_n^L(p,T) \bigl] \bigl \} (p^2-2s_0p+|s|^2)^{-1},
\end{align}\endgroup

}
where the three terms $\mathcal{A}_0(s,p,T)$, $\mathcal{B}_0(s,p,T)$ and $\mathcal{C}_0(s,p,T)$ are given by:
{\small

\begingroup\allowdisplaybreaks\begin{align}
 \mathcal{A}_0(s,p,T):&= -s^h \mathcal{F}_n^R(s,T)T \mathcal{F}_n^L(p,T) p^{h-1}-s^{h-1} \mathcal{F}_n^R(s,T)T \mathcal{F}_n^L(p,T) p^{h}
 \\ \nonumber
&
 +s^{h-1}\mathcal{F}_n^R(s,T)T^2 \mathcal{F}_n^L(p,T) p^{h-1}
 \\ \nonumber
&
+s^{h} \mathcal{F}_n^R(s,T) \sum_{k=1}^{\frac{h-1}{2}} \binom{\frac{h-1}{2}}{k}  (|T|^2-2T_0p)^k \mathcal{F}_n^L(p,T)p^{h-2k}
\\ \nonumber
&
-s^{h} \mathcal{F}_n^R(s,T) \sum_{k=1}^{\frac{h-1}{2}} \binom{\frac{h-1}{2}}{k}  (|T|^2-2T_0p)^k T\mathcal{F}_n^L(p,T)p^{h-1-2k}
\\ \nonumber
&
-s^{h-1} \mathcal{F}_n^{R}(s,T)\sum_{k=1}^{\frac{h-1}{2}} \binom{\frac{h-1}{2}}{k}  T(|T|^2-2T_0p)^k \mathcal{F}_n^L(p,T)p^{h-2k}
\\ \nonumber
&
+s^{h-1} \mathcal{F}_n^{R}(s,T)\sum_{k=1}^{\frac{h-1}{2}} \binom{\frac{h-1}{2}}{k}  T(|T|^2-2T_0p)^k T\mathcal{F}_n^L(p,T)p^{h-1-2k},
\end{align}\endgroup
}
{\small
\begingroup\allowdisplaybreaks\begin{align}
\mathcal{B}_0(s,p,T):&= \left(\sum_{k=1}^{\frac{h-1}{2}} \binom{\frac{h-1}{2}}{k} s^{h-2k} \mathcal{F}_n^R(s,T) (|T|^2-2T_0s)^k\right)\left(\sum_{k=1}^{\frac{h-1}{2}} \binom{\frac{h-1}{2}}{k}  (|T|^2-2T_0p)^k \mathcal{F}_n^L(p,T) p^{h-2k}\right)
\\ \nonumber
&
-\left(\sum_{k=1}^{\frac{h-1}{2}} \binom{\frac{h-1}{2}}{k} s^{h-2k} \mathcal{F}_n^R(s,T) (|T|^2-2T_0s)^k\right) \left(\sum_{k=1}^{\frac{h-1}{2}} \binom{\frac{h-1}{2}}{k}  (|T|^2-2T_0p)^k T\mathcal{F}_n^L(p,T) p^{h-1-2k}\right)
\\ \nonumber
&
-\left(\sum_{k=1}^{\frac{h-1}{2}} \binom{\frac{h-1}{2}}{k} s^{h-2k-1} \mathcal{F}_n^R(s,T) T(|T|^2-2T_0s)^k\right) \left(\sum_{k=1}^{\frac{h-1}{2}} \binom{\frac{h-1}{2}}{k}  (|T|^2-2T_0p)^k \mathcal{F}_n^L(p,T) p^{h-2k}\right)
\\ \nonumber
&
 +\left(\sum_{k=1}^{\frac{h-1}{2}} \binom{\frac{h-1}{2}}{k} s^{h-2k-1} \mathcal{F}_n^R(s,T) T(|T|^2-2T_0s)^k\right)\left(\sum_{k=1}^{\frac{h-1}{2}} \binom{\frac{h-1}{2}}{k}  (|T|^2-2T_0p)^k T \mathcal{F}_n^L(p,T) p^{h-1-2k}\right),
\end{align}\endgroup
}
and
{\small
\begingroup\allowdisplaybreaks\begin{align}
\mathcal{C}_0(s,p,T):&=
 \nonumber
  \left(\sum_{k=1}^{\frac{h-1}{2}} \binom{\frac{h-1}{2}}{k} s^{h-2k} \mathcal{F}_n^R(s,T) (|T|^2-2T_0s)^k \right) \mathcal{F}_n^{L}(p,T)p^h
  \\
  &
-\left(\sum_{k=1}^{\frac{h-1}{2}} \binom{\frac{h-1}{2}}{k} s^{h-2k} \mathcal{F}_n^R(s,T) (|T|^2-2T_0s)^k T \right)\mathcal{F}_n^{L}(p,T)p^{h-1}
\\ \nonumber
&
-\left(\sum_{k=1}^{\frac{h-1}{2}} \binom{\frac{h-1}{2}}{k} s^{h-1-2k} \mathcal{F}_n^R(s,T) T (|T|^2-2T_0s)^k \right) \mathcal{F}_n^{L}(p,T)p^h
\\ \nonumber
&
+\left(\sum_{k=1}^{\frac{h-1}{2}} \binom{\frac{h-1}{2}}{k} s^{h-1-2k} \mathcal{F}_n^R(s,T) T (|T|^2-2T_0s)^k T \right) \mathcal{F}_n^{L}(p,T)p^{h-1}.
\end{align}\endgroup

}
\end{theorem}
\begin{proof}
We start by rewriting formula \eqref{resn} as
{\small

\begingroup\allowdisplaybreaks\begin{align}
& \, \, \, \,\mathcal{F}_n^R(s,T)S^{-1}_L(p,T)+S^{-1}_R(s,T)\mathcal{F}_n^L(p,T)
\\ \nonumber
&
+ \gamma_n \biggl[ \sum_{i=0}^{h-2} \mathcal{Q}^{h-i-1}_s(T) S^{-1}_R(s,T)S^{-1}_L(p,T) \mathcal{Q}_p^{i+1}(T)
 + \mathcal{Q}_s^{\frac{h+1}{2}}(T) \mathcal{Q}_p^{\frac{h+1}{2}}(T)+ \sum_{i=0, i \neq \frac{h-1}{2}}^{h-1}  \mathcal{Q}_s^{h-i}(T) \mathcal{Q}_p^{i+1}(T) \biggl]
\\ \nonumber
&
=\bigl \{ \bigl[\mathcal{F}_n^R(s,T)-\mathcal{F}_n^L(p,T) \bigl]p- \bar{s}\bigl[\mathcal{F}_n^R(s,T)-\mathcal{F}_n^L(p,T) \bigl] \bigl \} (p^2-2s_0p+|s|^2)^{-1}.
\end{align}\endgroup

}
Now, we focus on the term $\mathcal{Q}_s^{\frac{h+1}{2}}(T) \mathcal{Q}_p^{\frac{h+1}{2}}(T)$ and with some manipulations we obtain
{\small

\begingroup\allowdisplaybreaks\begin{align}
\mathcal{Q}_s^{\frac{h+1}{2}}(T) \mathcal{Q}_p^{\frac{h+1}{2}}(T)& =\mathcal{Q}_s^{\frac{h+1}{2}}(T) \mathcal{Q}_s^{\frac{h-1}{2}}(T) \mathcal{Q}_s^{\frac{1-h}{2}}(T) \mathcal{Q}_p^{\frac{1-h}{2}}(T)\mathcal{Q}_p^{\frac{h-1}{2}}(T)\mathcal{Q}_p^{\frac{h+1}{2}}(T)
\\ \nonumber
&
= \mathcal{Q}_s^{h}(T)\mathcal{Q}_s^{\frac{1-h}{2}}(T)\mathcal{Q}_p^{\frac{1-h}{2}}(T)\mathcal{Q}_p^{h}(T).
\end{align}\endgroup

}
By the binomial formula we get
{\small

\begingroup\allowdisplaybreaks\begin{align}
&
\mathcal{Q}_s^{\frac{h+1}{2}}(T) \mathcal{Q}_p^{\frac{h+1}{2}}(T)= \mathcal{Q}_s^{h}(T) (s^{2} \mathcal{I}-2sT_0+T \overline{T})^{\frac{h-1}{2}}(p^{2} \mathcal{I}-2pT_0+T \overline{T})^{\frac{h-1}{2}}\mathcal{Q}_p^{h}(T)
\\ \nonumber
&
=\mathcal{Q}_s^{h}(T) \left( \sum_{k=0}^{\frac{h-1}{2}} \binom{\frac{h-1}{2}}{k} s^{h-1-2k} (|T|^2-2T_0s)^k\right)\left( \sum_{k=0}^{\frac{h-1}{2}} \binom{\frac{h-1}{2}}{k}  (|T|^2-2T_0p)^k p^{h-1-2k}\right)\mathcal{Q}_p^{h}(T)
\\ \nonumber
&
=\mathcal{Q}_s^{h}(T) \left( s^{h-1}+\sum_{k=1}^{\frac{h-1}{2}} \binom{\frac{h-1}{2}}{k} s^{h-1-2k} (|T|^2-2T_0s)^k\right)
\left( \sum_{k=1}^{\frac{h-1}{2}}
\binom{\frac{h-1}{2}}{k}  (|T|^2-2T_0p)^k p^{h-1-2k}+p^{h-1}\right)\mathcal{Q}_p^{h}(T)
\\ \nonumber
&
=s^{h-1} \mathcal{Q}_s^{h}(T)  \mathcal{Q}_p^{h}(T)p^{h-1}+s^{h-1} \mathcal{Q}_s^{h}(T)\left( \sum_{k=1}^{\frac{h-1}{2}} \binom{\frac{h-1}{2}}{k}  (|T|^2-2T_0p)^kp^{h-1-2k} \right)\mathcal{Q}_p^{h}(T)
\\ \nonumber
& +\mathcal{Q}_s^{h}(T)\left(\sum_{k=1}^{\frac{h-1}{2}} \binom{\frac{h-1}{2}}{k} s^{h-1-2k} (|T|^2-2T_0s)^k\right) \left(\sum_{k=1}^{\frac{h-1}{2}} \binom{\frac{h-1}{2}}{k}  (|T|^2-2T_0p)^k p^{h-1-2k}\right) \mathcal{Q}_p^{h}(T)
\\ \nonumber
& +\mathcal{Q}_s^{h}(T)\left(\sum_{k=1}^{\frac{h-1}{2}} \binom{\frac{h-1}{2}}{k} s^{h-1-2k} (|T|^2-2T_0s)^k \right) \mathcal{Q}_p^h(T)p^{h-1}.
\end{align}\endgroup

}
Now, we use the left and right $ \mathcal{F}$-resolvent equations, (see Theorem \ref{FRE})
{\small
$$
\mathcal{F}_{n}^L(p,T)p-T \mathcal{F}_{n}^L(p,T)= \gamma_n Q_p^{h}(T)
$$
}
and
{\small
$$
s \mathcal{F}_{n}^R(s,T)- \mathcal{F}_{n}^R(s,T)T= \gamma_n Q_s^{h}(T).
$$
}
We go through the computations term by term
{\small

\begingroup\allowdisplaybreaks\begin{align}
\label{R0}
&               s^{h-1} \mathcal{Q}_s^{h}(T)  \mathcal{Q}_p^{h}(T)p^{h-1}
\\ \nonumber
&
= \gamma_{n}^{-2} \biggl[ s^{h} \mathcal{F}_{n}^{R}(s,T) \mathcal{F}_n^L(p,T) p^h-s^h \mathcal{F}_n^R(s,T)T \mathcal{F}_n^L(p,T) p^{h-1}
\\
\nonumber
&
-s^{h-1} \mathcal{F}_n^R(s,T)T \mathcal{F}_n^L(p,T) p^{h}+s^{h-1}\mathcal{F}_n^R(s,T)T^2 \mathcal{F}_n^L(p,T) p^{h-1}\biggl].
\end{align}\endgroup

}
Then we consider
{\small

\begingroup\allowdisplaybreaks\begin{align}
&              s^{h-1} \mathcal{Q}_s^{h}(T)\left( \sum_{k=1}^{\frac{h-1}{2}} \binom{\frac{h-1}{2}}{k}  (|T|^2-2T_0p)^k\mathcal{Q}_p^{h}(T)p^{h-1-2k} \right)
\\ \nonumber
&
= \gamma_n^{-2} \biggl[s^{h} \mathcal{F}_n^R(s,T) \sum_{k=1}^{\frac{h-1}{2}} \binom{\frac{h-1}{2}}{k}  (|T|^2-2T_0p)^k \mathcal{F}_n^L(p,T)p^{h-2k}
\\ \nonumber
&
-s^{h} \mathcal{F}_n^R(s,T) \sum_{k=1}^{\frac{h-1}{2}} \binom{\frac{h-1}{2}}{k}  (|T|^2-2T_0p)^k T\mathcal{F}_n^L(p,T)p^{h-1-2k}
\\ \nonumber
&
-s^{h-1} \mathcal{F}_n^{R}(s,T)\sum_{k=1}^{\frac{h-1}{2}} \binom{\frac{h-1}{2}}{k}  T(|T|^2-2T_0p)^k \mathcal{F}_n^L(p,T)p^{h-2k}
\\ \nonumber
&
+s^{h-1} \mathcal{F}_n^{R}(s,T)\sum_{k=1}^{\frac{h-1}{2}} \binom{\frac{h-1}{2}}{k}  T(|T|^2-2T_0p)^k T\mathcal{F}_n^L(p,T)p^{h-1-2k}\biggl].
\end{align}\endgroup

}
Then we compute the term
{\small

\begingroup\allowdisplaybreaks\begin{align}
&
\left(\sum_{k=1}^{\frac{h-1}{2}} \binom{\frac{h-1}{2}}{k} s^{h-1-2k}\mathcal{Q}_s^{h}(T) (|T|^2-2T_0s)^k\right) \left(\sum_{k=1}^{\frac{h-1}{2}}
 \binom{\frac{h-1}{2}}{k}  (|T|^2-2T_0p)^k \mathcal{Q}_p^{h}(T)p^{h-1-2k}\right)
\\ \nonumber
&
 =\gamma_n^{-2} \biggl[\left(\sum_{k=1}^{\frac{h-1}{2}} \binom{\frac{h-1}{2}}{k} s^{h-2k} \mathcal{F}_n^R(s,T) (|T|^2-2T_0s)^k\right)\left(\sum_{k=1}^{\frac{h-1}{2}} \binom{\frac{h-1}{2}}{k}  (|T|^2-2T_0p)^k \mathcal{F}_n^L(p,T) p^{h-2k}\right)
 \\ \nonumber
&
-\left(\sum_{k=1}^{\frac{h-1}{2}} \binom{\frac{h-1}{2}}{k} s^{h-2k} \mathcal{F}_n^R(s,T) (|T|^2-2T_0s)^k\right) \left(\sum_{k=1}^{\frac{h-1}{2}}
\binom{\frac{h-1}{2}}{k}
 (|T|^2-2T_0p)^k T\mathcal{F}_n^L(p,T) p^{h-1-2k}\right)
\\ \nonumber
&
-\left(\sum_{k=1}^{\frac{h-1}{2}} \binom{\frac{h-1}{2}}{k} s^{h-2k-1} \mathcal{F}_n^R(s,T) T(|T|^2-2T_0s)^k\right) \left(\sum_{k=1}^{\frac{h-1}{2}}
\binom{\frac{h-1}{2}}{k}  (|T|^2-2T_0p)^k \mathcal{F}_n^L(p,T) p^{h-2k}\right)
\\ \nonumber
&
 +\left(\sum_{k=1}^{\frac{h-1}{2}} \binom{\frac{h-1}{2}}{k} s^{h-2k-1} \mathcal{F}_n^R(s,T) T(|T|^2-2T_0s)^k\right) \left(\sum_{k=1}^{\frac{h-1}{2}}
  \binom{\frac{h-1}{2}}{k}  (|T|^2-2T_0p)^k T \mathcal{F}_n^L(p,T) p^{h-1-2k}\right)\biggl].
\end{align}\endgroup

}
We have also
{\small

\begingroup\allowdisplaybreaks\begin{align}
&
\left(\sum_{k=1}^{\frac{h-1}{2}} \binom{\frac{h-1}{2}}{k} s^{h-1-2k}\mathcal{Q}_s^{h}(T) (|T|^2-2T_0s)^k \right) \mathcal{Q}_p^h(T)p^{h-1}
\\ \nonumber
&
=\gamma_n^{-2} \biggl[\left(\sum_{k=1}^{\frac{h-1}{2}} \binom{\frac{h-1}{2}}{k} s^{h-2k} \mathcal{F}_n^R(s,T) (|T|^2-2T_0s)^k \right) \mathcal{F}_n^{L}(p,T)p^{h}
\\ \nonumber
&
-\left(\sum_{k=1}^{\frac{h-1}{2}} \binom{\frac{h-1}{2}}{k} s^{h-2k} \mathcal{F}_n^R(s,T) (|T|^2-2T_0s)^k T \right)\mathcal{F}_n^{L}(p,T)p^{h-1}
\\ \nonumber
&
-\left(\sum_{k=1}^{\frac{h-1}{2}} \binom{\frac{h-1}{2}}{k} s^{h-1-2k} \mathcal{F}_n^R(s,T) T (|T|^2-2T_0s)^k \right) \mathcal{F}_n^{L}(p,T)p^h
\\ \nonumber
&
+\left(\sum_{k=1}^{\frac{h-1}{2}} \binom{\frac{h-1}{2}}{k} s^{h-1-2k} \mathcal{F}_n^R(s,T) T (|T|^2-2T_0s)^k T \right) \mathcal{F}_n^{L}(p,T)p^{h-1}\biggl].
\end{align}\endgroup

}
Finally by using the definition of left and right S-resolvent operators we get
{\small

\begingroup\allowdisplaybreaks\begin{align}
& \sum_{i=0}^{h-2} \mathcal{Q}^{h-i-1}_s(T) S^{-1}_R(s,T)S^{-1}_L(p,T) \mathcal{Q}_p^{i+1}(T)=\sum_{i=0}^{h-2} \mathcal{Q}^{h-i}_s(T)(s \mathcal{I}- \overline{T})(p \mathcal{I}- \overline{T}) \mathcal{Q}_p^{i+2}(T)\\ \nonumber
&=s \sum_{i=0}^{h-2} \mathcal{Q}^{h-i}_s(T)\mathcal{Q}_p^{i+2}(T)p-s \sum_{i=0}^{h-2} \mathcal{Q}^{h-i}_s(T) \overline{T}\mathcal{Q}_p^{i+2}(T)-\sum_{i=0}^{h-2} \mathcal{Q}^{h-i}_s(T)\overline{T}\mathcal{Q}_p^{i+2}(T)p
\\ \nonumber
& \, \, \, \, \,+ \sum_{i=0}^{h-2} \mathcal{Q}^{h-i}_s(T)\overline{T}^2\mathcal{Q}_p^{i+2}(T),
\end{align}\endgroup

}
and this concludes the proof.
\end{proof}

\subsection{The general structure of the
pseudo $\mathcal{F}$-resolvent equation for $h$ even number}
In this last subsection we consider the case in which
 $h=(n-1)/2$ is an even number.

\begin{theorem}[The general structure of the
pseudo $\mathcal{F}$-resolvent equation for $h$ even number]\label{FREeven}
\label{2T}
Let $n>3$ be an odd number and $h$  be even. Let $T \in \mathcal{BC}^{0,1}(V_n)$. Then for $p,s \in \rho_{\mathcal{F}}(T)$ the following equation holds
{\small
\begingroup
\allowdisplaybreaks
\begin{align}
&
\mathcal{F}_n^R(s,T)S^{-1}_L(p,T)+S^{-1}_R(s,T)\mathcal{F}_n^L(p,T)
\\
\nonumber
&
+ \gamma_n \biggl[-s\mathcal{Q}_s^{\frac{h+2}{2}}(T)
\overline{T}\mathcal{Q}_p^{\frac{h+2}{2}}(T)
-\mathcal{Q}_s^{\frac{h+2}{2}}(T)\overline{T}\mathcal{Q}_p^{\frac{h+2}{2}}(T)p+
\mathcal{Q}_s^{\frac{h+2}{2}}(T)\overline{T}^2\mathcal{Q}_p^{\frac{h+2}{2}}(T)
\\ \nonumber
&
+  \sum_{i=0}^{h-1}  \mathcal{Q}_s^{h-i}(T) \mathcal{Q}_p^{i+1}(T)
+s    \sum_{i=0, i \neq \frac{h-2}{2}}^{h-2} \mathcal{Q}^{h-i}_s(T)\mathcal{Q}_p^{i+2}(T)p-s
       \sum_{i=0, i \neq \frac{h-2}{2}}^{h-2} \mathcal{Q}^{h-i}_s(T) \overline{T}\mathcal{Q}_p^{i+2}(T)
 \\ \nonumber
&
-\sum_{i=0, i \neq \frac{h-2}{2}}^{h-2} \mathcal{Q}^{h-i}_s(T)\overline{T}\mathcal{Q}_p^{i+2}(T)p+ \sum_{i=0, i \neq \frac{h-2}{2}}^{h-2} \mathcal{Q}^{h-i}_s(T)\overline{T}^2\mathcal{Q}_p^{i+2}(T) \biggl]
\\ \nonumber
&
+ \gamma_n^{-1}\left [\mathcal{A}_1(s,p,T) +\mathcal{B}_1(s,p,T) +\mathcal{C}_1(s,p,T) +s^{h} \mathcal{F}_{n}^{R}(s,T) \mathcal{F}_n^L(p,T) p^h\right]\\ \nonumber
&=\bigl \{ \bigl[\mathcal{F}_n^R(s,T)-\mathcal{F}_n^L(p,T) \bigl]p
- \bar{s}\bigl[\mathcal{F}_n^R(s,T)-\mathcal{F}_n^L(p,T) \bigl] \bigl\}  (p^2-2s_0p+|s|^2)^{-1},
\end{align}
\endgroup
}
where
{\small

\begingroup\allowdisplaybreaks\begin{align}
 \mathcal{A}_1(s,p,T):&= -s^h \mathcal{F}_n^R(s,T)T \mathcal{F}_n^L(p,T) p^{h-1}-s^{h-1} \mathcal{F}_n^R(s,T)T \mathcal{F}_n^L(p,T) p^{h}
 \\ \nonumber
 &
 +s^{h-1}\mathcal{F}_n^R(s,T)T^2 \mathcal{F}_n^L(p,T) p^{h-1}
\end{align}\endgroup

}
and
{\small

\begingroup\allowdisplaybreaks\begin{align}
\mathcal{B}_1(s,p,T):&=
s^{h} \mathcal{F}_n^R(s,T) \sum_{k=1}^{\frac{h-2}{2}} \binom{\frac{h-2}{2}}{k}  (|T|^2-2T_0p)^k \mathcal{F}_n^L(p,T)p^{h-2k}
\\ \nonumber
&
 -s^{h} \mathcal{F}_n^R(s,T) \sum_{k=1}^{\frac{h-2}{2}} \binom{\frac{h-2}{2}}{k}  (|T|^2-2T_0p)^k T\mathcal{F}_n^L(p,T)p^{h-1-2k}+
 \\ \nonumber
&
-s^{h-1} \mathcal{F}_n^{R}(s,T)\sum_{k=1}^{\frac{h-2}{2}} \binom{\frac{h-2}{2}}{k}  T(|T|^2-2T_0p)^k \mathcal{F}_n^L(p,T)p^{h-2k}
\\ \nonumber
&
+s^{h-1} \mathcal{F}_n^{R}(s,T)\sum_{k=1}^{\frac{h-2}{2}} \binom{\frac{h-2}{2}}{k}  T(|T|^2-2T_0p)^k T\mathcal{F}_n^L(p,T)p^{h-1-2k}
\end{align}\endgroup

}
and
{\small
\begingroup
\allowdisplaybreaks
\begin{align}
   &\mathcal{C}_1(s,p,T):=
\left(\sum_{k=1}^{\frac{h-2}{2}} \binom{\frac{h-2}{2}}{k} s^{h-2k} \mathcal{F}_n^R(s,T) (|T|^2-2T_0s)^k\right)\left(\sum_{k=1}^{\frac{h-2}{2}} \binom{\frac{h-2}{2}}{k}  (|T|^2-2T_0p)^k \mathcal{F}_n^L(p,T) p^{h-2k}\right)
\\ \nonumber
&
-\left(\sum_{k=1}^{\frac{h-2}{2}} \binom{\frac{h-2}{2}}{k} s^{h-2k} \mathcal{F}_n^R(s,T) (|T|^2-2T_0s)^k\right) \left(\sum_{k=1}^{\frac{h-2}{2}} \binom{\frac{h-2}{2}}{k}  (|T|^2-2T_0p)^k T\mathcal{F}_n^L(p,T) p^{h-1-2k}\right)
\\ \nonumber
&
-\left(\sum_{k=1}^{\frac{h-2}{2}} \binom{\frac{h-2}{2}}{k} s^{h-2k-1} \mathcal{F}_n^R(s,T) T(|T|^2-2T_0s)^k\right) \left(\sum_{k=1}^{\frac{h-1}{2}} \binom{\frac{h-2}{2}}{k}  (|T|^2-2T_0p)^k \mathcal{F}_n^L(p,T) p^{h-2k}\right)\\ \nonumber
&
+\left(\sum_{k=1}^{\frac{h-2}{2}} \binom{\frac{h-2}{2}}{k} s^{h-2k-1} \mathcal{F}_n^R(s,T) T(|T|^2-2T_0s)^k\right)\left(\sum_{k=1}^{\frac{h-2}{2}} \binom{\frac{h-2}{2}}{k}  (|T|^2-2T_0p)^k T \mathcal{F}_n^L(p,T) p^{h-1-2k}\right)
\\ \nonumber
&
+
\left(\sum_{k=1}^{\frac{h-2}{2}} \binom{\frac{h-2}{2}}{k} s^{h-2k} \mathcal{F}_n^R(s,T) (|T|^2-2T_0s)^k \right) \mathcal{F}_n^{L}(p,T)p^h
\\ \nonumber
&
-\left(\sum_{k=1}^{\frac{h-2}{2}} \binom{\frac{h-2}{2}}{k} s^{h-2k} \mathcal{F}_n^R(s,T) (|T|^2-2T_0s)^k T \right)\mathcal{F}_n^{L}(p,T)p^{h-1}
\\ \nonumber
&
-\left(\sum_{k=1}^{\frac{h-2}{2}} \binom{\frac{h-2}{2}}{k} s^{h-1-2k} \mathcal{F}_n^R(s,T) T (|T|^2-2T_0s)^k \right) \mathcal{F}_n^{L}(p,T)p^h
\\ \nonumber
&
+\left(\sum_{k=1}^{\frac{h-2}{2}} \binom{\frac{h-2}{2}}{k} s^{h-1-2k} \mathcal{F}_n^R(s,T) T (|T|^2-2T_0s)^k T \right) \mathcal{F}_n^{L}(p,T)p^{h-1}.
\end{align}
\endgroup
}
\end{theorem}
\begin{proof}
Let us begin by writing formula \eqref{resn} as
{\small

\begingroup\allowdisplaybreaks\begin{align}
& \, \, \, \,\mathcal{F}_n^R(s,T)S^{-1}_L(p,T)+S^{-1}_R(s,T)\mathcal{F}_n^L(p,T)+ \gamma_n \biggl[\mathcal{Q}_s^{\frac{h}{2}}(T) S^{-1}_R(s,T)S^{-1}_L(p,T) \mathcal{Q}_p^{\frac{h}{2}}(T)
\\ \nonumber
& + \sum_{i=0, i \neq  \frac{h-2}{2}}^{h-2} \mathcal{Q}^{h-i-1}_s(T) S^{-1}_R(s,T)S^{-1}_L(p,T) \mathcal{Q}_p^{i+1}(T)+ \sum_{i=0}^{h-1}  \mathcal{Q}_s^{h-i}(T) \mathcal{Q}_p^{i+1}(T) \biggl]
\\ \nonumber
&
=\bigl \{ \bigl[\mathcal{F}_n^R(s,T)-\mathcal{F}_n^L(p,T) \bigl]p
- \bar{s}\bigl[\mathcal{F}_n^R(s,T)-\mathcal{F}_n^L(p,T) \bigl] \bigl \}  (p^2-2s_0p+|s|^2)^{-1}.
\end{align}\endgroup

}
Now, we focus on the $\mathcal{Q}_s^{\frac{h}{2}}(T) S^{-1}_R(s,T)S^{-1}_L(p,T) \mathcal{Q}_p^{\frac{h}{2}}(T)$. By definition of left and right S-resolvent operators we get
{\small

\begingroup\allowdisplaybreaks\begin{align}
& \mathcal{Q}_s^{\frac{h}{2}}(T) S^{-1}_R(s,T)S^{-1}_L(p,T) \mathcal{Q}_p^{\frac{h}{2}}(T)=\mathcal{Q}_s^{\frac{h+2}{2}}(T)(s \mathcal{I}-\overline{T})(p \mathcal{I}- \overline{T})\mathcal{Q}_p^{\frac{h+2}{2}}(T)
 \\ \nonumber
&
=s\mathcal{Q}_s^{\frac{h+2}{2}}(T)\mathcal{Q}_p^{\frac{h+2}{2}}(T)p
-s\mathcal{Q}_s^{\frac{h+2}{2}}(T)\overline{T}\mathcal{Q}_p^{\frac{h+2}{2}}(T)
-\mathcal{Q}_s^{\frac{h+2}{2}}(T)\overline{T}\mathcal{Q}_p^{\frac{h+2}{2}}(T)p
+\mathcal{Q}_s^{\frac{h+2}{2}}(T)\overline{T}^2\mathcal{Q}_p^{\frac{h+2}{2}}(T).
\end{align}\endgroup

}
We continue the calculations only on the term $s\mathcal{Q}_s^{\frac{h+2}{2}}(T)\mathcal{Q}_p^{\frac{h+2}{2}}(T)p$.
By the binomial formula we get
{\small

\begingroup\allowdisplaybreaks\begin{align}
&
s \mathcal{Q}_s^{\frac{h+2}{2}}(T) \mathcal{Q}_p^{\frac{h+2}{2}}(T)p=s \mathcal{Q}_s^{h}(T) (s^{2} \mathcal{I}-2sT_0+T \overline{T})^{\frac{h-2}{2}}(p^{2} \mathcal{I}-2pT_0+T \overline{T})^{\frac{h-2}{2}}\mathcal{Q}_p^{h}(T)p
\\ \nonumber
&
=s\mathcal{Q}_s^{h}(T) \left( \sum_{k=0}^{\frac{h-2}{2}} \binom{\frac{h-2}{2}}{k} s^{h-2-2k} (|T|^2-2T_0s)^k\right)\left( \sum_{k=0}^{\frac{h-2}{2}} \binom{\frac{h-2}{2}}{k}  (|T|^2-2T_0p)^k p^{h-2-2k}\right)\mathcal{Q}_p^{h}(T)p
\\ \nonumber
&
=s\mathcal{Q}_s^{h}(T) \left( s^{h-2}+\sum_{k=1}^{\frac{h-2}{2}} \binom{\frac{h-2}{2}}{k} s^{h-2-2k} (|T|^2-2T_0s)^k\right)\times
\\ \nonumber
&
\times\left( \sum_{k=1}^{\frac{h-1}{2}} \binom{\frac{h-1}{2}}{k}  (|T|^2-2T_0p)^k p^{h-2-2k}+p^{h-2}\right)\mathcal{Q}_p^{h}(T)p
\\ \nonumber
&
=s^{h-1} \mathcal{Q}_s^{h}(T)  \mathcal{Q}_p^{h}(T)p^{h-1}+s^{h-1} \mathcal{Q}_s^{h}(T)\left( \sum_{k=1}^{\frac{h-2}{2}} \binom{\frac{h-2}{2}}{k}  (|T|^2-2T_0p)^kp^{h-1-2k} \right)\mathcal{Q}_p^{h}(T)
\\ \nonumber
&
+\mathcal{Q}_s^{h}(T)\left(\sum_{k=1}^{\frac{h-2}{2}} \binom{\frac{h-1}{2}}{k} s^{h-1-2k} (|T|^2-2T_0s)^k\right) \left(\sum_{k=1}^{\frac{h-2}{2}} \binom{\frac{h-2}{2}}{k}  (|T|^2-2T_0p)^k p^{h-1-2k}\right) \mathcal{Q}_p^{h}(T)
\\ \nonumber
& +\mathcal{Q}_s^{h}(T)\left(\sum_{k=1}^{\frac{h-2}{2}} \binom{\frac{h-2}{2}}{k} s^{h-1-2k} (|T|^2-2T_0s)^k \right) \mathcal{Q}_p^h(T)p^{h-1}.
\end{align}\endgroup

}
Now, we use the left and right $ \mathcal{F}$-resolvent equations in Theorem \ref{FRE} and
we go through the computations term by term
{\small
 \begingroup\allowdisplaybreaks\begin{align}
\, \, \, \, \, \,  \, \, \, \, \,s^{h-1} \mathcal{Q}_s^{h}(T)  \mathcal{Q}_p^{h}(T)p^{h-1} &= \gamma_{n}^{-2} \biggl[ s^{h} \mathcal{F}_{n}^{R}(s,T) \mathcal{F}_n^L(p,T) p^h-s^h \mathcal{F}_n^R(s,T)T \mathcal{F}_n^L(p,T) p^{h-1}
\\ \nonumber
\nonumber
& \, \, \, \, \, \, \, \, \, \, \,-s^{h-1} \mathcal{F}_n^R(s,T)T \mathcal{F}_n^L(p,T) p^{h}+s^{h-1}\mathcal{F}_n^R(s,T)T^2 \mathcal{F}_n^L(p,T) p^{h-1}\biggl],
\end{align}\endgroup
}
and
{\small

\begingroup\allowdisplaybreaks\begin{align}
&              s^{h-1} \mathcal{Q}_s^{h}(T)\left( \sum_{k=1}^{\frac{h-2}{2}} \binom{\frac{h-2}{2}}{k}  (|T|^2-2T_0p)^k\mathcal{Q}_p^{h}(T)p^{h-1-2k} \right)
\\ \nonumber
&
= \gamma_n^{-2} \biggl[s^{h} \mathcal{F}_n^R(s,T) \sum_{k=1}^{\frac{h-2}{2}} \binom{\frac{h-2}{2}}{k}  (|T|^2-2T_0p)^k \mathcal{F}_n^L(p,T)p^{h-2k}
\\ \nonumber
&             -s^{h} \mathcal{F}_n^R(s,T) \sum_{k=1}^{\frac{h-2}{2}} \binom{\frac{h-2}{2}}{k}  (|T|^2-2T_0p)^k T\mathcal{F}_n^L(p,T)p^{h-1-2k}
\\ \nonumber
&
-s^{h-1} \mathcal{F}_n^{R}(s,T)\sum_{k=1}^{\frac{h-2}{2}} \binom{\frac{h-2}{2}}{k}  T(|T|^2-2T_0p)^k \mathcal{F}_n^L(p,T)p^{h-2k}
\\ \nonumber
&
+s^{h-1} \mathcal{F}_n^{R}(s,T)\sum_{k=1}^{\frac{h-2}{2}} \binom{\frac{h-2}{2}}{k}  T(|T|^2-2T_0p)^k T\mathcal{F}_n^L(p,T)p^{h-1-2k}\biggl],
\end{align}\endgroup

}
then we consider the term
{\small

\begingroup\allowdisplaybreaks\begin{align}
&
\left(\sum_{k=1}^{\frac{h-2}{2}} \binom{\frac{h-2}{2}}{k} s^{h-1-2k}\mathcal{Q}_s^{h}(T) (|T|^2-2T_0s)^k\right) \left(\sum_{k=1}^{\frac{h-2}{2}} \binom{\frac{h-2}{2}}{k}  (|T|^2-2T_0p)^k \mathcal{Q}_p^{h}(T)p^{h-1-2k}\right) \\ \nonumber
& =\gamma_n^{-2} \biggl[\left(\sum_{k=1}^{\frac{h-2}{2}} \binom{\frac{h-2}{2}}{k} s^{h-2k} \mathcal{F}_n^R(s,T) (|T|^2-2T_0s)^k\right)\left(\sum_{k=1}^{\frac{h-2}{2}} \binom{\frac{h-2}{2}}{k}  (|T|^2-2T_0p)^k \mathcal{F}_n^L(p,T) p^{h-2k}\right)
\\ \nonumber
&-\left(\sum_{k=1}^{\frac{h-2}{2}} \binom{\frac{h-1}{2}}{k} s^{h-2k} \mathcal{F}_n^R(s,T) (|T|^2-2T_0s)^k\right) \left(\sum_{k=1}^{\frac{h-2}{2}} \binom{\frac{h-2}{2}}{k}  (|T|^2-2T_0p)^k T\mathcal{F}_n^L(p,T) p^{h-1-2k}\right)
\\ \nonumber
& -\left(\sum_{k=1}^{\frac{h-2}{2}} \binom{\frac{h-2}{2}}{k} s^{h-2k-1} \mathcal{F}_n^R(s,T) T(|T|^2-2T_0s)^k\right) \left(\sum_{k=1}^{\frac{h-2}{2}} \binom{\frac{h-2}{2}}{k}  (|T|^2-2T_0p)^k \mathcal{F}_n^L(p,T) p^{h-2k}\right)
\\ \nonumber
& +\left(\sum_{k=1}^{\frac{h-2}{2}} \binom{\frac{h-2}{2}}{k} s^{h-2k-1} \mathcal{F}_n^R(s,T) T(|T|^2-2T_0s)^k\right) \left(\sum_{k=1}^{\frac{h-2}{2}} \binom{\frac{h-2}{2}}{k}  (|T|^2-2T_0p)^k T \mathcal{F}_n^L(p,T) p^{h-1-2k}\right)\biggl],
\end{align}\endgroup

}
and the other term
{\small

\begingroup\allowdisplaybreaks\begin{align}
&                         \left(\sum_{k=1}^{\frac{h-2}{2}} \binom{\frac{h-2}{2}}{k} s^{h-1-2k}\mathcal{Q}_s^{h}(T) (|T|^2-2T_0s)^k \right) \mathcal{Q}_p^h(T)p^{h-1}
\\ \nonumber
&
=\gamma_n^{-2} \biggl[\left(\sum_{k=1}^{\frac{h-2}{2}} \binom{\frac{h-2}{2}}{k} s^{h-2k} \mathcal{F}_n^R(s,T) (|T|^2-2T_0s)^k \right) \mathcal{F}_n^{L}(p,T)p^h
\\ \nonumber
&
-\left(\sum_{k=1}^{\frac{h-2}{2}} \binom{\frac{h-2}{2}}{k} s^{h-2k} \mathcal{F}_n^R(s,T) (|T|^2-2T_0s)^k T \right)\mathcal{F}_n^{L}(p,T)p^{h-1}
\\ \nonumber
&
-\left(\sum_{k=1}^{\frac{h-2}{2}} \binom{\frac{h-2}{2}}{k} s^{h-1-2k} \mathcal{F}_n^R(s,T) T (|T|^2-2T_0s)^k \right) \mathcal{F}_n^{L}(p,T)p^h
\\ \nonumber
&
+\left(\sum_{k=1}^{\frac{h-2}{2}} \binom{\frac{h-2}{2}}{k} s^{h-1-2k} \mathcal{F}_n^R(s,T) T (|T|^2-2T_0s)^k T \right) \mathcal{F}_n^{L}(p,T)p^{h-1}\biggl].
\end{align}\endgroup

}
Finally by using the definition of left and right S-resolvent operators we obtain
{\small

\begingroup\allowdisplaybreaks\begin{align}
& \sum_{i=0, i \neq \frac{h-2}{2}}^{h-2} \mathcal{Q}^{h-i-1}_s(T) S^{-1}_R(s,T)S^{-1}_L(p,T) \mathcal{Q}_p^{i+1}(T)=\sum_{i=0, i \neq \frac{h-2}{2}}^{h-2} \mathcal{Q}^{h-i}_s(T)(s \mathcal{I}- \overline{T})(p \mathcal{I}- \overline{T}) \mathcal{Q}_p^{i+2}(T)\\ \nonumber
&=s       \sum_{i=0, i \neq \frac{h-2}{2}}^{h-2} \mathcal{Q}^{h-i}_s(T)\mathcal{Q}_p^{i+2}(T)p-s   \sum_{i=0, i \neq \frac{h-2}{2}}^{h-2} \mathcal{Q}^{h-i}_s(T) \overline{T}\mathcal{Q}_p^{i+2}(T)-\sum_{i=0, i \neq \frac{h-2}{2}}^{h-2} \mathcal{Q}^{h-i}_s(T)\overline{T}\mathcal{Q}_p^{i+2}(T)p
\\ \nonumber
& \, \, \, \, \,+ \sum_{i=0, i \neq \frac{h-2}{2}}^{h-2} \mathcal{Q}^{h-i}_s(T)\overline{T}^2\mathcal{Q}_p^{i+2}(T)
\end{align}\endgroup
}
and this concludes the proof.
\end{proof}

\section{\bf The Riesz projectors for the $\mathcal{F}$-functional calculus for $n=5$}\label{RPROJ5}
\setcounter{equation}{0}

In this section we study the
Riesz projectors for the $\mathcal{F}$-functional calculus for $n=5$. This case is contained in the general result in the next section, but the explicit computation that can be done in this particular case
shows the path for the general case and why we have introduced the pseudo $\mathcal{F}$-resolvent equation.

We recall two preliminary lemmas that will be useful in the sequel and in the next section.
\begin{lemma}
\label{res2}
Let $B \in \mathcal{B}(V_n)$. Let $G$ be a bounded slice Cauchy domain and let $f$ be an intrinsic slice monogenic function whose domain contains $G$. Then for $p \in G$, and for any $I\in\mathbb S$ we have
$$ \frac{1}{2 \pi} \int_{\partial(G \cap \mathbb{C}_I)}f(s) d s_I( \bar{s}B-Bp)(p^{2}-2s_0p+|s|^2)^{-1}=Bf(p).$$
\end{lemma}
We need also the following consequence of the the Stokes's theorem
\begin{lemma}
\label{cau}
Let $f$ and $g$ be left slice monogenic and right slice monogenic functions, respectively, defined on an open set $U$. For any $I \in \mathbb{S}$ and any open bounded set $D_I$ in $U \cap \mathbb{C}_I$ whose boundary is a finite union of continuously differentiable Jordan curves, we have
$$ \int_{\partial D_I} g(s) ds_I f(s)=0.$$
\end{lemma}

In this section we give an answer to \cite[Rem. 5.5]{CG}. Indeed, we prove that due to the equations showed in the previous section the operator defined in \eqref{Rp} are projectors.
  We begin by recalling the definition of projectors.
\begin{definition}
Let $V_n$ be a Banach module and let $P: V_n \to V_n$ be a linear operator. If $P^2=P$ we say that $P$ is a projector.
\end{definition}

In the sequel we need the following lemma, which is based on the monogenic functional calculus.

\begin{lemma}
\label{res32}
Let $ T \in \mathcal{BC}(V_5)$. Suppose that $G$ contains just some points of the $ \mathcal{F}$-spectrum of $T$ and assume that the closed smooth curve $ \partial(G \cap \mathbb{C}_I)$ belongs to the $ \mathcal{F}$-resolvent set of $T$, for every $I \in \mathbb{S}$. Then
{\small
$$ \int_{\partial(G \cap \mathbb{C}_I)}s^{n} ds_I \mathcal{F}_5^{R}(s,T)=0,$$
$$ \int_{\partial(G \cap \mathbb{C}_I)} \mathcal{F}_5^{L}(p,T) d p_I p^n=0,$$
}
for $n=0,1,2,3$.
\end{lemma}
\begin{proof}
Since $ \Delta^2 x^0=\Delta^2 x^1=\Delta^2 x^2=\Delta^2 x^3=0$ we have that
{\small
$$ \int_{\partial(G \cap \mathbb{C}_I)}s^{n} ds_I \mathcal{F}_5^{R}(s,x)=0,$$
$$ \int_{\partial(G \cap \mathbb{C}_I)} \mathcal{F}_5^{L}(p,x) d p_I p^n=0,$$
}
for $n=0,1,2,3$ and for all $x$ such that $x \notin [s]$ if $s \in \partial(G \cap \mathbb{C}_I)$ (respectively, for all $x$ such that $x \notin [p]$ if $p \in \partial(G \cap \mathbb{C}_I)$). We will work only on the case $ \mathcal{F}_5^L(p,x)$, since the other is similar. We recall that $ \mathcal{F}_5^L(p,x)$ is left monogenic in $x$ for every $p$, such that $x \notin [p]$. Therefore we can use the definition of the monogenic functional calculus
{\small
$$ \mathcal{F}_5^L(p,T)= \int_{\partial \Omega} \mathcal{G}_{\omega}(T) \textbf{n}(\omega) \mathcal{F}_5^L(p, \omega) d \mu (\omega),$$
}
where the open set $ \Omega$ contains the monogenic spectrum of $T$, $ \mathcal{G}_{\omega}(T)$ is the monogenic resolvent operator, $ \textbf{n}(\omega)$ is the unit normal vector to $ \partial \Omega$ and $ d \mu(\omega)$ is the surface element. From the Fubini's theorem it follows that
{\small
 \begingroup\allowdisplaybreaks\begin{align}
\int_{\partial(G \cap \mathbb{C}_I)} \mathcal{F}_5^{L}(p,T) d p_I p^n &= \int_{\partial(G \cap \mathbb{C}_I)} \int_{\partial \Omega}\left(\mathcal{G}_{\omega}(T) \textbf{n}(\omega) \mathcal{F}_5^L(p, \omega) d \mu (\omega)\right) p^n dp_I\\ \nonumber
&= \int_{\partial \Omega} \mathcal{G}_{\omega}(T)\textbf{n}(\omega) \left(\int_{\partial(G \cap \mathbb{C}_I)} \mathcal{F}_5^{L}(s,x) p^ndp_I\right)d \mu( \omega)\\ \nonumber
&= 0,
\end{align}\endgroup
}
which concludes the proof.
\end{proof}

\begin{theorem}
Let $ T \in \mathcal{BC}^{0.1}(V_5)$ be such that $T= \sum_{\ell =1}^5 e_i T_i$. Let $ \sigma_{\mathcal{F}}(T)= \sigma_{\mathcal{F},1}(T) \cup \sigma_{\mathcal{F},2}(T)$ with
$$
\hbox{dist} \left(\sigma_{\mathcal{F},1}(T),\sigma_{\mathcal{F},2}(T)\right)>0
$$
and with
$$
\sigma(T_\ell)\subset \mathbb{R}\ \  {\rm for \ all} \ \ell=1,...,5.
$$
Let $G_1$, $G_2$ be two admissible sets for $T$ such that $ \sigma_{\mathcal{F},1}(T) \subset G_1$ and $ \bar{G}_1 \subset G_2$ and such that $dist \left(G_2, \sigma_{\mathcal{F},2}(T) \right)>0$. Then the operator
{\small
$$\check{P}= \frac{1}{\gamma_5 (2 \pi)} \int_{\partial(G_1 \cap \mathbb{C}_I)} \mathcal{F}_5^L(p,T) dp_I p^4=\frac{1}{\gamma_5 (2 \pi)} \int_{\partial(G_2 \cap \mathbb{C}_I)}  s^4 ds_I \mathcal{F}_5^R(s,T) $$
}
is a projector.
\end{theorem}
\begin{proof}
If we multiply the $ \mathcal{F}$- resolvent equation in Theorem \ref{res3} by $s^2$ on the left and $p^2$ on the right we get
{\small
 \begingroup\allowdisplaybreaks\begin{align}
&
  s^2 \mathcal{F}_5^R(s,T)S^{-1}_L(p,T)p^2+s^2 S^{-1}_R(s,T)\mathcal{F}_{5}^L(p,T)p^2
\\ \nonumber
&
+\gamma_5^{-1} \biggl( s^{4} \mathcal{F}_{5}^R(s,T)\mathcal{F}_{5}^L(p,T) p^4- 3s^{4} \mathcal{F}_{5}^R(s,T)T\mathcal{F}_{5}^L(p,T) p^3
-3s^3 \mathcal{F}_{5}^R(s,T)T\mathcal{F}_{5}^L(p,T) p^4
\\ \nonumber
&
+3s^3\mathcal{F}_{5}^R(s,T) T^2\mathcal{F}_{5}^L(p,T)p^3
-2s^3\mathcal{F}_{5}^R(s,T)|T|^2T\mathcal{F}_{5}^L(p,T)p^2
+2s^3 \mathcal{F}_{5}^R(s,T)|T|^2\mathcal{F}_{5}^L(p,T)p^3
 \\ \nonumber
&
-2s^2\mathcal{F}_{5}^R(s,T)|T|^2T\mathcal{F}_{5}^L(p,T)p^3 +s^3 \mathcal{F}_{5}^R(s,T) T^2\mathcal{F}_{5}^L(p,T) p^3+ s^3 \mathcal{F}_{5}^R(s,T)|T|^2 T\mathcal{F}_{5}^L(p,T)p^2
\\ \nonumber
&
+s^2 \mathcal{F}_{5}^R(s,T)|T|^2 T \mathcal{F}_{5}^L(p,T)p^3+ s^2 \mathcal{F}_{5}^R(s,T)|T|^4\mathcal{F}_{5}^L(p,T)p^2+s^3 \mathcal{F}_5^R(s,T) \mathcal{F}_5^L(p,T)p^5
\\ \nonumber
&
-s^2 \mathcal{F}_5^R(s,T) T\mathcal{F}_5^L(p,T)p^5+2s^2 \mathcal{F}_5^R(s,T) T^2\mathcal{F}_5^L(p,T)p^4-s^2\mathcal{F}_5^R(s,T) T^3 \mathcal{F}_5^L(p,T)p^3
\\ \nonumber
&
+2s^2 \mathcal{F}_5^R(s,T) T^2|T|^2 \mathcal{F}_5^L(p,T)p^2+s^5\mathcal{F}_5^R(s,T) \mathcal{F}_5^L(p,T)p^3-s^5\mathcal{F}_5^R(s,T) T\mathcal{F}_5^L(p,T)p^2
\\ \nonumber
&
+2s^4\mathcal{F}_5^R(s,T) T^2 \mathcal{F}_5^L(p,T)p^2
-s^3\mathcal{F}_5^R(s,T) T^3 \mathcal{F}_5^L(p,T)p^2\biggl)= s^2 \bigl \{[\mathcal{F}_5^R(s,T)- \mathcal{F}_5^L(p,T)]p
\\ \nonumber
&
- \bar{s}[\mathcal{F}_5^R(s,T)- \mathcal{F}_5^L(p,T)]\}  (p^2-2s_0 p+|s|^2)^{-1} p^2.
\end{align}\endgroup
}
Now, we multiply the equation by $ds_I$ on the left, integrate it over $ \partial(G_2 \cap \mathbb{C}_I)$ with respect to $ds_I$ and then we multiply it by $ dp_I$ on the right and integrate over $ \partial (G_1 \cap \mathbb{C}_I)$ with respect to $dp_I$, we obtain
{\footnotesize
 \begingroup\allowdisplaybreaks\begin{align}
&   \int_{\partial (G_2 \cap \mathbb{C}_I)} s^2 ds_I \mathcal{F}_5^R(s,T)\int_{\partial (G_1 \cap \mathbb{C}_I)}S^{-1}_L(p,T)dp_Ip^2
+\int_{\partial (G_2 \cap \mathbb{C}_I)} s^2 ds_I S^{-1}_R(s,T)\int_{\partial (G_1 \cap \mathbb{C}_I)}\mathcal{F}_{5}^L(p,T)dp_I p^2
\\
\nonumber
&
+\gamma_5^{-1} \biggl( \int_{\partial (G_2 \cap \mathbb{C}_I)}  s^{4} ds_I \mathcal{F}_{5}^R(s,T)\int_{\partial (G_1 \cap \mathbb{C}_I)} \mathcal{F}_{5}^L(p,T) dp_I p^4
- 3\int_{\partial (G_2 \cap \mathbb{C}_I)} s^{4}ds_I \mathcal{F}_{5}^R(s,T)T \int_{\partial (G_2 \cap \mathbb{C}_I)} \mathcal{F}_{5}^L(p,T) p^3
\\ \nonumber
&
-3\int_{\partial (G_2 \cap \mathbb{C}_I)} s^3 ds_I \mathcal{F}_{5}^R(s,T)T\int_{\partial (G_1 \cap \mathbb{C}_I)} \mathcal{F}_{5}^L(p,T) dp_I p^4
+3 \int_{\partial (G_2 \cap \mathbb{C}_I)} s^3 ds_I \mathcal{F}_{5}^R(s,T) T^2\int_{\partial (G_1 \cap \mathbb{C}_I)} \mathcal{F}_{5}^L(p,T)dp_Ip^3 \\ \nonumber
&
-2\int_{\partial (G_2 \cap \mathbb{C}_I)}s^3ds_I \mathcal{F}_{5}^R(s,T)|T|^2T \int_{\partial (G_1 \cap \mathbb{C}_I)}\mathcal{F}_{5}^L(p,T)dp_Ip^2+2\int_{\partial (G_2 \cap \mathbb{C}_I)}s^3 ds_I\mathcal{F}_{5}^R(s,T)|T|^2 \int_{\partial (G_1 \cap \mathbb{C}_I)}\mathcal{F}_{5}^L(p,T)dp_I p^3
\\ \nonumber
&
-2 \int_{\partial (G_2 \cap \mathbb{C}_I)}s^2ds_I \mathcal{F}_{5}^R(s,T)|T|^2T \int_{\partial (G_1 \cap \mathbb{C}_I)}\mathcal{F}_{5}^L(p,T) dp_I p^3+\int_{\partial (G_2 \cap \mathbb{C}_I)} s^3 ds_I \mathcal{F}_{5}^R(s,T) T^2\int_{\partial (G_1 \cap \mathbb{C}_I)} \mathcal{F}_{5}^L(p,T) dp_I p^3
\\ \nonumber
&
+ \int_{\partial (G_2 \cap \mathbb{C}_I)} s^3 ds_I \mathcal{F}_{5}^R(s,T)|T|^2 T \int_{\partial (G_1 \cap \mathbb{C}_I)} \mathcal{F}_{5}^L(p,T) dp_I p^2+\int_{\partial (G_2 \cap \mathbb{C}_I)} s^2 ds_I \mathcal{F}_{5}^R(s,T)|T|^2 T \int_{\partial (G_1 \cap \mathbb{C}_I)} \mathcal{F}_{5}^L(p,T) dp_Ip^3
\\ \nonumber
&
+ \int_{\partial (G_2 \cap \mathbb{C}_I)} s^2 ds_I \mathcal{F}_{5}^R(s,T)|T|^4\int_{\partial (G_1 \cap \mathbb{C}_I)} \mathcal{F}_{5}^L(p,T)dp_Ip^2+\int_{\partial (G_2 \cap \mathbb{C}_I)}s^3 ds_I \mathcal{F}_5^R(s,T) \int_{\partial (G_1 \cap \mathbb{C}_I)}\mathcal{F}_5^L(p,T) dp_I p^5
\\ \nonumber
&
-\int_{\partial (G_2 \cap \mathbb{C}_I)}s^2 ds_I \mathcal{F}_5^R(s,T) T \int_{\partial (G_1 \cap \mathbb{C}_I)}\mathcal{F}_5^L(p,T) dp_Ip^5+2\int_{\partial (G_2 \cap \mathbb{C}_I)}s^2 ds_I \mathcal{F}_5^R(s,T) T^2\int_{\partial (G_1 \cap \mathbb{C}_I)}\mathcal{F}_5^L(p,T)dp_I p^4
\\ \nonumber
&
-\int_{\partial (G_2 \cap \mathbb{C}_I)}s^2ds_I\mathcal{F}_5^R(s,T) T^3 \int_{\partial (G_1 \cap \mathbb{C}_I)}\mathcal{F}_5^L(p,T)dp_Ip^3+2\int_{\partial (G_2 \cap \mathbb{C}_I)}s^2 ds_I \mathcal{F}_5^R(s,T) T^2|T|^2 \int_{\partial (G_1 \cap \mathbb{C}_I)}\mathcal{F}_5^L(p,T) dp_Ip^2
\\ \nonumber
&
+\int_{\partial (G_2 \cap \mathbb{C}_I)}s^5 ds_I \mathcal{F}_5^R(s,T) \int_{\partial (G_1 \cap \mathbb{C}_I)}\mathcal{F}_5^L(p,T) dp_Ip^3-\int_{\partial (G_2 \cap \mathbb{C}_I)}s^5 ds_I\mathcal{F}_5^R(s,T) T \int_{\partial (G_1 \cap \mathbb{C}_I)}\mathcal{F}_5^L(p,T)dp_Ip^2
\\ \nonumber
&
+2\int_{\partial (G_2 \cap \mathbb{C}_I)}s^4ds_I \mathcal{F}_5^R(s,T) T^2 \int_{\partial (G_1 \cap \mathbb{C}_I)}\mathcal{F}_5^L(p,T)dp_I p^2 -\int_{\partial (G_2 \cap \mathbb{C}_I)}s^3 ds_I \mathcal{F}_5^R(s,T) T^3 \int_{\partial (G_1 \cap \mathbb{C}_I)}\mathcal{F}_5^L(p,T) dp_Ip^2\biggl)
\\ \nonumber
&
= \int_{\partial (G_2 \cap \mathbb{C}_I)} ds_I \int_{\partial (G_1 \cap \mathbb{C}_I)} s^2 \bigl \{[\mathcal{F}_5^R(s,T)- \mathcal{F}_5^L(p,T)]p
- \bar{s}[\mathcal{F}_5^R(s,T)- \mathcal{F}_5^L(p,T)]\}  (p^2-2s_0 p+|s|^2)^{-1}dp_I p^2.
\end{align}\endgroup
}

From Lemma \ref{res32} the expression simplifies to
{\small
 \begingroup\allowdisplaybreaks\begin{align}
& \gamma_5 (2 \pi)^2 \frac{1}{\gamma_5} \left( \frac{1}{2 \pi} \int_{G_2 \cap \mathbb{C}_I} s^4 ds_I \mathcal{F}_5^R(s,T)\right) \frac{1}{\gamma_5}\left( \frac{1}{2 \pi} \int_{G_1 \cap \mathbb{C}_I} \mathcal{F}_5^R(p,T)dp_Ip^4  \right)\\ \nonumber
& =\int_{\partial(G_2 \cap \mathbb{C}_I)}ds_I \int_{\partial(G_1 \cap \mathbb{C}_I)}s^2\bigl \{ [\mathcal{F}_5^R(s,T)- \mathcal{F}_5^L(p,T)]p- \bar{s}[\mathcal{F}_5^R(s,T)- \mathcal{F}_5^L(p,T)]\} (p^2-2s_0 p+|s|^2)^{-1} p^2dp_I.
\end{align}\endgroup
}
By definition of projectors we have
{\small
$$ \frac{(2 \pi)^2}{\gamma_5^{-1}} \check{P}^2=\int_{\partial(G_2 \cap \mathbb{C}_I)}ds_I \int_{\partial(G_1 \cap \mathbb{C}_I)}s^2\bigl \{ [\mathcal{F}_5^R(s,T)- \mathcal{F}_5^L(p,T)]p- \bar{s}[\mathcal{F}_5^R(s,T)- \mathcal{F}_5^L(p,T)]\} (p^2-2s_0 p+|s|^2)^{-1} p^2dp_I.$$
}
Now, we work on the integral on the right hand side. As $ \bar{G}_1 \subset G_2$, for any $s \in \partial(G_2 \cap \mathbb{C}_I)$ the functions
{\small
$$ p \mapsto p(p^2-2s_0p+|s|^2)^{-1}p^2,$$
$$ p \mapsto (p^2-2s_0p+|s|^2)^{-1}p^2$$
}
are slice monogenic on $ \bar{G_1}$. By Lemma \ref{cau} we have
{\small
$$ \int_{\partial (G_1 \cap \mathbb{C}_I)} p(p^2-2s_0p+|s|^2)^{-1}dp_Ip^2=0,$$
$$ \int_{\partial (G_1 \cap \mathbb{C}_I)}(p^2-2s_0p+|s|^2)^{-1}p^2dp_I=0.$$
}
This implies that
{\small
$$\int_{\partial(G_2 \cap \mathbb{C}_I)}ds_I \int_{\partial(G_1 \cap \mathbb{C}_I)} s^2 \mathcal{F}_5^{R}(s,T)p(p^2-2s_0p+|s|^2)^{-1}dp_Ip^2=0$$
}
and
{\small
$$\int_{\partial(G_2 \cap \mathbb{C}_I)}ds_I \int_{\partial(G_1 \cap \mathbb{C}_I)} s^2 \bar{s} \mathcal{F}_5^{R}(s,T)(p^2-2s_0p+|s|^2)^{-1}dp_Ip^2=0,$$
}
from which we deduce
{\small
$$ \frac{(2 \pi)^2}{\gamma_5^{-1}} \check{P}^2=\int_{\partial(G_2 \cap \mathbb{C}_I)}s^2 ds_I \int_{\partial(G_1 \cap \mathbb{C}_I)}[\bar{s} \mathcal{F}_5^L(p,T)- \mathcal{F}_5^L(p,T)p] (p^2-2s_0 p+|s|^2)^{-1} dp_Ip^2.$$
}

From Lemma \ref{res2} with $B=:\mathcal{F}_5^L(p,T)$ and $f(s):=s^2$ we get
{\small
$$ \check{P}^2 = \frac{1}{(2 \pi) \gamma_5} \int_{\partial(G_1 \cap \mathbb{C}_I)} \mathcal{F}_5^L(p,T) dp_I p^4= \check{P}.$$
}
\end{proof}

The above computations can be generalized to study the general case where, however, we have to consider separately the
 case in which the Sce exponent is even or odd.

\section{\bf The Riesz Projectors for the $\mathcal{F}$-functional calculus: the general case of $n$ odd}\label{RPROJ}
\setcounter{equation}{0}

Next result follows as in the case $n=5$, see the proof of Lemma \ref{res32}.
\begin{lemma}
\label{ann}
Let $ T \in \mathcal{BC}(V_n)$ and $ h= \frac{n-1}{2}$. Suppose that $G$ contains just some points of the $ \mathcal{F}$-spectrum of $T$ and assume that the closed smooth curve $ \partial(G \cap \mathbb{C}_I)$ belongs to the $ \mathcal{F}$-resolvent set of $T$, for every $I \in \mathbb{S}$. Then
{\small
$$ \int_{\partial(G \cap \mathbb{C}_I)}s^{m} ds_I \mathcal{F}_n^{R}(s,T)=0,$$
$$ \int_{\partial(G \cap \mathbb{C}_I)} \mathcal{F}_n^{L}(p,T) d p_I p^m=0,$$
}
for all $m \leq 2h-1$.
\end{lemma}
We now prove the following

\begin{theorem}
Let $n>3$ be an odd number and let $ T=\sum_{i=1}^{n} e_i T_i \in \mathcal{BC}^{0,1}(V_n)$. Let $ \sigma_{\mathcal{F}}(T)= \sigma_{\mathcal{F},1}(T) \cup \sigma_{\mathcal{F},2}(T)$ with
$$ \hbox{dist} \left(\sigma_{\mathcal{F},1}(T),\sigma_{\mathcal{F},2}(T)\right)>0,$$
and
$$
\sigma(T_\ell)\subset \mathbb{R}\ \  {\rm for \ all} \ \ell=1,...,n.
$$
Let $G_1$, $G_2$ be two admissible sets for $T$ such that $ \sigma_{\mathcal{F},1}(T) \subset G_1$ and $ \bar{G}_1 \subset G_2$ and such that $dist \left(G_2, \sigma_{\mathcal{F},2}(T) \right)>0$. Then the operator
{\small
\begin{equation}
\label{Rp}
\check{P}= \frac{1}{\gamma_n (2 \pi)} \int_{\partial(G_1 \cap \mathbb{C}_I)} \mathcal{F}_n^L(p,T) dp_I p^{n-1}=\frac{1}{\gamma_n (2 \pi)} \int_{\partial(G_2 \cap \mathbb{C}_I)}  s^{n-1} ds_I \mathcal{F}_n^R(s,T).
\end{equation}
}
is a projector.
\end{theorem}
\begin{proof}
We divide the proof in two cases, according to the parity of $h= \frac{n-1}{2}$.

{\bf CASE I: The Sce exponent $h$ is odd.}

We start by multiplying the equation of Theorem \ref{1T} by $s^h$ on the left and $p^h$ on the right, and since $T_0=0$ we get
{\small
\begin{eqnarray}
\label{eqg1}
& \, \, \, \, s^h\mathcal{F}_n^R(s,T)S^{-1}_L(p,T)p^h+s^hS^{-1}_R(s,T)\mathcal{F}_n^L(p,T)p^h+ \gamma_n \biggl[ s^{h+1} \sum_{i=0}^{h-2} \mathcal{Q}^{h-i}_s(T)\mathcal{Q}_p^{i+2}(T)p^{h+1}+\\ \nonumber
\nonumber
&+s^{h+1} \sum_{i=0}^{h-2} \mathcal{Q}^{h-i}_s(T) T\mathcal{Q}_p^{i+2}(T)p^h +s^h\sum_{i=0}^{h-2} \mathcal{Q}^{h-i}_s(T)T\mathcal{Q}_p^{i+2}(T)p^{h+1}+s^h  \sum_{i=0}^{h-2} \mathcal{Q}^{h-i}_s(T)T^2\mathcal{Q}_p^{i+2}(T)p^h+\\ \nonumber
\nonumber
&+s^h      \sum_{i=0, i \neq \frac{h-1}{2}}^{h-1}  \mathcal{Q}_s^{h-i}(T) \mathcal{Q}_p^{i+1}(T) p^h \biggl]+ \gamma_n^{-1} \biggl[\ s^{2h} \mathcal{F}_{n}^{R}(s,T) \mathcal{F}_n^L(p,T) p^{2h}+s^h\mathcal{A}_0(s,p,T)p^h +s^h\mathcal{B}_0(s,p,T)p^h\\ \nonumber
\nonumber
& +s^h\mathcal{C}_0(s,p,T)p^h \biggl]=s^h\bigl \{ \bigl[\mathcal{F}_n^R(s,T)-\mathcal{F}_n^L(p,T) \bigl]p- \bar{s}\bigl[\mathcal{F}_n^R(s,T)-\mathcal{F}_n^L(p,T) \bigl] \bigl \} (p^2-2s_0p+|s|^2)^{-1} p^h.
\end{eqnarray}
}
Now, we multiply equation \eqref{eqg1} by $ds_I$ on the left, integrate it over $ \partial(G_2 \cap \mathbb{C}_I)$ with respect to $ds_I$ and then we multiply it by $ dp_I$ on the right and integrate over $ \partial (G_1 \cap \mathbb{C}_I)$ with respect to $dp_I$. We obtain
{\small
 \begingroup\allowdisplaybreaks\begin{align}
\label{new1}
&  \int_{\partial(G_2 \cap \mathbb{C}_I)}s^h ds_I\mathcal{F}_n^R(s,T)\int_{\partial(G_1 \cap \mathbb{C}_I)}S^{-1}_L(p,T)dp_Ip^h
+\int_{\partial(G_2 \cap \mathbb{C}_I)}s^hds_IS^{-1}_R(s,T)\int_{\partial(G_1 \cap \mathbb{C}_I)}\mathcal{F}_n^L(p,T)dp_Ip^h
\\ \nonumber
& + \gamma_n \biggl[ \int_{\partial(G_2 \cap \mathbb{C}_I)} s^{h+1} \sum_{i=0}^{h-2} \mathcal{Q}^{h-i}_s(T)ds_I \int_{\partial(G_1 \cap \mathbb{C}_I)}\mathcal{Q}_p^{i+2}(T) dp_I p^{h+1}
\\ \nonumber
&
+\int_{\partial(G_2 \cap \mathbb{C}_I)}s^{h+1} \sum_{i=0}^{h-2} \mathcal{Q}^{h-i}_s(T)ds_I T \int_{\partial(G_1 \cap \mathbb{C}_I)}\mathcal{Q}_p^{i+2}(T)dp_Ip^h
\\ \nonumber
\nonumber
& +\int_{\partial(G_2 \cap \mathbb{C}_I)}s^h \sum_{i=0}^{h-2} \mathcal{Q}^{h-i}_s(T)ds_IT \int_{\partial(G_1 \cap \mathbb{C}_I)}\mathcal{Q}_p^{i+2}(T)dp_Ip^{h+1}
\\ \nonumber
&
+\int_{\partial(G_2 \cap \mathbb{C}_I)}s^h  \sum_{i=0}^{h-2} \mathcal{Q}^{h-i}_s(T)ds_IT^2\int_{\partial(G_1 \cap \mathbb{C}_I)}\mathcal{Q}_p^{i+2}(T)dp_Ip^h
\\
\nonumber
&+\int_{\partial(G_2 \cap \mathbb{C}_I)}s^h\sum_{i=0, i \neq \frac{h-1}{2}}^{h-1}  \mathcal{Q}_s^{h-i}(T) ds_I \int_{\partial(G_1 \cap \mathbb{C}_I)}\mathcal{Q}_p^{i+1}(T) dp_Ip^h \biggl]
\\ \nonumber
&
+ \gamma_n^{-1} \biggl[\int_{\partial(G_2 \cap \mathbb{C}_I)} s^{2h} ds_I\mathcal{F}_{n}^{R}(s,T) \int_{\partial(G_1 \cap \mathbb{C}_I)}\mathcal{F}_n^L(p,T) dp_I p^{2h}+\\ \nonumber
&+\int_{\partial(G_2 \cap \mathbb{C}_I)} \int_{\partial(G_1 \cap \mathbb{C}_I)}s^h ds_I \mathcal{A}_0(s,p,T)dp_Ip^h+s^h ds_I \mathcal{B}_0(s,p,T)dp_Ip^h+s^h ds_I \mathcal{C}_0(s,p,T)dp_Ip^h\biggl]
\\ \nonumber
& = \int_{\partial(G_2 \cap \mathbb{C}_I)} ds_I \int_{\partial(G_1 \cap \mathbb{C}_I)}s^h\bigl \{ \bigl[\mathcal{F}_n^R(s,T)-\mathcal{F}_n^L(p,T) \bigl]p -\bar{s}\bigl[\mathcal{F}_n^R(s,T)-\mathcal{F}_n^L(p,T) \bigl] \bigl \} (p^2-2s_0p+|s|^2)^{-1} dp_I p^h.
\end{align}\endgroup
}
 Recalling the definition of $ \mathcal{A}_0$, $\mathcal{B}_0$, $\mathcal{C}_0$ and the fact that $T_0=0$ we have
 {\small
 \begingroup\allowdisplaybreaks\begin{align}
&                               \int_{\partial(G_2 \cap \mathbb{C}_I)} \int_{\partial(G_1 \cap \mathbb{C}_I)}s^h ds_I \mathcal{A}_0(s,p,T)dp_Ip^h+s^h ds_I \mathcal{B}_0(s,p,T)dp_Ip^h+s^h ds_I \mathcal{C}_0(s,p,T)dp_Ip^h
\\ \nonumber
&
= -\int_{\partial(G_2 \cap \mathbb{C}_I)}s^{2h} \mathcal{F}_n^R(s,T)ds_I T \int_{\partial(G_1 \cap \mathbb{C}_I)}\mathcal{F}_n^L(p,T)dp_I p^{2h-1}
\\ \nonumber
&                              -\int_{\partial(G_2 \cap \mathbb{C}_I)}s^{2h-1} \mathcal{F}_n^R(s,T)ds_IT \int_{\partial(G_1 \cap \mathbb{C}_I)} \mathcal{F}_n^L(p,T)dp_I p^{2h}
\\ \nonumber
&
+\int_{\partial(G_2 \cap \mathbb{C}_I)}s^{2h-1}ds_I\mathcal{F}_n^R(s,T)T^2 \int_{\partial(G_1 \cap \mathbb{C}_I)}\mathcal{F}_n^L(p,T)dp_I p^{2h-1}
\\ \nonumber
&
+\int_{\partial(G_2 \cap \mathbb{C}_I)}s^{2h} ds_I\mathcal{F}_n^R(s,T) \sum_{k=1}^{\frac{h-1}{2}} \binom{\frac{h-1}{2}}{k}  |T|^{2k}\int_{\partial(G_1 \cap \mathbb{C}_I)} \mathcal{F}_n^L(p,T)dp_Ip^{2h-2k}
\\ \nonumber
&
 -\int_{\partial(G_2 \cap \mathbb{C}_I)}s^{2h}ds_I \mathcal{F}_n^R(s,T) \sum_{k=1}^{\frac{h-1}{2}} \binom{\frac{h-1}{2}}{k}  |T|^{2k} T\int_{\partial(G_1 \cap \mathbb{C}_I)}\mathcal{F}_n^L(p,T)p^{2h-1-2k}
 \\ \nonumber
 &
 -\int_{\partial(G_2 \cap \mathbb{C}_I)}s^{2h-1} ds_I \mathcal{F}_n^{R}(s,T)\sum_{k=1}^{\frac{h-1}{2}} \binom{\frac{h-1}{2}}{k}  T|T|^{2k} \int_{\partial(G_1 \cap \mathbb{C}_I)} \mathcal{F}_n^L(p,T)dp_Ip^{2h-2k}
 \\ \nonumber
&
+ \int_{\partial(G_2 \cap \mathbb{C}_I)}s^{2h-1} \mathcal{F}_n^{R}(s,T)ds_I\sum_{k=1}^{\frac{h-1}{2}} \binom{\frac{h-1}{2}}{k}  T|T|^{2k}T \int_{\partial(G_1 \cap \mathbb{C}_I)}\mathcal{F}_n^L(p,T)dp_Ip^{2h-1-2k}
\\ \nonumber
&
+\left(\sum_{k=1}^{\frac{h-1}{2}} \binom{\frac{h-1}{2}}{k}  \int_{\partial(G_2 \cap \mathbb{C}_I)}s^{2h-2k}ds_I \mathcal{F}_n^R(s,T) |T|^{2k}\right)\left(\sum_{k=1}^{\frac{h-1}{2}} \binom{\frac{h-1}{2}}{k}  |T|^{2k}  \int_{\partial(G_1 \cap \mathbb{C}_I)}\mathcal{F}_n^L(p,T) dp_Ip^{2h-2k}\right)
\\ \nonumber
&
-\left(\sum_{k=1}^{\frac{h-1}{2}} \binom{\frac{h-1}{2}}{k}  \int_{\partial(G_2 \cap \mathbb{C}_I)}s^{2h-2k} ds_I\mathcal{F}_n^R(s,T) |T|^{2k}\right) \left(\sum_{k=1}^{\frac{h-1}{2}} \binom{\frac{h-1}{2}}{k} |T|^{k} T  \int_{\partial(G_1 \cap \mathbb{C}_I)}\mathcal{F}_n^L(p,T) dp_Ip^{2h-1-2k}\right)
\\
\nonumber
&
-\left(\sum_{k=1}^{\frac{h-1}{2}} \binom{\frac{h-1}{2}}{k}  \int_{\partial(G_2 \cap \mathbb{C}_I)}s^{2h-2k-1}ds_I\mathcal{F}_n^R(s,T) T|T|^{2k}\right) \left(\sum_{k=1}^{\frac{h-1}{2}} \binom{\frac{h-1}{2}}{k} |T|^{2k}  \int_{\partial(G_1 \cap \mathbb{C}_I)}\mathcal{F}_n^L(p,T) dp_I p^{2h-2k}\right)
\\
\nonumber
&
+\left(\sum_{k=1}^{\frac{h-1}{2}} \binom{\frac{h-1}{2}}{k}  \int_{\partial(G_2 \cap \mathbb{C}_I)}s^{2h-2k-1}ds_I \mathcal{F}_n^R(s,T) T|T|^{2k}\right)\left(\sum_{k=1}^{\frac{h-1}{2}} \binom{\frac{h-1}{2}}{k}  |T|^{2k} T  \int_{\partial(G_1 \cap \mathbb{C}_I)}\mathcal{F}_n^L(p,T)dp_I p^{2h-1-2k}\right)
\\
\nonumber
&
+\left(\sum_{k=1}^{\frac{h-1}{2}} \binom{\frac{h-1}{2}}{k}  \int_{\partial(G_2 \cap \mathbb{C}_I)} s^{2h-2k} ds_I\mathcal{F}_n^R(s,T) |T|^{2k}\right) \int_{\partial(G_1 \cap \mathbb{C}_I)} \mathcal{F}_n^{L}(p,T)dp_I p^{2h}
\\ \nonumber
&
-\left(\sum_{k=1}^{\frac{h-1}{2}} \binom{\frac{h-1}{2}}{k} \int_{\partial(G_2 \cap \mathbb{C}_I)}s^{2h-2k}ds_I \mathcal{F}_n^R(s,T) |T|^{2k} T \right)\int_{\partial(G_1 \cap \mathbb{C}_I)}\mathcal{F}_n^{L}(p,T)dp_Ip^{2h-1}
\\
\nonumber
&
-\left(\sum_{k=1}^{\frac{h-1}{2}} \binom{\frac{h-1}{2}}{k}\int_{\partial(G_2 \cap \mathbb{C}_I)}s^{2h-1-2k} ds_I \mathcal{F}_n^R(s,T) T |T|^{2k} \right) \int_{\partial(G_1 \cap \mathbb{C}_I)}\mathcal{F}_n^{L}(p,T)dp_Ip^{2h}
\\ \nonumber
&
+\left(\sum_{k=1}^{\frac{h-1}{2}} \binom{\frac{h-1}{2}}{k} \int_{\partial(G_2 \cap \mathbb{C}_I)} s^{2h-1-2k} ds_I \mathcal{F}_n^R(s,T) T^2 |T|^{2k}  \right)\int_{\partial(G_1 \cap \mathbb{C}_I)} \mathcal{F}_n^{L}(p,T)dp_I p^{2h-1}.
\end{align}\endgroup
}
Now, since $h \leq 2h-1$ by Lemma \ref{ann} we get
{\small
$$
\int_{\partial(G_2 \cap\mathbb{C}_I)}s^h ds_I\mathcal{F}_n^R(s,T)\int_{\partial(G_1 \cap \mathbb{C}_I)}S^{-1}_L(p,T)dp_Ip^h=\int_{\partial(G_2 \cap \mathbb{C}_I)}s^hds_IS^{-1}_R(s,T)\int_{\partial(G_1 \cap \mathbb{C}_I)}\mathcal{F}_n^L(p,T)dp_Ip^h=0.$$
}
Moreover, since $2h-2k \leq 2h-1$ and $2h-1-2k \leq 2h-1$ we obtain
{\small
$$ \int_{\partial(G_2 \cap \mathbb{C}_I)} \int_{\partial(G_1 \cap \mathbb{C}_I)}s^h ds_I \mathcal{A}_0(s,p,T)dp_Ip^h+s^h ds_I \mathcal{B}_0(s,p,T)dp_Ip^h+s^h ds_I \mathcal{C}_0(s,p,T)dp_Ip^h=0.$$
}
Now, we focus on the term
{\small
$$ \int_{\partial(G_2 \cap \mathbb{C}_I)} s^{h+1} \sum_{i=0}^{h-2} \mathcal{Q}^{h-i}_s(T)ds_I\int_{\partial(G_1 \cap \mathbb{C}_I)}\mathcal{Q}_p^{i+2}(T) dp_I p^{h+1}.$$
}
First of all we split the sum in two parts and write
{\small
$$ \sum_{i=0}^{h-2} \mathcal{Q}_s^{h-i}(T) \mathcal{Q}_p^{i+2}(T)=\sum_{i=0}^{ \lfloor \frac{h-2}{2} \rfloor} \mathcal{Q}_s^{h-i}(T) \mathcal{Q}_p^{i+2}(T)+\sum_{i=\lfloor \frac{h-2}{2} \rfloor +1}^{ h-2} \mathcal{Q}_s^{h-i}(T) \mathcal{Q}_p^{i+2}(T),
$$
}
where $ \lfloor. \rfloor$ is the floor of a number. In the first sum the powers of $ \mathcal{Q}_s(T)$ are more than the powers of $ \mathcal{Q}_p(T)$, and conversely in the second sum.

\medskip
Since $T_0=0$, by the binomial formula we get
{\small
 \begingroup\allowdisplaybreaks\begin{align}
\sum_{i=0}^{ \lfloor \frac{h-2}{2} \rfloor} \mathcal{Q}_s^{h-i}(T) \mathcal{Q}_p^{i+2}(T)&= \mathcal{Q}_s^{h}(T)\sum_{i=0}^{ \lfloor \frac{h-2}{2} \rfloor} \sum_{k=0}^i \binom{i}{k}s^{2k} |T|^{2(i-k)}  \mathcal{Q}_p^{i+2}(T)+\\ \nonumber
&+\sum_{i=\lfloor \frac{h-2}{2} \rfloor +1}^{ h-2}   \mathcal{Q}_s^{h-i}(T) \sum_{k=0}^{h-2-i} \binom{h-2-i}{k} p^{2k} |T|^{2(h-2-i-k)} \mathcal{Q}_p^{h}(T).
\end{align}\endgroup
}
Consider the first sum. By the $ \mathcal{F}$- resolvent equation, see \eqref{eq3} we get
{\small
$$\sum_{i=0}^{ \lfloor \frac{h-2}{2} \rfloor} \sum_{k=0}^i \binom{i}{k}s^{2k}\mathcal{Q}_s^{h}(T) |T|^{2(i-k)} \mathcal{Q}_p^{i+2}(T)=\gamma_n^{-1} \sum_{i=0}^{ \lfloor \frac{h-2}{2} \rfloor} \sum_{k=0}^i \binom{i}{k} s^{2k}\left( s \mathcal{F}_n^R(s,T)- \mathcal{F}_n^R(s,T)T \right) |T|^{2(i-k)} \mathcal{Q}_p^{i+2}(T).$$
}
Hence we have to compute the following integrals
{\small
$$
\gamma_{n}^{-1}\sum_{i=0}^{ \lfloor \frac{h-2}{2} \rfloor} \sum_{k=0}^i \binom{i}{k} \int_{\partial(G_2 \cap \mathbb{C}_I)} s^{h+2+2k}ds_I \mathcal{F}_n^R(s,T)  |T|^{2(i-k)}\int_{\partial(G_1 \cap \mathbb{C}_I)}\mathcal{Q}_p^{i+2}(T) dp_I p^{h+1},
$$
$$
\gamma_{n}^{-1}\sum_{i=0}^{ \lfloor \frac{h-2}{2} \rfloor} \sum_{k=0}^i \binom{i}{k} \int_{\partial(G_2 \cap \mathbb{C}_I)} s^{h+1+2k}ds_I \mathcal{F}_n^R(s,T) T|T|^{2(i-k)} \int_{\partial(G_1 \cap \mathbb{C}_I)}\mathcal{Q}_p^{i+2}(T) dp_I p^{h+1}.
$$
}
Now, since $h$ is odd then we can write $h=2N+1$, with $N \in \mathbb{N}$. This implies that
{\small
$$ h+2+2k \leq 2i+2+2N+1 \leq 2 \lfloor \frac{h-2}{2} \rfloor +2+2N+1=2(N-1)+2+2N+1
=4N+1=2h-1.$$
}
Similarly we get
{\small
$$ h+1+2k \leq 2h-1.$$
}
Therefore by Lemma \ref{ann} we get
{\small
 \begingroup\allowdisplaybreaks\begin{align}
& \gamma_{n}^{-1}\sum_{i=0}^{ \lfloor \frac{h-2}{2} \rfloor} \sum_{k=0}^i \binom{i}{k} \int_{\partial(G_2 \cap \mathbb{C}_I)} s^{h+2+2k}ds_I \mathcal{F}_n^R(s,T) \int_{\partial(G_1 \cap \mathbb{C}_I)}\mathcal{Q}_p^{i+2}(T) dp_I p^{h+1}=0\\ \nonumber
&\gamma_{n}^{-1}\sum_{i=0}^{ \lfloor \frac{h-2}{2} \rfloor} \sum_{k=0}^i \binom{i}{k} \int_{\partial(G_2 \cap \mathbb{C}_I)} s^{h+1+2k}ds_I \mathcal{F}_n^R(s,T) \int_{\partial(G_1 \cap \mathbb{C}_I)}\mathcal{Q}_p^{i+2}(T) dp_I p^{h+1}=0.
\end{align}\endgroup
}
Now, we focus on the second sum. By the $ \mathcal{F}$- resolvent equation, see \eqref{eq2}, we get
{\small
 \begingroup\allowdisplaybreaks\begin{align}
&                \sum_{i=\lfloor \frac{h-2}{2} \rfloor +1}^{ h-2}   \mathcal{Q}_s^{h-i}(T) \sum_{k=0}^{h-2-i} \binom{h-2-i}{k} |T|^{2(h-2-i-k)}\mathcal{Q}_p^{h}(T)p^{2k} \\ \nonumber
&               =\gamma_{n}^{-1} \sum_{i=\lfloor \frac{h-2}{2} \rfloor +1}^{ h-2}   \mathcal{Q}_s^{h-i}(T) \sum_{k=0}^{h-2-i} \binom{h-2-i}{k} |T|^{2(h-2-i-k)}\left( \mathcal{F}_{n}^R(p,T)p- T\mathcal{F}_{n}^R(p,T) \right) p^{2k} .
\end{align}\endgroup
}
Hence we have to compute the following integrals
{\small
$$ \gamma_n^{-1}\sum_{i=\lfloor \frac{h-2}{2} \rfloor +1}^{ h-2} \sum_{k=0}^{2+i-h} \binom{h-2-i}{k} \int_{\partial(G_2 \cap \mathbb{C}_I)} s^{h+1}ds_I\mathcal{Q}_s^{h-i}(T)  |T|^{2(h-2-i-k)}\int_{\partial(G_1 \cap \mathbb{C}_I)}\mathcal{F}_{n}^R(p,T)dp_Ip^{h+2k+2} ,
$$
}
{\small
$$ \gamma_n^{-1}\sum_{i=\lfloor \frac{h-2}{2} \rfloor +1}^{ h-2} \sum_{k=0}^{h-2-i} \binom{h-2-i}{k} \int_{\partial(G_2 \cap \mathbb{C}_I)} s^{h+1}ds_I\mathcal{Q}_s^{h-i}(T)  |T|^{2(h-2-i-k)}T\int_{\partial(G_1 \cap \mathbb{C}_I)}\mathcal{F}_{n}^R(p,T)dp_Ip^{h+2k+1}.$$
}
Since $h=2N+1$, with $N \in \mathbb{N}$ we get
{\small
$$ 2k+2+h \leq 2(h-2-i)+2+h= 2h-4-2i+2+h \leq 4N+2-4-2 \left( \lfloor \frac{h-2}{2} \rfloor+1 \right)+2+2N+1=4N+1=2h-1,$$
}
and similarly
{\small
$$ 2k+1+h \leq 2h-1,$$
}
together with Lemma \ref{ann} we get
{\small
 \begingroup\allowdisplaybreaks\begin{align}
&\gamma_n^{-1}\sum_{i=\lfloor \frac{h-2}{2} \rfloor +1}^{ h-2} \sum_{k=0}^{h-2-i} \binom{h-2-i}{k}  \int_{\partial(G_2 \cap \mathbb{C}_I)} s^{h+1}ds_I\mathcal{Q}_s^{h-i}(T)  |T|^{2(h-2-i-k)}\int_{\partial(G_1 \cap \mathbb{C}_I)}\mathcal{F}_{n}^R(p,T)dp_Ip^{h+2k+2}=0\\ \nonumber
&\gamma_n^{-1}\sum_{i=\lfloor \frac{h-2}{2} \rfloor +1}^{ h-2} \sum_{k=0}^{h-2-i} \binom{h-2-i}{k}  \int_{\partial(G_2 \cap \mathbb{C}_I)} s^{h+1}ds_I\mathcal{Q}_s^{h-i}(T)  |T|^{2(h-2-i-k)}T\int_{\partial(G_1 \cap \mathbb{C}_I)}\mathcal{F}_{n}^R(p,T)dp_Ip^{h+2k+1}=0.
\end{align}\endgroup
}
Similar arguments applied to the other members of \eqref{new1} lead to
{\small
 \begingroup\allowdisplaybreaks\begin{align}
&
\gamma_n^{-1} \int_{\partial(G_2 \cap \mathbb{C}_I)} s^{2h} ds_I\mathcal{F}_{n}^{R}(s,T) \int_{\partial(G_1 \cap \mathbb{C}_I)}\mathcal{F}_n^L(p,T) dp_I p^{2h}
\\ \nonumber
&
=\int_{\partial(G_2 \cap \mathbb{C}_I)} ds_I \int_{\partial(G_1 \cap \mathbb{C}_I)}s^h\bigl \{ \bigl[\mathcal{F}_n^R(s,T)-\mathcal{F}_n^L(p,T) \bigl]p -\bar{s}\bigl[\mathcal{F}_n^R(s,T)-\mathcal{F}_n^L(p,T) \bigl] \bigl \} (p^2-2s_0p+|s|^2)^{-1} dp_I p^h.
\end{align}\endgroup
}
Since $h= \frac{n-1}{2}$, by formula \eqref{Rp} we get
{\small
$$\frac{(2 \pi)^{2}}{\gamma_{n}^{-1}} \check{P}^2=\int_{\partial(G_2 \cap \mathbb{C}_I)} ds_I \int_{\partial(G_1 \cap \mathbb{C}_I)}s^h\bigl \{ \bigl[\mathcal{F}_n^R(s,T)-\mathcal{F}_n^L(p,T) \bigl]p-\bar{s}\bigl[\mathcal{F}_n^R(s,T)-\mathcal{F}_n^L(p,T) \bigl] \bigl \} (p^2-2s_0p+|s|^2)^{-1}  p^hdp_I.$$
}
Now, we work on the integral on the right hand side. As $ \bar{G}_1 \subset G_2$, for any $s \in \partial(G_2 \cap \mathbb{C}_I)$ the functions
{\small
$$ p \mapsto p(p^2-2s_0p+|s|^2)^{-1}p^h,$$
$$ p \mapsto (p^2-2s_0p+|s|^2)^{-1}p^h$$
}
are slice monogenic on $ \bar{G_1}$. By Lemma \ref{cau} we have
{\small
$$ \int_{\partial (G_1 \cap \mathbb{C}_I)} p(p^2-2s_0p+|s|^2)^{-1}dp_Ip^h=0,$$
$$ \int_{\partial (G_1 \cap \mathbb{C}_I)}(p^2-2s_0p+|s|^2)^{-1}p^hdp_I=0.$$
}
This implies that
$$\int_{\partial(G_2 \cap \mathbb{C}_I)}ds_I \int_{\partial(G_1 \cap \mathbb{C}_I)} s^h \mathcal{F}_n^{R}(s,T)p(p^2-2s_0p+|s|^2)^{-1}dp_Ip^h=0$$
and
$$\int_{\partial(G_2 \cap \mathbb{C}_I)}ds_I \int_{\partial(G_1 \cap \mathbb{C}_I)} s^h \bar{s} \mathcal{F}_n^{R}(s,T)(p^2-2s_0p+|s|^2)^{-1}dp_Ip^h=0.$$

Then we have
$$ \frac{(2 \pi)^2}{\gamma_n^{-1}} \check{P}^2=\int_{\partial(G_2 \cap \mathbb{C}_I)}s^h ds_I \int_{\partial(G_1 \cap \mathbb{C}_I)}[\bar{s} \mathcal{F}_n^L(p,T)- \mathcal{F}_7^L(p,T)p] (p^2-2s_0 p+|s|^2)^{-1} dp_Ip^h.$$

From Lemma \ref{res2} with $B=:\mathcal{F}_n^L(p,T)$ and $f(s):=s^h$ we get
$$ \check{P}^2 = \frac{1}{(2 \pi) \gamma_n} \int_{\partial(G_1 \cap \mathbb{C}_I)} \mathcal{F}_n^L(p,T) dp_I p^{2h}= \check{P}.
$$

{\bf CASE II: The Sce exponent $h$ is even.}

We multiply the equation of Theorem \ref{2T} by $s^h$ left and $p^h$ on the right, and since $T_0=0$ we get

{\small
 \begingroup\allowdisplaybreaks\begin{align}
& s^h\mathcal{F}_n^R(s,T)S^{-1}_L(p,T)p^h+s^hS^{-1}_R(s,T)\mathcal{F}_n^L(p,T)p^h
\\ \nonumber
&
+ \gamma_n \biggl[s^{h+1}\mathcal{Q}_s^{\frac{h+2}{2}}(T)T\mathcal{Q}_p^{\frac{h+2}{2}}(T)p^h
+s^h\mathcal{Q}_s^{\frac{h+2}{2}}(T)T\mathcal{Q}_p^{\frac{h+2}{2}}(T)p^{h+1}
\\
\nonumber
&
+s^h\mathcal{Q}_s^{\frac{h+2}{2}}(T)T^2\mathcal{Q}_p^{\frac{h+2}{2}}(T)p^h+  s^h\sum_{i=0}^{h-1}  \mathcal{Q}_s^{h-i}(T) \mathcal{Q}_p^{i+1}(T)p^h +s^{h+1}
\sum_{i=0, i \neq \frac{h-2}{2}}^{h-2} \mathcal{Q}^{h-i}_s(T)\mathcal{Q}_p^{i+2}(T)p^{h+1}
\\
\nonumber
&
+s^{h+1}        \sum_{i=0, i \neq \frac{h-2}{2}}^{h-2} \mathcal{Q}^{h-i}_s(T) T\mathcal{Q}_p^{i+2}(T)p^{h}+s^h        \sum_{i=0, i \neq \frac{h-2}{2}}^{h-2} \mathcal{Q}^{h-i}_s(T)\mathcal{Q}_p^{i+2}(T)p^{h+1}
\\ \nonumber
&
+ s^h        \sum_{i=0, i \neq \frac{h-2}{2}}^{h-2} \mathcal{Q}^{h-i}_s(T)T^2\mathcal{Q}_p^{i+2}(T) p^{h} \biggl]
\\
\nonumber
&
+ \gamma_n^{-1}\left [s^h\mathcal{A}_1(s,p,T)p^h+ s^h\mathcal{B}_1(s,p,T)p^h+s^h\mathcal{C}_1(s,p,T)p^h+s^{2h} \mathcal{F}_{n}^{R}(s,T) \mathcal{F}_n^L(p,T) p^{2h}\right]
\\
\nonumber
&
=s^h\bigl \{ \bigl[\mathcal{F}_n^R(s,T)-\mathcal{F}_n^L(p,T) \bigl]p
- \bar{s}\bigl[\mathcal{F}_n^R(s,T)-\mathcal{F}_n^L(p,T) \bigl] (p^2-2s_0p+|s|^2)^{-1} p^h.
\end{align}\endgroup
}
Now, we multiply  by $ds_I$ on the left, integrate it over $ \partial(G_2 \cap \mathbb{C}_I)$ with respect to $ds_I$ and then we multiply it by $ dp_I$ on the right and integrate over $ \partial (G_1 \cap \mathbb{C}_I)$ with respect to $dp_I$, and we obtain
{\small

\begingroup\allowdisplaybreaks\begin{align}
&                               \int_{\partial (G_2 \cap \mathbb{C}_I)}s^h\mathcal{F}_n^R(s,T)ds_I \int_{\partial (G_1 \cap \mathbb{C}_I)}S^{-1}_L(p,T)dp_Ip^h
\\ \nonumber
&
+\int_{\partial (G_2 \cap \mathbb{C}_I)}s^h ds_I S^{-1}_R(s,T)\int_{\partial (G_1 \cap \mathbb{C}_I)}\mathcal{F}_n^L(p,T)p^h
+
\\ \nonumber
&
+\gamma_n \biggl[\int_{\partial (G_2 \cap \mathbb{C}_I)}s^{h+1}ds_I\mathcal{Q}_s^{\frac{h+2}{2}}(T)T\int_{\partial (G_1 \cap \mathbb{C}_I)}\mathcal{Q}_p^{\frac{h+2}{2}}(T)dp_Ip^h
\\ \nonumber
&
+\int_{\partial (G_2 \cap \mathbb{C}_I)}s^h ds_I\mathcal{Q}_s^{\frac{h+2}{2}}(T)T\int_{\partial (G_1 \cap \mathbb{C}_I)}\mathcal{Q}_p^{\frac{h+2}{2}}(T)dp_Ip^{h+1}+
\\ \nonumber
&
+\int_{\partial (G_2 \cap \mathbb{C}_I)}s^h ds_I\mathcal{Q}_s^{\frac{h+2}{2}}(T)T^2\int_{\partial (G_1 \cap \mathbb{C}_I)}\mathcal{Q}_p^{\frac{h+2}{2}}(T)dp_Ip^h
\\ \nonumber
&
+  \int_{\partial (G_2 \cap \mathbb{C}_I)}s^h ds_I \sum_{i=0}^{h-1}  \mathcal{Q}_s^{h-i}(T) \int_{\partial (G_1 \cap \mathbb{C}_I)}\mathcal{Q}_p^{i+1}(T)dp_Ip^h +
\\ \nonumber
&
+\int_{\partial (G_2 \cap \mathbb{C}_I)}s^{h+1} ds_I\sum_{i=0, i \neq \frac{h-2}{2}}^{h-2} \mathcal{Q}^{h-i}_s(T)\int_{\partial (G_2 \cap \mathbb{C}_I)}\mathcal{Q}_p^{i+2}(T)dp_I p^{h+1}
\\ \nonumber
&
+\int_{\partial (G_2 \cap \mathbb{C}_I)}s^{h+1} \sum_{i=0, i \neq \frac{h-2}{2}}^{h-2} \mathcal{Q}^{h-i}_s(T)T\int_{\partial (G_1 \cap \mathbb{C}_I)} \mathcal{Q}_p^{i+2}(T)dp_Ip^{h}
\\ \nonumber
&
+\int_{\partial (G_2 \cap \mathbb{C}_I)}s^h ds_I\sum_{i=0, i \neq \frac{h-2}{2}}^{h-2} \mathcal{Q}^{h-i}_s(T)\int_{\partial (G_1 \cap \mathbb{C}_I)}\mathcal{Q}_p^{i+2}(T)dp_I p^{h+1}
\\ \nonumber
&
+ \int_{\partial (G_2 \cap \mathbb{C}_I)}s^h ds_I\sum_{i=0, i \neq \frac{h-2}{2}}^{h-2} \mathcal{Q}^{h-i}_s(T)T^2\int_{\partial (G_1 \cap \mathbb{C}_I)}\mathcal{Q}_p^{i+2}(T)dp_I p^{h} \biggl]+
\\ \nonumber
&
+ \gamma_n^{-1}\Big[\int_{\partial (G_2 \cap \mathbb{C}_I)}\int_{\partial (G_1 \cap \mathbb{C}_I)}s^h ds_I\mathcal{A}_1(s,p,T)dp_Ip^h+s^h ds_I\mathcal{B}_1(s,p,T)dp_Ip^h+s^h ds_I\mathcal{C}_1(s,p,T)dp_Ip^h
\\ \nonumber
&
+\int_{\partial (G_2 \cap \mathbb{C}_I)}s^{2h} ds_I\mathcal{F}_{n}^{R}(s,T) \int_{\partial (G_1 \cap \mathbb{C}_I)}\mathcal{F}_n^L(p,T)dp_I p^{2h}\Big]
\\ \nonumber
&
=\int_{\partial (G_2 \cap \mathbb{C}_I)}ds_I \int_{\partial (G_1 \cap \mathbb{C}_I)}s^h\bigl \{ \bigl[\mathcal{F}_n^R(s,T)-\mathcal{F}_n^L(p,T) \bigl]p
- \bar{s}\bigl[\mathcal{F}_n^R(s,T)-\mathcal{F}_n^L(p,T) \bigl]  (p^2-2s_0p+|s|^2)^{-1} dp_Ip^h.
\end{align}\endgroup
}
From the definition of $ \mathcal{A}_1$, $ \mathcal{B}_1$, $\mathcal{C}_1$ and recalling that $T_0=0$ we have
{\small
 \begingroup\allowdisplaybreaks\begin{align}
&
\int_{\partial(G_2 \cap \mathbb{C}_I)} \int_{\partial(G_1 \cap \mathbb{C}_I)}s^h ds_I \mathcal{A}_1(s,p,T)dp_Ip^h+s^h ds_I \mathcal{B}_1(s,p,T)dp_Ip^h+s^h ds_I \mathcal{C}_1(s,p,T)dp_Ip^h
\\ \nonumber
&
= -\int_{\partial(G_2 \cap \mathbb{C}_I)}s^{2h} \mathcal{F}_n^R(s,T)ds_I T \int_{\partial(G_1 \cap \mathbb{C}_I)}\mathcal{F}_n^L(p,T)dp_I p^{2h-1}+
\\ \nonumber
&
-\int_{\partial(G_2 \cap \mathbb{C}_I)}s^{2h-1} \mathcal{F}_n^R(s,T)ds_IT \int_{\partial(G_1 \cap \mathbb{C}_I)} \mathcal{F}_n^L(p,T)dp_I p^{2h}
\\ \nonumber
&
+\int_{\partial(G_2 \cap \mathbb{C}_I)}s^{2h-1}ds_I\mathcal{F}_n^R(s,T)T^2 \int_{\partial(G_1 \cap \mathbb{C}_I)}\mathcal{F}_n^L(p,T)dp_I p^{2h-1}+
\\ \nonumber
&
+\int_{\partial(G_2 \cap \mathbb{C}_I)}s^{2h} ds_I\mathcal{F}_n^R(s,T) \sum_{k=1}^{\frac{h-2}{2}} \binom{\frac{h-2}{2}}{k}  |T|^{2k}\int_{\partial(G_1 \cap \mathbb{C}_I)} \mathcal{F}_n^L(p,T)dp_Ip^{2h-2k}
\\ \nonumber
&
-\int_{\partial(G_2 \cap \mathbb{C}_I)}s^{2h}ds_I \mathcal{F}_n^R(s,T) \sum_{k=1}^{\frac{h-2}{2}} \binom{\frac{h-2}{2}}{k} T |T|^{2k} \int_{\partial(G_1 \cap \mathbb{C}_I)}\mathcal{F}_n^L(p,T)p^{2h-1-2k}
\\ \nonumber
&
-\int_{\partial(G_2 \cap \mathbb{C}_I)}s^{2h-1} ds_I \mathcal{F}_n^{R}(s,T)\sum_{k=1}^{\frac{h-2}{2}} \binom{\frac{h-2}{2}}{k}  T|T|^{2k} \int_{\partial(G_1 \cap \mathbb{C}_I)} \mathcal{F}_n^L(p,T)dp_Ip^{2h-2k}
\\ \nonumber
&
 + \int_{\partial(G_2 \cap \mathbb{C}_I)}s^{2h-1} \mathcal{F}_n^{R}(s,T)ds_I\sum_{k=1}^{\frac{h-2}{2}} \binom{\frac{h-2}{2}}{k}  T^2|T|^{2k} \int_{\partial(G_1 \cap \mathbb{C}_I)}\mathcal{F}_n^L(p,T)dp_Ip^{2h-1-2k}
 +\\ \nonumber
&
+\left(\sum_{k=1}^{\frac{h-2}{2}} \binom{\frac{h-2}{2}}{k}  \int_{\partial(G_2 \cap \mathbb{C}_I)}s^{2h-2k}ds_I \mathcal{F}_n^R(s,T) |T|^{2k}\right)\left(\sum_{k=1}^{\frac{h-2}{2}} \binom{\frac{h-2}{2}}{k}  |T|^{2k}  \int_{\partial(G_1 \cap \mathbb{C}_I)}\mathcal{F}_n^L(p,T) dp_Ip^{2h-2k}\right)
\\ \nonumber
&                              -\left(\sum_{k=1}^{\frac{h-2}{2}} \binom{\frac{h-2}{2}}{k}  \int_{\partial(G_2 \cap \mathbb{C}_I)}s^{2h-2k} ds_I\mathcal{F}_n^R(s,T) |T|^{2k}\right) \left(\sum_{k=1}^{\frac{h-2}{2}} \binom{\frac{h-2}{2}}{k} |T|^{k} T  \int_{\partial(G_1 \cap \mathbb{C}_I)}\mathcal{F}_n^L(p,T) dp_Ip^{2h-1-2k}\right)
\\ \nonumber
&                               -\left(\sum_{k=1}^{\frac{h-2}{2}} \binom{\frac{h-2}{2}}{k}  \int_{\partial(G_2 \cap \mathbb{C}_I)}s^{2h-2k-1}ds_I\mathcal{F}_n^R(s,T) T|T|^{2k}\right) \left(\sum_{k=1}^{\frac{h-2}{2}} \binom{\frac{h-2}{2}}{k} |T|^{2k}  \int_{\partial(G_1 \cap \mathbb{C}_I)}\mathcal{F}_n^L(p,T) dp_I p^{2h-2k}\right)
\\ \nonumber
&
+\left(\sum_{k=1}^{\frac{h-2}{2}} \binom{\frac{h-2}{2}}{k}  \int_{\partial(G_2 \cap \mathbb{C}_I)}s^{2h-2k-1}ds_I \mathcal{F}_n^R(s,T) T|T|^{2k}\right)\left(\sum_{k=1}^{\frac{h-2}{2}} \binom{\frac{h-2}{2}}{k}  |T|^{2k} T  \int_{\partial(G_1 \cap \mathbb{C}_I)}\mathcal{F}_n^L(p,T)dp_I p^{2h-1-2k}\right)\\ \nonumber
& +\left(\sum_{k=1}^{\frac{h-2}{2}} \binom{\frac{h-2}{2}}{k}  \int_{\partial(G_2 \cap \mathbb{C}_I)} s^{2h-2k} ds_I\mathcal{F}_n^R(s,T) |T|^{2k}\right) \int_{\partial(G_1 \cap \mathbb{C}_I)} \mathcal{F}_n^{L}(p,T)dp_I p^{2h}
\\ \nonumber
&
 -\left(\sum_{k=1}^{\frac{h-2}{2}} \binom{\frac{h-2}{2}}{k} \int_{\partial(G_2 \cap \mathbb{C}_I)}s^{2h-2k}ds_I \mathcal{F}_n^R(s,T) |T|^{2k} T \right)\int_{\partial(G_1 \cap \mathbb{C}_I)}\mathcal{F}_n^{L}(p,T)dp_Ip^{2h-1}
\\ \nonumber
&                              -\left(\sum_{k=1}^{\frac{h-2}{2}} \binom{\frac{h-2}{2}}{k}\int_{\partial(G_2 \cap \mathbb{C}_I)}s^{2h-1-2k} ds_I \mathcal{F}_n^R(s,T) T |T|^{2k} \right) \int_{\partial(G_1 \cap \mathbb{C}_I)}\mathcal{F}_n^{L}(p,T)dp_Ip^{2h}
\\ \nonumber
&                              +\left(\sum_{k=1}^{\frac{h-2}{2}} \binom{\frac{h-2}{2}}{k} \int_{\partial(G_2 \cap \mathbb{C}_I)} s^{2h-1-2k} ds_I \mathcal{F}_n^R(s,T) T^2 |T|^{2k}  \right)\int_{\partial(G_1 \cap \mathbb{C}_I)} \mathcal{F}_n^{L}(p,T)dp_I p^{2h-1}.
\end{align}\endgroup
}
Now we observe that since $h \leq 2h-1$ by Lemma \ref{ann} we have
{\small
$$
\int_{\partial(G_2 \cap\mathbb{C}_I)}s^h ds_I\mathcal{F}_n^R(s,T)\int_{\partial(G_1 \cap \mathbb{C}_I)}S^{-1}_L(p,T)dp_Ip^h=\int_{\partial(G_2 \cap \mathbb{C}_I)}s^hds_IS^{-1}_R(s,T)\int_{\partial(G_1 \cap \mathbb{C}_I)}\mathcal{F}_n^L(p,T)dp_Ip^h=0.$$
}
Moreover, since $2h-2k \leq 2h-1$ and $2h-1-2k \leq 2h-1$ we get
{\small
$$ \int_{\partial(G_2 \cap \mathbb{C}_I)} \int_{\partial(G_1 \cap \mathbb{C}_I)}s^h ds_I \mathcal{A}_1(s,p,T)dp_Ip^h+s^h ds_I \mathcal{B}_1(s,p,T)dp_Ip^h+s^h ds_I \mathcal{C}_1(s,p,T)dp_Ip^h=0.$$
}
Now, we focus on computing the integral
{\small
$$\int_{\partial (G_2 \cap \mathbb{C}_I)}s^{h+1}ds_I\mathcal{Q}_s^{\frac{h+2}{2}}(T)T\int_{\partial (G_1 \cap \mathbb{C}_I)}\mathcal{Q}_p^{\frac{h+2}{2}}(T)dp_Ip^h.$$
}
By the binomial formula and recalling that $T_0=0$ we get
{\small
$$ \mathcal{Q}_p^{\frac{h+2}{2}}(T)=\mathcal{Q}_p^{\frac{2-h}{2}}(T)\mathcal{Q}_p^{\frac{h-2}{2}}(T)\mathcal{Q}_p^{\frac{h+2}{2}}(T)= (p^2+|T|^2)^{\frac{h-2}{2}}\mathcal{Q}_p^{h}(T)= \sum_{k=0}^{\frac{h-2}{2}} \binom{\frac{h-2}{2}}{k} |T|^{2 \left( \frac{h-2}{2}-k \right)}\mathcal{Q}_{p}^h(T) p^{2k} .$$
}
By the $ \mathcal{F}$-resolvent, see \eqref{eq2}, we deduce that
{\small
\begingroup
\allowdisplaybreaks
\begin{align}
\label{int1}
&\int_{\partial (G_2 \cap \mathbb{C}_I)}s^{h+1}ds_I\mathcal{Q}_s^{\frac{h+2}{2}}(T)T\int_{\partial (G_1 \cap \mathbb{C}_I)}\mathcal{Q}_p^{\frac{h+2}{2}}(T)dp_Ip^h
\\ \nonumber
\nonumber
&=\int_{\partial (G_2 \cap \mathbb{C}_I)}s^{h+1}ds_I\mathcal{Q}_s^{\frac{h+2}{2}}(T)T \sum_{k=0}^{\frac{h-2}{2}} \binom{\frac{h-2}{2}}{k} |T|^{2 \left( \frac{h-2}{2}-k \right)}  \int_{\partial (G_1 \cap \mathbb{C}_I)}\mathcal{Q}_p^{h}(T)dp_Ip^{2k+h}\\ \nonumber
\nonumber
&= \int_{\partial (G_2 \cap \mathbb{C}_I)}s^{h+1}ds_I\mathcal{Q}_s^{\frac{h+2}{2}}(T)T \sum_{k=0}^{\frac{h-2}{2}} \binom{\frac{h-2}{2}}{k} |T|^{2 \left( \frac{h-2}{2}-k \right)}  \int_{\partial (G_1 \cap \mathbb{C}_I)} \mathcal{F}_n^L(p,T) dp_I p^{2k+h+1}
\\ \nonumber
\nonumber
&-\int_{\partial (G_2 \cap \mathbb{C}_I)}s^{h+1}ds_I\mathcal{Q}_s^{\frac{h+2}{2}}(T)T^2 \sum_{k=0}^{\frac{h-2}{2}} \binom{\frac{h-2}{2}}{k} |T|^{2 \left( \frac{h-2}{2}-k \right)}  \int_{\partial (G_1 \cap \mathbb{C}_I)} \mathcal{F}_n^L(p,T) dp_I p^{2k+h}.
\end{align}
\endgroup
}
We observe that
{\small
$$ h+1+2k \leq h+1+h-2=2h-1,$$
}
similarly we have $h+2k \leq 2h-1$, so formula \eqref{int1} together with Lemma \ref{ann} imply that
{\small
$$\int_{\partial (G_2 \cap \mathbb{C}_I)}s^{h+1}ds_I\mathcal{Q}_s^{\frac{h+2}{2}}(T)T\int_{\partial (G_1 \cap \mathbb{C}_I)}\mathcal{Q}_p^{\frac{h+2}{2}}(T)dp_Ip^h=0.$$
}
Using similar arguments, we obtain
{\small
 \begingroup\allowdisplaybreaks\begin{align}
&\int_{\partial (G_2 \cap \mathbb{C}_I)}s^h ds_I\mathcal{Q}_s^{\frac{h+2}{2}}(T)T\int_{\partial (G_1 \cap \mathbb{C}_I)}\mathcal{Q}_p^{\frac{h+2}{2}}(T)dp_Ip^{h+1}=0\\ \nonumber
& \int_{\partial (G_2 \cap \mathbb{C}_I)}s^h ds_I\mathcal{Q}_s^{\frac{h+2}{2}}(T)T^2\int_{\partial (G_1 \cap \mathbb{C}_I)}\mathcal{Q}_p^{\frac{h+2}{2}}(T)dp_Ip^h=0.
\end{align}\endgroup
}
By similar computations made when $h$ is odd we get
{\small
$$\int_{\partial (G_2 \cap \mathbb{C}_I)}s^h ds_I \sum_{i=0}^{h-1}  \mathcal{Q}_s^{h-i}(T) \int_{\partial (G_1 \cap \mathbb{C}_I)}\mathcal{Q}_p^{i+1}(T)dp_Ip^h=0,$$
}
{\small
$$\int_{\partial (G_2 \cap \mathbb{C}_I)}s^{h+1} ds_I\sum_{i=0, i \neq \frac{h-2}{2}}^{h-2} \mathcal{Q}^{h-i}_s(T)\int_{\partial (G_2 \cap \mathbb{C}_I)}\mathcal{Q}_p^{i+2}(T)dp_I p^{h+1}=0,$$
$$ \int_{\partial (G_2 \cap \mathbb{C}_I)}s^{h+1} \sum_{i=0, i \neq \frac{h-2}{2}}^{h-2} \mathcal{Q}^{h-i}_s(T)T\int_{\partial (G_1 \cap \mathbb{C}_I)} \mathcal{Q}_p^{i+2}(T)dp_Ip^{h}=0,$$
$$\int_{\partial (G_2 \cap \mathbb{C}_I)}s^h ds_I\sum_{i=0, i \neq \frac{h-2}{2}}^{h-2} \mathcal{Q}^{h-i}_s(T)\int_{\partial (G_1 \cap \mathbb{C}_I)}\mathcal{Q}_p^{i+2}(T)dp_I p^{h+1}=0,$$
$$ \int_{\partial (G_2 \cap \mathbb{C}_I)}s^h ds_I\sum_{i=0, i \neq \frac{h-2}{2}}^{h-2} \mathcal{Q}^{h-i}_s(T)T^2\int_{\partial (G_1 \cap \mathbb{C}_I)}\mathcal{Q}_p^{i+2}(T)dp_I p^{h} =0.$$
}
By formula \eqref{Rp} we get
{\small
$$\frac{(2 \pi)^{2}}{\gamma_{n}^{-1}} \check{P}^2=\int_{\partial(G_2 \cap \mathbb{C}_I)} ds_I \int_{\partial(G_1 \cap \mathbb{C}_I)}s^h\bigl \{ \bigl[\mathcal{F}_n^R(s,T)-\mathcal{F}_n^L(p,T) \bigl]p-\bar{s}\bigl[\mathcal{F}_n^R(s,T)-\mathcal{F}_n^L(p,T) \bigl] \bigl \} (p^2-2s_0p+|s|^2)^{-1}  p^hdp_I.$$
}
Finally, by following exactly the same steps done when $h$ is odd we get
$$ \check{P}^2=\check{P}.$$
\end{proof}
\begin{remark}
Other properties of the Riesz projectors, are proved in \cite[Thm. 5.8]{CG}.
\end{remark}

\section{\bf Appendix}
\setcounter{equation}{0}
In this section we give the details of a result used in the paper.

\begin{lemma}
\label{app2}
Let $h= \frac{n-1}{2}$ be a fixed number. Then we have
{\small
\begin{equation}
\label{sum1}
\sum_{\ell=1}^{m-2h+1} \frac{(m- \ell-h+1)!( \ell+h-2)!}{(\ell-1)! (m- \ell-2h+1)!}= \frac{(h-1)! h! m!}{(2h)!(m-2h)!}, \qquad m \geq 2h
\end{equation}
}

\end{lemma}
\begin{proof}
Firstly we substitute $ h= \frac{n-1}{2}$, we get
{\small
$$ \sum_{\ell=1}^{m-2h+1} \frac{(m- \ell-h+1)!( \ell+h-2)!}{(\ell-1)! (m- \ell-2h+1)!}=\sum_{\ell=1}^{m-n+2} \frac{\left(m- \ell- \frac{n-3}{2}\right)!\left( \ell+\frac{n-5}{2}\right)!}{(\ell-1)! (m- \ell-n+2)!}.$$
}
Now, we observe that
{\small
\begin{equation}
\label{gamma}
\left(\frac{n-1}{2}\right)\left(\frac{n-1}{2}\right)!\left(\frac{n-3}{2}\right)!= \left[\Gamma\left(\frac{n+1}{2}\right)\right]^2.
\end{equation}
}
We denote by $(a)_n:=\frac{\Gamma(a+n)}{\Gamma(a)}$ the Pochhammer symbol. Therefore, we have
{\small
 \begingroup\allowdisplaybreaks\begin{align}
&\sum_{\ell=1}^{m-n+2} \frac{\left(m- \ell- \frac{n-3}{2}\right)!\left( \ell+\frac{n-5}{2}\right)!}{(\ell-1)! (m- \ell-n+2)!}= \frac{\left[\Gamma\left(\frac{n+1}{2}\right)\right]^2}{\left(\frac{n-1}{2}\right)} \, \, \,\sum_{\ell=1}^{m-n+2} \frac{\left(m- \ell- \frac{n-3}{2}\right)!\left( \ell+\frac{n-5}{2}\right)!}{(\ell-1)! \left(\frac{n-1}{2}\right)! \left(\frac{n-3}{2}\right)! (m- \ell-n+2)!}\\ \nonumber
&=\frac{\left[\Gamma\left(\frac{n+1}{2}\right)\right]^2}{\left(\frac{n-1}{2}\right)} \, \, \,\sum_{\ell=1}^{m-n+2} \frac{\Gamma\left(m- \ell - \frac{n}{2}+ \frac{5}{2}\right)\Gamma \left(\ell+ \frac{n}{2}- \frac{3}{2}\right)}{(\ell-1)! \Gamma\left(\frac{n+1}{2}\right) \Gamma\left(\frac{n-1}{2}\right)(m- \ell-n+2)!}\\ \nonumber
&=\frac{\left[\Gamma\left(\frac{n+1}{2}\right)\right]^2}{\left(\frac{n-1}{2}\right)} \, \, \,\sum_{\ell=1}^{m-n+2} \frac{\Gamma\left(m- \ell - n+2+ \frac{n+1}{2}\right)\Gamma \left(\ell -1+ \frac{n-1}{2}\right)}{(\ell-1)! \Gamma\left(\frac{n+1}{2}\right) \Gamma\left(\frac{n-1}{2}\right) (m- \ell-n+2)!}\\ \nonumber
&=\frac{\left[\Gamma\left(\frac{n+1}{2}\right)\right]^2}{\left(\frac{n-1}{2}\right)} \, \, \,\sum_{\ell=1}^{m-n+2}  \frac{1}{( \ell-1)!(m-n- \ell+2)!} \left(\frac{n+1}{2}\right)_{m-n-\ell+2} \left(\frac{n-1}{2}\right)_{\ell-1}\\ \nonumber
&= \frac{\left[\Gamma\left(\frac{n+1}{2}\right)\right]^2}{\left(\frac{n-1}{2}\right)} \, \, \,\sum_{\ell=1}^{m-n+2}  \frac{1}{(m-n+1)!} \binom{m-n+1}{\ell-1} \left(\frac{n+1}{2}\right)_{m-n-\ell+2} \left(\frac{n-1}{2}\right)_{\ell-1}\\ \nonumber
&= \frac{\left[\Gamma\left(\frac{n+1}{2}\right)\right]^2}{\left(\frac{n-1}{2}\right)} \, \, \,\sum_{\ell=1}^{m-n+2}  \frac{(n)_{m-n+1}}{(m-n+1)!} \binom{m-n+1}{\ell-1} \frac{\left(\frac{n+1}{2}\right)_{m-n-\ell+2} \left(\frac{n-1}{2}\right)_{\ell-1}}{(n)_{m-n+1}}.
\end{align}\endgroup
}
If we put
{\small
$$ T_{\ell-1}^{m-n+1}(n):=\binom{m-n+1}{\ell-1} \frac{\left(\frac{n+1}{2}\right)_{m-n-\ell+2} \left(\frac{n-1}{2}\right)_{\ell-1}}{(n)_{m-n+1}},$$
}
we get
{\small
 \begingroup\allowdisplaybreaks\begin{align}
\sum_{\ell=1}^{m-n+2} \frac{\left(m- \ell- \frac{n-3}{2}\right)!\left( \ell+\frac{n-5}{2}\right)!}{(\ell-1)! (m- \ell-n+2)!}&=
\frac{\left[\Gamma\left(\frac{n+1}{2}\right)\right]^2}{\left(\frac{n-1}{2}\right)} \, \, \,\sum_{\ell=1}^{m-n+2}  \frac{(n)_{m-n+1}}{(m-n+1)!}T_{\ell-1}^{m-n+1}(n)
\\ \nonumber
&= \frac{\left[\Gamma\left(\frac{n+1}{2}\right)\right]^2}{\left(\frac{n-1}{2}\right)}\frac{(n)_{m-n+1}}{(m-n+1)!} \sum_{\ell=0}^{m-n+1}  T_{\ell}^{m-n+1}(n).
\end{align}\endgroup
}
By \cite[Theorem 1]{CFM} we know that
{\small
$$ \sum_{\ell=0}^{m-n+1}  T_{\ell}^{m-n+1}(n)=1.$$
}
Therefore
{\small
$$ \sum_{\ell=1}^{m-n+2} \frac{\left(m- \ell- \frac{n-3}{2}\right)!\left( \ell+\frac{n-5}{2}\right)!}{(\ell-1)! (m- \ell-n+2)!}=\frac{\left[\Gamma\left(\frac{n+1}{2}\right)\right]^2}{\left(\frac{n-1}{2}\right)}\frac{(n)_{m-n+1}}{(m-n+1)!} .$$
}
By \eqref{gamma} we get
{\small
 \begingroup\allowdisplaybreaks\begin{align}
\sum_{\ell=1}^{m-n+2} \frac{\left(m- \ell- \frac{n-3}{2}\right)!\left( \ell+\frac{n-5}{2}\right)!}{(\ell-1)! (m- \ell-n+2)!}&= \left(\frac{n-1}{2} \right)!\left(\frac{n-3}{2} \right)! \frac{1}{(m-n+1)!} \frac{\Gamma(m+1)}{\Gamma(n)}\\ \nonumber
&= \left(\frac{n-1}{2} \right)!\left(\frac{n-1}{2}-1 \right)! \frac{m!}{(m-n+1)!(n-1)!} .
\end{align}\endgroup
}
Since $n=2h+1$ we get
{\small
$$ \sum_{\ell=1}^{m-2h+1} \frac{(m- \ell-h+1)!( \ell+h-2)!}{(\ell-1)! (m- \ell-2h+1)!}= \frac{(h-1)! h! m!}{(2h)!(m-2h)!}.$$
}
\end{proof}

\medskip
{\em Acknowledgements}.

The first author is partially supported by the PRIN project Direct and inverse problems for partial differential equations: theoretical aspects and applications.

\end{document}